\documentclass[a4paper,12pt]{article}
\usepackage{amsmath,amssymb,amsthm}
\numberwithin{equation}{section}
\usepackage[margin=20mm]{geometry}
\newtheorem{thm}{Theorem}[section]
\newtheorem{prop}[thm]{Proposition}
\newtheorem{cor}[thm]{Corollary}
\newtheorem{exam}[thm]{Example}
\newtheorem{rem}[thm]{Remark}
\newtheorem{lem}[thm]{Lemma}
\newtheorem{defn}[thm]{Definition}
\newtheorem{assum}[thm]{Assumption}

\def\Capa{\mathop{\rm Cap}}
\def\supp{{\rm supp}}
\def\B{{\cal B}}

\def\E{{\cal E}}
\def\F{{\cal F}}
\def\N{{\cal N}}

\def\d{{\rm d}}

\allowdisplaybreaks

\title{Hausdorff dimensions of
inverse images and collision
time sets for symmetric Markov processes}
\author{Yuichi Shiozawa\thanks{Department of Mathematics,
Graduate School of Science,
Osaka University, Toyonaka, Osaka, 560-0043,
Japan; \texttt{shiozawa@math.sci.osaka-u.ac.jp}}\qquad
Jian Wang\thanks{School of Mathematics and Statistics  \&
Fujian Key Laboratory of Mathematical Analysis and Applications (FJKLMAA)
\& Center for Applied Mathematics of Fujian Province (FJNU),
Fujian Normal University, Fuzhou, 350007, P.R. China; \texttt{jianwang@fjnu.edu.cn}}}

\begin{document}
\maketitle
\begin{abstract}
In this paper, we establish the Hausdorff dimensions
of inverse images and collision time sets for
a large class of symmetric Markov processes on metric measure spaces.
We apply the approach in the works by Hawkes and
Jain--Pruitt,
and make full use of heat kernel estimates.
In particular, the results efficiently apply to symmetric diffusion processes, symmetric stable-like processes, and symmetric
diffusion processes
with jumps in $d$-sets.
\end{abstract}



\section{Introduction}
Sample path properties of
Markov processes have been extensively studied in the literature, in particular, for L\'evy processes.
The readers are referred to the survey paper \cite{Xiao14} and the references therein for more details.
Nowadays,
there are a few developments in the extensions of part
of results from L\'evy processes to L\'evy-type processes.
For example, based on two-sided heat kernel estimates for a class of symmetric jump processes on metric measure spaces,
the laws of the iterated logarithm (LILs) for sample paths, local times and ranges are established in \cite{KKW17}.
We also obtained  in \cite{SW17} the integral tests on the escape rates, which are quantitative expressions of
recurrence, transience and conservativeness.
By applying the behavior
of the symbol of the corresponding generator,
Schilling established in \cite{Rene98} the results on the Hausdorff dimensions
of the image sets for L\'evy-type processes,
see also the monograph \cite[Chapter 5.2]{BSW14}.
Recently, a general method is provided in \cite{SXXZ} to establish uniform Hausdorff and packing dimension results
for the images of more general Markov processes, including stable jump
diffusion processes
and non-symmetric stable-type processes.

The main purpose of this paper is devoted to the Hausdorff dimensions of
level sets, inverse images and collision time sets
for a large class of symmetric Markov processes on metric measure spaces.
Here,
if we let
$X:=(X_t)_{t\ge0}$ be a symmetric Markov process on the metric measure space $(M,d,\mu)$,
then
the inverse image is referred to be
$$\{t>0: X_t\in D\}\quad \hbox{ for any Borel set } D\subset M.$$
In particular,
when $D=\{x\}$ with some $x\in M$,
this is reduced into the level set; on the other hand,
the collision time set
is defined by
$$\{t>0 : X_t^1=X_t^2\},$$
where $X^i:=(X^i_t)_{t\ge0}$, $i=1,2$, are two independent copies of $X$.
Concrete examples of the Markov processes included in the framework of the present
paper are symmetric diffusion processes,
symmetric stable-like processes,
symmetric diffusion processes
with jumps in $d$-sets, and so on.
Note that, as seen from the survey paper
\cite[Sections 6 and 7]{Xiao14},
dimension results and their proofs for level sets, inverse images and collision
time sets are more complex than those for images.

This work is inspired by the Hausdorff dimension results
of the inverse images and collision time sets for stable processes on Euclidean space.
Jain and Pruitt \cite{Jain} established the Hausdorff dimensions on
the collision sets of two independent stable processes $X^1:=(X_t^1)_{t\geq 0}$ and $X^2:=(X_t^2)_{t\geq 0}$
on ${\mathbb R}$
possibly with different indices.
Their idea is to  regard the collision of $X^1$ and $X^2$
as their  direct product process $X^1\otimes X^2$ hitting  the diagonal set  in ${\mathbb R}^2$,
and to compare the polarity of $X^1\otimes X^2$ with that of some stable process in ${\mathbb R}^2$.
Jain and Pruitt \cite[Introduction]{Jain} also pointed out that,
if $X^1$ and $X^2$ have the common
index  $\alpha\in (1,2)$,
then the collision time set of $X^1$ and $X^2$ has the same Hausdorff dimension
as that of the level set of the one-dimensional $\alpha$-stable process.
This property follows from the fact that
the difference process
$(X_t^1-X_t^2)_{t\geq 0}$ is also a $\alpha$-stable process.
However, if the indices of $X^1$ and $X^2$ are different, then it is unclear
how to establish the Hausdorff dimension of the collision time set.

Motivated by \cite{Jain}, Hawkes \cite{H71,H74} established the Hausdorff dimension
of the inverse image for one-dimensional $\alpha$-stable processes with $\alpha\in (1,2)$.
The idea of these works is to parametrize the stable indices by using the stable subordinators,
and to utilize the regularity and polarity properties of the stable processes.
Liu \cite{L95} applied this idea to the inverse images of compact sets
for L\'evy processes on Euclidean space.
Recently, Knopova and Schilling \cite{KS15} further
applied this idea to the inverse image
of Feller processes on Euclidean space, with application to the collision time sets of the two independent copies.

Our approach is based on heat kernel estimates for the associated Markov processes,
together with the development of the ideas of \cite{Jain,H71,H74, KS15} as mentioned before.
More precisely, we will make full use of the subordinate processes and the associated potential theory.
However, since the present paper is concerned with
symmetric Markov processes on
metric measure spaces,
there are a few difficulties and differences compared with
the papers cited above.
For instance,
\begin{enumerate}
\item Concerning the inverse image,
we follow the idea of Hawkes \cite{H71,H74} to make use of the stable subordinator.
However, since the subordinate process of a Markov process is not a stable process in general,
we can not utilize the polarity property of stable processes as \cite{H71,H74}.

\item
Inspired by \cite{Jain} and \cite{H71,H74},
we determine the Hausdorff dimension of the collision time set
by studying the regularity and polarity
of the stable subordinate process of the direct product process.
However, we need further consideration on the regularity property;
it should be noted that,
even for the direct product process of the two independent stable processes on ${\mathbb R}$,
its stable subordinate process
is not a stable process on ${\mathbb R}^2$ in general.
Moreover, since the state space is a metric measure space,
the approach with aid of the difference process is not applicable to the collision time set.
\end{enumerate}

Due to these difficulties and differences above, we need some new ideas and some efforts in the present paper.
To state the contribution of our paper, let us restrict on the following special setting.

\begin{thm}\label{Thm:main}
Let
$(M,d,\mu)$ be a connected $d$-set
such that
any closed ball in $M$ is compact.
For a subset $F$ of $M$, let ${\rm dim}_{\cal H}(F)$ denote its Hausdorff dimension.
Let $X:=(X_t)_{t\ge0}$ be
the $\mu$-symmetric diffusion process
with walk dimension $\alpha$ or the  symmetric $\alpha$-stable-like process
$($that is, the associated scaling function
of each process
 is  $\phi(r)=r^\alpha$$)$ on $M$.
Then the following statements hold.
\begin{itemize}
\item[{\rm(1)}] Suppose that $d\le \alpha$. Then, for any $a\in M$,
$${\rm dim}_{\cal H}\{t>0 : X_t=a\}
=1-\frac{d}{\alpha}, \quad \text{$P_x$-a.s.\ for any $x\in M$.}$$
More generally,
if $F\subset M$ is a Borel set such that
${\rm dim}_{\cal H}(F)>0$,
then
$$\dim_{{\cal H}}\{t>0 : X_t\in F\}
=1-\frac{d-{\rm dim}_{\cal H}(F)}{\alpha}, \quad \text{$P_x$-a.s.\ for any $x\in M$.}$$

\item[{\rm(2)}] Suppose that $d<\alpha$.
Let $F\subset M$ be an $s$-set with some $s>0$ $($in particular, ${\rm dim}_{\cal H}(F)=s>0)$. If
${\rm dim}_{\cal H}(F)>\max\{2d-\alpha,0\}$,
then
$$
{\rm dim}_{\cal H}\{t>0 : X_t^1=X_t^2\in F\}
=1-\frac{2d-{\rm dim}_{\cal H}(F)}{\alpha},
\quad \text{$P_x$-a.s.\ for any $x\in M\times M$,}
$$ where $X^i:=(X^i_t)_{t\ge0}$, $i=1,2$, are two independent copies of $X$. In particular,
$$
{\rm dim}_{\cal H}\{t>0 : X_t^1=X_t^2\}
=1-\frac{ d}{\alpha},
\quad \text{$P_x$-a.s.\ for any $x\in M\times M$.}
$$
\end{itemize}
\end{thm}

As mentioned above, the proof of Theorem \ref{Thm:main} is
partly based on heat kernel estimates for symmetric Markov processes,
which are now
developed greatly in recent years (see, e.g., \cite{AB,CKKW2022,CK,CKW20,CKW,CKW2022,GT12}).
Indeed, according to general results of our paper, we also can get by Remark \ref{rem2:heat-kernel}
below
and
\cite[Remark 1.12(iii) and Example 7.2]{CKW20}
that ---

\smallskip

{\it  Let $(M,d,\mu)$ be a connected $d$-set
such that
any closed ball in $M$ is compact.
Let $X:=(X_t)_{t\ge0}$ be
the $\mu$-symmetric diffusion process
with jumps on $M$, where the scaling functions of diffusion part and jump part are given
respectively by
$\phi_c(r)=r^\alpha$ and $\phi_j(r)=r^\beta$ for some $0<\beta<\alpha$.
Then, the conclusion {\rm(1)} of Theorem $\ref{Thm:main}$ holds
when $d\le \beta$, and the conclusion {\rm(2)} of Theorem $\ref{Thm:main}$ still holds when $d<\beta$.}

\smallskip

We make some
comments on how to overcome the difficulties mentioned in (i) and (ii).
For the inverse images,
we derive the polarity of the subordinate processes
by employing the Frostman lemma
on the complete separable metric space
in Subsection \ref{subsect:frostman}.
For the collision time sets,
we first prove the zero-one law for the tail events (Proposition \ref{prop:tail}),
and then  establish the Wiener tests for the recurrence and regularity of $X^1\otimes X^2$
(Propositions \ref{prop:rec-set} and \ref{prop:reg-set}).
Under our general setting, we can see from Examples \ref{exam:leve2} and \ref{exam3} that,
the local properties of the volume growth and walk dimensions
determine the Hausdorff dimensions of the inverse images and collision time sets.
With regard to  the collision time sets, our general results in Subsection \ref{subsect:collision}
allow two independent symmetric Markov processes to be different.
We also note that, as far as the authors know,
the Wiener tests are unavailable for general symmetric Markov processes,
even though those are well known for Brownian motion or other L\'evy processes on Euclidean spaces
(see, e.g., \cite{PS78}).

We mention that Shieh \cite{Shieh89,Shieh95} studied the possibility of collisions of two independent Hunt processes
in terms of the heat kernels,
with applications to  L\'evy processes on Euclidean space and Brownian motions on fractals.
Our results in the present paper provide quantitative information on the collision times,
and are applicable to symmetric jump processes of variable order
on $d$-sets, fractals and ultra-metric spaces.
We also characterize the Hausdorff dimension of the set of collision times on a given set by its Hausdorff dimension.

\bigskip

The rest of the paper is arranged as follows.
In the next section, we present preliminaries and assumptions used in the paper.
In Section \ref{section3}, we obtain Hausdorff dimensions of level sets and inverse images,
where we will first consider heat kernel and resolvent for the stable-subordinate process.
In Section \ref{section4}, we study Hausdorff dimensions of
the collision time sets.
For this, we establish estimates for the resolvent of stable-subordinate direct-product process.
In the appendix, we collect some statements
used in the proofs of our results, which include
the Wiener tests
for the recurrence and
regularity of symmetric Markov processes on
metric measure spaces.

We close this introduction with some words on notations.
For nonnegative functions $f$ and $g$ on a set $T$, we write $f(t)\preceq g(t)$ (resp.\ $f(t)\succeq g(t)$) for any $t\in T$
if there
exists a constant $c>0$ such that $f(t)\leq cg(t)$ (resp.\ $f(t)\geq cg(t)$) for any $t\in T$.
We write $f(t)\simeq g(t)$ for any  $t\in T$ if $f(t)\preceq g(t)$ and $f(t)\succeq g(t)$ for any $t\in T$.

\section{Preliminaries and assumptions}
\subsection{Dirichlet form, transience, recurrence and capacity}
\label{subsect:transience}
We first recall from  \cite{FOT11} the notions of Dirichlet forms and global properties of the associated
Markovian semigroups.
Let $(M,d)$ be a locally compact separable metric space,
and $\mu$ a positive Radon measure on $M$ with full support.
For $u,v\in L^2(M;\mu)$, let $(u,v)=\int_M uv\,\d\mu$ be the $L^2$-inner product.
Let $\{T_t\}_{t>0}$ be a strongly continuous Markovian semigroup on $L^2(M;\mu)$,
and $({\cal E},{\cal F})$ the associated Dirichlet form.
More precisely, $({\cal E},{\cal F})$ is a closed Markovian symmetric form
on $L^2(M;\mu)$ defined by
\begin{equation}\label{eq:form}
\E(u,u)=\lim_{t\rightarrow 0}\frac{1}{t}(u-T_tu,u), \quad
u\in \F:=\left\{u\in L^2(M;\mu): \lim_{t\rightarrow 0}\frac{1}{t}(u-T_tu,u)<\infty\right\}
\end{equation}
(see \cite[Lemma 1.3.4]{FOT11}).
For $\alpha>0$, let
$$\E_\alpha(u,v)=\E(u,v)+\alpha (u,v), \quad u,v\in \F.$$

For $t>0$ and $f\in L^2(M;\mu)$, we can define the integral
$$S_tf=\int_0^t T_sf \,\d s$$
as the $L^2$-strong convergence limit of the Riemann sum.
Then, by \cite[p.37]{FOT11}, $T_t$ and $S_t$ are bounded symmetric operators on $L^2(M;\mu)$.
They are further extended uniquely from $L^1(M;\mu)\cap L^2(M;\mu)$ to $L^1(M;\mu)$.
We can also extend $T_t$ from $L^2(M;\mu)\cap L^{\infty}(M;\mu)$ to $L^{\infty}(M;\mu)$
(see \cite[p.56]{FOT11}).
We will use the same notation $T_t$ for the corresponding operators on $L^1(M;\mu)$ and on $L^{\infty}(M;\mu)$.

Let $L_+^1(M;\mu)=\{u\in L^1(M;\mu): \text{$u\geq 1$, $\mu$-a.e.\ on $M$}\}$
and $Gf=\lim_{N\rightarrow\infty}S_Nf$ for $f\in L_+^1(M;\mu)$.
\begin{defn}
Let $\{T_t\}_{t>0}$ be a strongly continuous Markovian semigroup on $L^2(M;\mu)$,
and $({\cal E},{\cal F})$ the associated Dirichlet form.
\begin{enumerate}
\item $(\E,\F)$ is conservative if $T_t1=1$, $\mu$-a.e.\ for any $t>0$.
\item $(\E,\F)$ is transient if $Gf<\infty$, $\mu$-a.e.\ for any $f\in L_+^1(M;\mu)$,
and recurrent if $Gf=0$ or $\infty$, $\mu$-a.e.\ for any $f\in L_+^1(M;\mu)$.
\item A $\mu$-measurable set $A\subset M$ is invariant, if for any $f\in L^2(M;\mu)$ and $t>0$,
$T_t({\bf 1}_Af)={\bf 1}_AT_tf$, $\mu$-a.e.
If any invariant set $A\subset M$ satisfies $\mu(A)=0$ or $\mu(M\setminus A)=0$,
then  $(\E,\F)$ is called irreducible.
\end{enumerate}
\end{defn}

We know by \cite[Lemma 1.6.4 (iii)]{FOT11} that
any irreducible Dirichlet form is either transient or recurrent.

Let ${\cal F}_e$ denote the totality of $\mu$-measurable functions $u$ on $M$ such that
$|u|<\infty$ $\mu$-a.e.\ on $M$ and there exists a sequence $\{u_n\}\subset {\cal F}$ such that
$\lim_{n\rightarrow \infty}u_n=u$ $\mu$-a.e.\ on $M$ and
$\lim_{m,n\rightarrow\infty}{\cal E}(u_n-u_m,u_n-u_m)=0$.
The sequence $\{u_n\}$ is called an approximating sequence of $u$.
For any $u\in \F_e$ and its approximating sequence $\{u_n\}$,
the limit $\E(u,u)=\lim_{n\rightarrow\infty}{\cal E}(u_n,u_n)$ exists,
and does not depend on the choice of the approximating sequence for $u$
(\cite[Theorem 1.5.2]{FOT11}).
We call $(\F_e,\E)$ the extended Dirichlet space of $(\E,\F)$ (\cite[p.41]{FOT11}).
We also know by \cite[Lemma 1.5.5]{FOT11} that,
if $(\E,\F)$ is transient, then $\F_e$ is complete with respect to $\sqrt{\E}$.

We next recall from \cite{FOT11} the notion of the capacity relative to  $({\cal E}, {\cal F})$.
Let $C_0(M)$ denote the totality of continuous functions on $M$ with compact support.
In what follows, we suppose that $(\E,\F)$ is regular;
that is, ${\cal F}\cap C_0(M)$ is dense both in ${\cal F}$ with respect to $\sqrt{\E_1}$,
and in $C_0(M)$ with respect  to the uniform norm, where $\E_1(f,f)=\E(f,f)+\|f\|_2^2.$
Let ${\cal O}$ be the totality of open subsets of $M$.
For $O\in {\cal O}$, set
\begin{equation}\label{eq:cap-1}
{\cal L}_O=\left\{u\in {\cal F} : \text{$u\geq 1$, $\mu$-a.e.\ on $O$}\right\}.
\end{equation}
Define the ($1$-)capacity of $O\in {\cal O}$ by
\begin{equation}\label{eq:cap-2}
{\rm Cap}(O)=\begin{cases}
\inf_{u\in {\cal L}_O}{\cal E}_1(u,u), & {\cal L}_O\ne \emptyset,\\
\infty, & {\cal L}_O=\emptyset.
\end{cases}
\end{equation}
We then define the ($1$-)capacity of any subset $A$ of $M$ by
$${\rm Cap}(A)=\inf_{O\in {\cal O}, A\subset O}{\rm Cap}(O).$$

We say that a statement $S(x)$ depending on $x\in M$ holds quasi everywhere (q.e.\ in short)
if there exists a set $\N\subset M$ with $\Capa(\N)=0$ such that $S(x)$ holds for any $x\in M\setminus \N$.
For $f\in {\cal F}$, let $\tilde{f}$ be its quasi-continuous $\mu$-version;
that is, $f=\tilde{f}$, $\mu$-a.e.\ on $M$, and for any $\varepsilon>0$,
there exists a closed subset $F$ of $M$ such that
${\Capa}(M\setminus F)<\varepsilon$ and $\tilde{f}$ is finite continuous on $F$
(\cite[Section 2.1]{FOT11}).

Let $\nu$ be a positive Radon measure on $M$.
According to \cite[p.77, (2.2.1)]{FOT11}, we say that $\nu$ is of finite energy integral, if
there exists $C>0$ such that
$$
\int_M |v|\,{\rm d}\nu\leq C\sqrt{{\cal E}_1(v,v)}, \quad v\in {\cal F}\cap C_0(M).
$$
Let $S_0$ denote the totality of measures of finite energy integral on $M$.
Then, there exists a unique function $U_1\nu\in {\cal F}$ such that
$${\cal E}_1(U_1\nu,v)=\int_M \tilde{v}\,{\rm d}\nu, \quad v\in {\cal F}.$$
The function $U_1\nu$ is called the 1-potential of $\nu$.
We note that any measure in $S_0$ charges no set of zero capacity (\cite[Theorem 2.2.3]{FOT11}).
Moreover,
if $K$ is a compact subset of $M$,
then there exist a unique element $e_K\in {\cal F}$ and a unique measure $\nu_K\in S_0$
such that $e_K=U_1\nu_K$ and
${\rm Cap}(K)={\cal E}_1(e_K,e_K)=\nu_K(K)$
(see \cite[(2.2.13)]{FOT11}).
The element $e_K$ and the measure $\nu_K$ are called
the $1$-equilibrium potential and the $1$-equilibrium measure of $K$,
respectively.

Let
$$S_{00}=\left\{\nu\in S_0: \nu(M)<\infty, \ \|U_1\nu\|_{\infty}\leq 1\right\}.$$
We then see by \cite[p.82, Exercise 2.2.2]{FOT11} that,
if $K$ is a compact subset of $M$, then
\begin{equation}\label{eq:cap-sup}
{\Capa}(K)=\sup\left\{\nu(K): \nu\in S_{00}, \, \supp[\nu]\subset K, \,
\widetilde{U_1\nu}\leq 1, \text{q.e.}\right\}.
\end{equation}

If $(\E,\F)$ is transient,
then we can define the 0-order capacity ${\rm Cap}_{(0)}(A)$ of $A\subset M$
by replacing $\F$ and $\E_1$ with $\F_e$ and $\E$, respectively,
in \eqref{eq:cap-1} and \eqref{eq:cap-2} (\cite[p.74]{FOT11}).
As we see from \cite[p.85]{FOT11}, we can also introduce
the notions of a class of measures of finite ($0$-order) finite energy integral ($S_0^{(0)}$ in notation),
and of ($0$-order) potential of the measure $\nu\in S_0^{(0)}$ ($U\nu$ in notation).
In particular, if $K$ is a compact subset of $M$,
then we have the corresponding $0$-order equilibrium potential $e_K^{(0)}\in \F_e$
and the $0$-order equilibrium measure $\nu_K\in S_0^{(0)}$
such that  $e_K=U\nu_K$ and
${\rm Cap}_{(0)}(K)={\cal E}(e_K,e_K)=\nu_K(K)$.

\subsection{Hunt process and measurability}

In this subsection, we first recall from \cite{BG68}
classes of measurable subsets of $M$ associated with Hunt processes.
As in Subsection \ref{subsect:transience},
$(M,d)$ is a locally compact separable metric space,
and $\mu$ is a positive Radon measure on $M$ with full support.
Let $M_{\Delta}:=M\cup\{\Delta\}$ be the one point compactification of $M$.
Let  $X=(\Omega, \F, \{X_t\}_{t \ge 0},\{P_x\}_{x\in M}, \{\theta_t\}_{t\ge0},\zeta)$
be a Hunt process on $M$.
Here $\theta_t:\Omega \to \Omega$ is the shift operator of the paths defined
by $X_s\circ\theta_t=X_{s+t}$ for every $s>0$,
and $\zeta=\inf\{t>0 : X_t=\Delta\}$ is the lifetime.

A subset $A$ of $M$ is called nearly Borel measurable (relative to the process $X$),
if for any probability measure $\nu$ on $M$,
there exist Borel subsets $B_1$ and $B_2$ of $M$ such that
$B_1\subset A\subset B_2$ and
$$P_{\nu}(\text{$X_t\in B_2\setminus B_1$ for some $t\geq 0$})=0
$$
(\cite[Definition 10.2 in Chapter I]{BG68}).
Let ${\cal B}^n(M)$ denote the totality of nearly Borel measurable subsets of $M$.
For $A\in {\cal B}^n(M)$, let $\sigma_A$ be the  hitting time of $X$ to $A$;
that is, $\sigma_A=\inf\{t>0 : X_t\in A\}$.
We say that a point $x\in M$ is regular for $A$, if $P_x(\sigma_A=0)=1$.
Let $A^r$ denote the totality of regular points for $A$, i.e.,
$$A^r=\left\{x\in M : P_x(\sigma_A=0)=1\right\}.$$
Then, $A^r$ is nearly Borel measurable (\cite[Corollary 2.13 in Chapter II]{BG68}).
If $A$ is a subset of $M$, then $A^r$ is defined as the totality of points regular
for all nearly Borel subsets containing $A$.
We call $A^r$ the regular set for $A$ (relative to the process $X$).

If $\nu$ is a Borel measure on $M$,
then ${\cal B}^{\nu}(M)$ denotes the completion of ${\cal B}(M)$ relative to $\nu$.
Define the $\sigma$-field ${\cal B}^*(M)=\bigcap_{\nu}{\cal B}^{\nu}(M)$,
where the intersection is taken over all Borel probability measures on $M$.
We call ${\cal B}^*(M)$ the $\sigma$-algebra
of universally measurable subsets over $(M,{\cal B}(M))$.
Then, by definition, ${\cal B}(M)\subset {\cal B}^n(M)\subset {\cal B}^*(M)$ (\cite[p.60]{BG68}).

Recall that $\mu$ is a positive Radon measure on $M$ with full support.
Since the state space $M$ is locally compact and separable,
there exists a strictly positive Borel measurable function $g$ on $M$ such that
$\mu^g=g\cdot \mu$ is a Borel probability measure on $M$ and thus
${\cal B}^{\mu^g}(M)={\cal B}^{\mu}(M)$.
Using this relation, we can uniquely extend the measure $\mu$ to ${\cal B}^*(M)$.
We use the same notation $\mu$ for such an extension.

We next recall from \cite{FOT11} the relation between symmetric Hunt processes and Dirichlet forms.
Let $\{p_t\}_{t>0}$ be the transition function of a Hunt process $X$ on $M$ defined by
$$\int_Mp_t(x,\d y)f(y)=E_x\left[f(X_t)\right],
\quad t>0, \ x\in M, \ f\in {\cal B}(M), \ f\geq 0,$$
with the convention that $f(\Delta)=0$.
The left hand side above is written as $p_t f$.
We now assume that the process $X$ is  $\mu$-symmetric, i.e.,
$(p_tu,v)=(u,p_t v)$ for any $t>0$ and nonnegative functions $u,v\in {\cal B}(M)$.
According to \cite[p.30 and p.160]{FOT11}, we can extend $\{p_t\}_{t>0}$ uniquely
to a strongly continuous Markovian semigroup $\{T_t\}_{t>0}$ on $L^2(M;\mu)$.
Then, by \eqref{eq:form}, we can associate a Dirichlet form $(\E,\F)$ on $L^2(M;\mu)$.

Conversely, if $(\E,\F)$ is a regular Dirichlet form on $L^2(M;\mu)$
associated with a strongly continuous Markovian semigroup $\{T_t\}_{t>0}$ on $L^2(M;\mu)$,
then there exists a $\mu$-symmetric Hunt process $X$ on $M$ such that
$$T_tf=p_tf, \ \text{$\mu$-a.e.\ for $t>0$ and $f\in L^2(M;\mu)\cap \B_b(M)$}$$
(\cite[Theorem 7.2.1]{FOT11}).

Let $X$ be a $\mu$-symmetric Hunt process on $M$ generated by a regular Dirichlet form $(\E,\F)$.
A set $\N\subset M$ is called exceptional,
if there exists a nearly Borel set $\tilde{\N}\supset \N$ such that
$P_x(\sigma_{\tilde{\N}}<\infty)=0
$ for $\mu$-a.e.\ $x\in M$.
A set $\N\subset M$ is called properly exceptional,
if it is nearly Borel measurable such that
$\mu(\N)=0$ and $M\setminus \N$ is $X$-invariant; that is,
$$P_x(\text{$X_t\in (M\setminus \N)_{\Delta}$ or $X_{t-}\in (M\setminus \N)_{\Delta}$ for any $t>0$})=1,
\ x\in M\setminus \N.$$
Here $(M\setminus \N)_{\Delta}=(M\setminus \N)\cup\{\Delta\}$ and $X_{t-}=\lim_{s\uparrow t}X_s$.
By definition, any properly exceptional set is exceptional.
In particular, if $(\E,\F)$ is regular,
then any compact subset of $M$ is of finite capacity so that a set $\N\subset M$ is exceptional
if and only if ${\rm Cap}(\N)=0$ (\cite[Theorem 4.2.1]{FOT11}).

\subsection{Heat kernel}
Let  $X=(\Omega, \F, \{X_t\}_{t \ge 0},\{P_x\}_{x\in M}, \{\theta_t\}_{t>0},\zeta)$
be a $\mu$-symmetric Hunt process on $M$
associated with the regular Dirichlet form $(\E,\F)$ on $L^2(M;\mu)$.
In what follows, we always impose the following Assumption {\bf (H)} on the process $X$.

\begin{assum}[Assumption {\bf (H)}]
\begin{enumerate}\it
\item[{\rm (i)}] $(\E,\F)$ is conservative and irreducible.
\item[{\rm (ii)}]
There exist a properly exceptional Borel set $\N\subset M$
and a Borel measurable function $p(t,x,y):(0,\infty)\times M\times M \to [0,\infty)$ such that
the next three conditions hold.
\begin{itemize}
\item
For any $t>0$, $x\in M\setminus \N$ and $A\in {\cal B}(M)$,
\begin{equation}\label{eq:a-trans}
P_x(X_t\in A)=\int_A p(t,x,y)\,\mu({\rm d}y).
\end{equation}
\item
For any $t>0$ and $x,y\in M\setminus \N$, $p(t,x,y)=p(t,y,x)$.
\item
For any $s,t>0$ and $x,y\in M\setminus \N$,
\begin{equation}\label{eq:c-k}
p(t+s,x,y)=\int_{M}p(t,x,z)p(s,z,y)\,\mu(\d z).
\end{equation}
\end{itemize}\end{enumerate}
\end{assum}

The function $p(t,x,y)$ in Assumption {\bf(H)} is called the heat kernel in the literature.
While \eqref{eq:a-trans} determines $p(t,x,y)$ for $\mu$-a.e.\ $y\in M$,
we can regularize $p(t,x,y)$ under the so-called ultracontractivity condition
so that the condition (ii) in Assumption {\bf (H)} is fulfilled
(see, e.g., \cite[Theorem 3.1]{BBCK09} and \cite[Subsection 2.2]{GT12} for details).

Under Assumption {\bf (H)}, we write $M_0=M\setminus \N$.

\begin{rem}\label{rem:global}\rm
Let Assumption {\bf (H)} hold.
\begin{enumerate}
\item \eqref{eq:a-trans} is true also for any $A\in {\cal B}^*(M)$.
\item We can characterize the global properties of $(\E,\F)$ in terms of the heat kernel
as follows (see \cite[Remark 2.2]{S16} and \cite[Remark 2.2]{SW17}):
\begin{itemize}
\item $(\E,\F)$ is transient if
\begin{equation}\label{eq:transient}
\int_1^{\infty}\left(\sup_{y\in M_0}p(t,x,y)\right)\,\d t<\infty, \quad x\in M_0,
\end{equation}
and recurrent if
$$
\int_1^{\infty}p(t,x,y)\,\d t=\infty, \quad x,y\in M_0.
$$
\item $(\E,\F)$ is irreducible if  $p(t,x,y)>0$ for any $t>0$ and $x,y\in M_0$.
\end{itemize}
We note that \cite[Remark 2.2]{S16} refers to
the condition \eqref{eq:transient} with $x\in M_0$ and
$\sup_{y\in M_0}p(t,x,y)$ replaced by
$x\in M$ and $\sup_{y\in M}p(t,x,y)$, respectively;
however, the argument there shows that the condition \eqref{eq:transient} suffices for transience.
\item By \eqref{eq:c-k} and the  Cauchy-Schwarz inequality,
we have
$p(t,
x,y)\leq \sqrt{p(t,x,x)p(t,y,y)}$ for any $t>0$ and $x,y\in M_0$.
Therefore, \eqref{eq:transient} holds if
\begin{equation}
\int_1^{\infty}\left(\sup_{y\in M_0}p(t,y,y)\right)\,\d t<\infty.
\end{equation}
\end{enumerate}
\end{rem}

Below, for $\lambda\geq 0$ and $A\in {\cal B}^*(M)$, define
$$U_{\lambda}(x,A):
=U_{\lambda}{\bf 1}_A(x)=\int_0^{\infty}e^{-\lambda t}P_x(X_t\in A)\,{\rm d}t,\quad
\ x\in M.$$
Similarly, for any nonnegative universally measurable function $f$ on $M$,
define
$$U_{\lambda}f(x)=E_x\left[\int_0^{\infty}e^{-\lambda t}f(X_t)\,{\rm d}t\right],\quad  x\in M.$$
Then, under Assumption {\bf(H)},  for any $x\in M_0$ and $A\in {\cal B}^*(M)$,
$$U_{\lambda}(x,A)=\int_Au_{\lambda}(x,y)\,\mu({\rm d}y),$$
where
$$u_{\lambda}(x,y)=\int_0^{\infty}e^{-\lambda t}p(t,x,y)\,{\rm d}t, \quad  x,y\in M_0.$$

To establish our results, we need to introduce various kinds of the heat kernel bounds.
For $x\in M$ and $r>0$, let $B(x,r)=\{y\in M: d(x,y)<r\}$ and $V(x,r)=\mu(B(x,r))$.
We always assume that there exist positive constants $c_1$, $c_2$, $d_1$, $d_2$ with $d_1\leq d_2$ so that
\begin{equation}\label{eq:vol-v}
c_1\left(\frac{R}{r}\right)^{d_1}\leq \frac{V(x,R)}{V(x,r)}\leq c_2\left(\frac{R}{r}\right)^{d_2},\quad x\in M, \ 0<r<R<\infty.
\end{equation}

\begin{defn}\label{def:heat-kernel-2}\it
\begin{itemize}
\item[{\rm(1)}]
The heat kernel $p(t,x,y)$ satisfies
the two-sided on-diagonal estimates {\rm(ODHK)}, if
\begin{equation}\label{e:on-d}
p(t,x,x)\simeq \frac{1}{V(x,\phi^{-1}(t))},\quad t>0, \ x\in M_0.
\end{equation}

\item[{\rm(2)}]
The heat kernel  $p(t,x,y)$ satisfies
the near-diagonal lower bounded estimates  {\rm(NDLHK)}, if there exists a constant $c_0>0$ so that
\begin{equation}\label{e:on-d-1}
p(t,x,y)\succeq \frac{1}{V(x,\phi^{-1}(t))},\quad t>0, \ x,y\in M_0  \hbox{ with } d(x,y)\le c_0\phi^{-1}(t).
\end{equation}

\item[{\rm(3)}]
The heat kernel  $p(t,x,y)$ satisfies the $($weak$)$ upper bounded estimates {\rm(WUHK)}, if
\begin{equation}\label{e:on-d-2}
p(t,x,y)\preceq\frac{1}{V(x,\phi^{-1}(t))}\wedge\frac{t}{V(x,d(x,y))\phi(d(x,y))},\quad t>0, \ x,y\in M_0.
\end{equation}
\end{itemize}
Here, $\phi:[0,\infty)\to[0,\infty)$ is a strictly increasing function  satisfying that $\phi(0)=0$, $\phi(1)=1$, and
that there exist positive constants $c_3$, $c_4$, $\alpha_1$, $\alpha_2$ with $\alpha_1\leq \alpha_2$ so that
\begin{equation}\label{eq:sca-phi}
c_3\left(\frac{R}{r}\right)^{\alpha_1}\leq \frac{\phi(R)}{\phi(r)}\leq c_4\left(\frac{R}{r}\right)^{\alpha_2},\quad 0<r<R<\infty.
\end{equation}
\end{defn}

Note that \eqref{eq:sca-phi} yields
\begin{equation}\label{eq:sca-phi-inv}
\frac{1}{c_4^{1/\alpha_2}}\left(\frac{R}{r}\right)^{1/\alpha_2}
\leq \frac{\phi^{-1}(R)}{\phi^{-1}(r)}
\leq \frac{1}{c_3^{1/\alpha_1}}\left(\frac{R}{r}\right)^{1/\alpha_1},\quad 0<r<R<\infty.
\end{equation}
Combining this with  \eqref{eq:vol-v}, we have
\begin{equation}\label{eq:v-inverse}
\frac{c_1}{c_4^{d_1/\alpha_2}}\left(\frac{T}{t}\right)^{d_1/\alpha_2}
\leq \frac{V(x,\phi^{-1}(T))}{V(x,\phi^{-1}(t))}
\leq \frac{c_2}{c_3^{d_2/\alpha_1}}\left(\frac{T}{t}\right)^{d_2/\alpha_1},\quad x\in M, \ 0<t<T<\infty.
\end{equation}

We also introduce
the H\"{o}lder regularity condition
for the heat kernel $p(t,x,y)$.

\begin{defn}\label{assum:heat-kernel-2}\it
The heat kernel  $p(t,x,y)$ satisfies the H\"{o}lder regularity condition {\rm(HR)}, if there exist constants
$\theta\in (0,1]$ and $C>0$ such that for any $t>0$ and
$x,y,z\in M$,
$$
|p(t,x,y)-p(t,x,z)|\leq \frac{C}{V(x,\phi^{-1}(t))}\left(\frac{d(y,z)}{\phi^{-1}(t)}\right)^{\theta}.
$$
\end{defn}

\begin{rem}\label{rem:heat-kernel} \rm \begin{itemize}
\item[(i)] According to \cite[Proposition 3.1(2)]{CKW},
if the regular Dirichlet form $({\cal E},{\cal F})$ admits no killing term
and the associated heat kernel $p(t,x,y)$ satisfies {\rm(NDLHK)}, then
$(\E,\F)$ is conservative.

\item[(ii)] Suppose that  the heat kernel $p(t,x,y)$ satisfies
{\rm (WUHK)} and
{\rm (HR)}.
If $u$ is a bounded continuous function on $M$, then so is $p_tu$ for any $t>0$.
In particular, there exists a version of the process $X$ such that
all the conditions in Assumption {\bf (H)} (ii) are valid by replacing $M\setminus \N$ with $M$.
If {\rm (WUHK)} and {\rm (HR)} are imposed on the heat kernel,
then we take the process $X$ as the version above.

\item[(iii)] We see by the proof of \cite[Proposition 5.4]{CKW} that,
if the heat kernel $p(t,x,y)$ satisfies {\rm(WUHK)} and {\rm(HR)},
then it satisfies {\rm(NDLHK)} as well.

\end{itemize}
\end{rem}

\begin{rem}\label{rem2:heat-kernel}\rm
The form in the right hand side of \eqref{e:on-d-2} for the definition (WUHK)
comes from two-sided heat kernel estimates for
the
mixture of symmetric stable-like (jump) processes in metric measure spaces; see \cite{CK,CKW}.
We should emphasize that this
kind of heat kernel upper bounds are satisfied for a large class of
symmetric
Markov processes, including
symmetric
diffusion processes generated by strongly local Dirichlet forms (see \cite{AB,GT12}),
symmetric diffusion processes
with jumps in metric measure spaces (see \cite{CKW20}), and symmetric jump processes
that allowed to have light tails of polynomial decay at infinity or
to have (sub- or super-) exponential decay jumps (see \cite{CKW2022, CKKW2022}).

To verify the assertion above, below we take the $\mu$-symmetric
diffusion process $X$ on an Ahlfors $d$-regular set $(M,d,\mu)$ with walk dimension $\alpha\ge 2$ for example.
Similar arguments work for all the processes mentioned above.
In this example, $V(x,r)\simeq r^d$, and the heat kernel $p(t,x,y)$ of the process $X$ enjoys
the following two-sided estimates:
$$p(t,x,y)\asymp t^{-d/\alpha} \exp\left(-\left(\frac{d(x,y)^\alpha}{t}\right)^{1/(\alpha-1)}\right).$$
Here, we write $f(s,x)\asymp g(s, x)$, if there exist constants
$c_k>0$, $k=1,2,3,4$, such that $c_1g(c_2s,x)\le f(s,x)\le c_3g(c_4s,x)$ for the specified range of $(s, x)$.
Then, by some calculations, one can see that
there are constants $c_5>0$ such that for all $x,y\in M$ and $t>0$,
$$\exp\left(-\left(\frac{d(x,y)^\alpha}{t}\right)^{1/(\alpha-1)}\right)
\le c_5\left(1+\frac{d(x,y)^\alpha}{t}\right)^{-(1+d/\alpha)}.$$
This implies that for all $x,y\in M$ and $t>0$,
$$p(t,x,y)\le c_6\left(t^{-d/\alpha}\wedge\frac{t}{d(x,y)^{d+\alpha}}\right).$$
In particular, (WUHK) holds with $\phi(r)=r^\alpha$.

Furthermore, according to results in all the cited papers,
we know that, for all the processes mentioned above, {\rm(ODHK)}, {\rm(NDLHK)}, {\rm(WUHK)} and {\rm(HR)} are satisfied.
\end{rem}

\section{Hausdorff dimensions
of level sets and inverse images}\label{section3}

\subsection{Heat kernel  and resolvent for stable-subordinate processes}

For $\gamma\in (0,1)$, let $S^{\gamma}:=(\{\tau_t\}_{t\geq 0},P^{\gamma})$ be the $\gamma$-stable subordinator
which is independent of the process $X$.
Let $\pi_t(s)$ denote the density function of $\tau_t$.
According to \cite[Theorem 4.4]{CKKW} (or the proof  of  \cite[Theorem 3.1]{BSS03}),
there exist positive constants $c_1$ and $c_2$ such that
\begin{equation}\label{eq:sub-upper}
\pi_t(s)
\leq\frac{c_1t}{s^{1+\gamma}}e^{-t/s^{\gamma}}
, \quad s,t>0
\end{equation}
and
\begin{equation}\label{eq:sub-lower}
\pi_t(s)\geq\frac{c_2t}{s^{1+\gamma}}, \quad s,t>0 \hbox{ with } s \geq t^{1/\gamma}.
\end{equation}

Let $X_t^{\gamma}=X_{\tau_t}$ for any $t\geq 0$,
and let $X^{\gamma}:=(X_t^{\gamma})_{t\geq 0}$ be the $\gamma$-stable subordinate process of $X$.
Then, the process $X^{\gamma}$ is a $\mu$-symmetric Hunt process.
Let $(\E^{\gamma},\F^{\gamma})$ be a Dirichlet form on $L^2(M;\mu)$ associated with $X^{\gamma}$.
Then, by \cite[Theorem 2.1 (ii) and Theorem 3.1 (i)-(ii)]{O02},
$(\E^{\gamma},\F^{\gamma})$ is also regular, irreducible and conservative.
We note that $M_0=M\setminus \N$ is $X^{\gamma}$-invariant by definition,
and $\N$ is of zero capacity relative to $(\E^{\gamma},\F^{\gamma})$ by \cite[Theorem 2.2 (i)]{O02};
hence $\N$ is also properly exceptional with respect to $X^{\gamma}$.
Moreover, the subordinate process $X^\gamma$
 possesses
the density function $q(t,x,y)$ with respect to the measure $\mu$ so that
$$q(t,x,y)=\int_0^{\infty}p(s,x,y)\,\pi_t(s)\,{\rm d}s,\quad t>0, \ x,y\in M_0.$$
Therefore, the process $X^{\gamma}$ satisfies Assumption {\bf (H)} as well.

For any $\lambda\ge0$, define
$$u^\gamma_\lambda (x,y)=\int_0^\infty e ^{-\lambda t}q(t,x,y)\,\d t,\quad x,y\in M_0.$$
Set $\phi^{\gamma}(t)=(\phi(t))^{\gamma}$, so that $(\phi^{\gamma})^{-1}(t)=\phi^{-1}(t^{1/\gamma})$.

\begin{lem}\label{lem:heat-sub}
Suppose that the process $X$ satisfies Assumption {\bf(H)}. Let $\gamma\in (0,1]$.
Then we have the following statements. \begin{itemize}
\item[{\rm(1)}]
Under {\rm (ODHK)},
$$q(t,x,x) \simeq
\frac{1}{V(x,(\phi^{\gamma})^{-1}(t))},\quad t>0, \ x\in M_0.$$
\item[{\rm(2)}]
Under {\rm (NDLHK)},
$$
q(t,x,y)\succeq
\frac{1}{V(x,(\phi^{\gamma})^{-1}(t))}\wedge\frac{t}{V(x,d(x,y))\phi^{\gamma}(d(x,y))},
\quad t>0, \ x,y\in M_0.
$$
Moreover,
\begin{equation}\label{eq:res-lower-1}
u^{\gamma}_1(x,y)
\succeq \int_{\phi^{\gamma}(d(x,y))}^{\infty} \frac{e^{-t}}{V(x,(\phi^{\gamma})^{-1}(t))}\,{\rm d}t,
\quad x,y\in M_0 \hbox{ with } d(x,y)\leq 1
\end{equation}
and
\begin{equation}\label{eq:res-lower-2}
u^{\gamma}_1(x,y)
\succeq \frac{1}{V(x,d(x,y))\phi^\gamma(d(x,y))}, \quad x,y\in M_0 \hbox{ with } d(x,y)\geq 1.
\end{equation}
\item[{\rm(3)}]
Under {\rm (WUHK)},
\begin{equation}\label{eq:wuhk-sub}
q(t,x,y)\preceq
\frac{1}{V(x,(\phi^{\gamma})^{-1}(t))}\wedge\frac{t}{V(x,d(x,y))\phi^{\gamma}(d(x,y))},
\quad t>0, \ x,y\in M_0.
 \end{equation}
Moreover,
$$
u^{\gamma}_1(x,y)
\preceq \int_{\phi^{\gamma}(d(x,y))}^{\infty} \frac{e^{-t}}{V(x,(\phi^{\gamma})^{-1}(t))}\,{\rm d}t
\quad x,y\in M_0 \hbox{ with } d(x,y)\leq 1
$$
and
$$
u^{\gamma}_1(x,y)
\preceq \frac{1}{V(x,d(x,y))\phi^\gamma(d(x,y))} \quad x,y\in M_0,  \hbox{ with } d(x,y)\geq 1.
$$
\end{itemize}
\end{lem}

\begin{rem}\rm
According to Lemma \ref{lem:heat-sub} above,
if the original process $X$ fulfills one of the conditions in Definition \ref{def:heat-kernel-2},
then the subordinate process $X^{\gamma}$ also satisfies the corresponding one,
with $\phi$ in \eqref{e:on-d} replaced by  $\phi^{\gamma}$.
\end{rem}

\begin{proof}[Proof of Lemma {\rm \ref{lem:heat-sub}}]
(1) Suppose that \eqref{e:on-d} holds.
Then, by \eqref{eq:sub-upper} and the change of variables formula with $u=t/s^{\gamma}$,
\begin{equation}\label{e:upper-1}
q(t,x,x)
\leq c_1\int_0^{\infty}\frac{1}{V(x,\phi^{-1}(s))}
e^{-t/s^{\gamma}}\frac{t}{s^{1+\gamma}}\,{\rm d}s
=\frac{c_1}{\gamma}
\int_0^{\infty}\frac{e^{-u}}{V(x,\phi^{-1}((t/u)^{1/\gamma}))}\,{\rm d}u.
\end{equation}
By \eqref{eq:v-inverse}, there exist positive constants $c_2$ and $\eta_1$ such that
$$V(x,\phi^{-1}((t/u)^{1/\gamma}))\geq c_2V(x,\phi^{-1}(t^{1/\gamma}))/u^{\eta_1}
=c_2V(x,(\phi^{\gamma})^{-1}(t))/u^{\eta_1}, \quad 0<u\leq 1.$$
Similarly,  there exist positive constants $c_3$ and $\eta_2$ such that
$$V(x,\phi^{-1}((t/u)^{1/\gamma}))\geq c_3 V(x,(\phi^{\gamma})^{-1}(t))/u^{\eta_2}, \quad u\geq 1.$$
Accordingly,
$$\int_0^{\infty}\frac{e^{-u}}{V(x,\phi^{-1}((t/u)^{1/\gamma}))}\,{\rm d}u
\leq \frac{c_4}{V(x,(\phi^{\gamma})^{-1}(t))}
\left(\int_0^1e^{-u}u^{\eta_1}\,{\rm d}u+\int_1^{\infty}e^{-u}u^{\eta_2}\,{\rm d}u\right).$$
Combining this with
\eqref{e:upper-1}, we get the desired upper bound of $q(t,x,x)$.

On the other hand, it follows by \eqref{eq:sub-lower} that
\begin{equation}\label{e:lower-1}
q(t,x,x)
\geq \int_{t^{1/{\gamma}}}^{\infty}p(s,x,x)\pi_t(s)\,{\rm d}s
\geq
c_5\int_{t^{1/{\gamma}}}^{\infty}\frac{1}{V(x,\phi^{-1}(s))}\frac{t}{s^{1+\gamma}}\,{\rm d}s.
\end{equation}
Fix a constant $\theta>1$ and let $\theta_n=t^{1/{\gamma}}\theta^n$.
Then, by \eqref{eq:v-inverse} again,
there exist positive constants $c_6$ and $\eta_3$ such that
\begin{equation*}
\begin{split}
\int_{\theta_n}^{\theta_{n+1}}\frac{1}{V(x,\phi^{-1}(s))s^{1+\gamma}}\,{\rm d}s
&\geq \frac{1}{\gamma V(x,\phi^{-1}(\theta_{n+1}))}\left(\theta_n^{-\gamma}-\theta_{n+1}^{-\gamma}\right)\\
&\geq \frac{c_6}{tV(x,(\phi^{\gamma})^{-1}(t))}\left(1-\theta^{-\gamma}\right)\theta^{-n(\gamma+\eta_3)}.
\end{split}
\end{equation*}
Hence
\begin{equation}\label{eq:int-sum-1}
\begin{split}
 \int_{t^{1/{\gamma}}}^{\infty}\frac{1}{V(x,\phi^{-1}(s))s^{1+\gamma}}\,{\rm d}s
&=\sum_{n=0}^{\infty}\int_{\theta_n}^{\theta_{n+1}}\frac{1}{V(x,\phi^{-1}(s))s^{1+\gamma}}\,{\rm d}s\\
&\geq \frac{c_6}{tV(x,(\phi^{\gamma})^{-1}(t))}(1-\theta^{-\gamma})\sum_{n=0}^{\infty}\theta^{-n(\gamma+\eta_3)}\\
&=\frac{c_6}{t V(x,(\phi^{\gamma})^{-1}(t))}\frac{1-\theta^{-\gamma}}{1-\theta^{-(\gamma+\eta_3)}}.
\end{split}
\end{equation}
Then, by \eqref{e:lower-1}, we obtain
$$
q(t,x,x)
\geq \frac{c_7}{V(x,(\phi^{\gamma})^{-1}(t))}.
$$
We thus arrive at  the desired lower bound of $q(t,x,x)$.

(2) Suppose that \eqref{e:on-d-1} holds,
and let $c_0$ be the constant in \eqref{e:on-d-1}.
Without
loss of generality, we may and do assume that $c_0=1$.
Since the heat kernel $p(t,x,y)$ satisfies {\rm (NLDHK)},
it follows by \eqref{eq:sub-lower} that for any $x,y\in M_0$,
\begin{align*}
q(t,x,y)
=\int_0^{\infty}p(s,x,y)\pi_t(s)\,\d s
&\geq c_1\int_{t^{1/\gamma}\vee \phi(d(x,y))}^{\infty}
\frac{1}{V(x,\phi^{-1}(s))}\frac{t}{s^{1+\gamma}}\,\d s\\
&=c_1t\int_{t^{1/\gamma}\vee \phi(d(x,y))}^{\infty}\frac{1}{V(x,\phi^{-1}(s))s^{1+\gamma}}\,\d s.
\end{align*}
In particular, if $d(x,y)\le (\phi^{\gamma})^{-1}(t)$, then, by \eqref{eq:int-sum-1},
$$
t\int_{t^{1/\gamma}\vee \phi(d(x,y))}^{\infty}\frac{1}{V(x,\phi^{-1}(s))s^{1+\gamma}}\,\d s
=t\int_{t^{1/\gamma}}^{\infty}\frac{1}{V(x,\phi^{-1}(s))s^{1+\gamma}}\,\d s
\geq \frac{c_2}{V(x,(\phi^{\gamma})^{-1}(t))}.
$$
We also see that, if  $d(x,y)\ge (\phi^{\gamma})^{-1}(t)$, then
$$
t\int_{t^{1/\gamma}\vee \phi(d(x,y))}^{\infty}\frac{1}{V(x,\phi^{-1}(s))s^{1+\gamma}}\,\d s
=t\int_{\phi(d(x,y))}^{\infty}\frac{1}{V(x,\phi^{-1}(s))s^{1+\gamma}}\,\d s
\geq \frac{c_3 t}{V(x,d(x,y))\phi^{\gamma}(d(x,y))}.
$$
Therefore, we arrive at the desired lower bound of $q(t,x,y)$.

Using the lower bound of $q(t,x,y)$ above,  we obtain
\begin{align*}
u_1^\lambda(x,y)
=&\int_0^\infty e^{-t} q(t,x,y)\,\d t
=\int_0^{\phi^{\gamma}(d(x,y))} e^{-t} q(t,x,y)\,\d t
+\int_{\phi^{\gamma}(d(x,y))}^{\infty} e^{-t} q(t,x,y)\,\d t\\
\ge &\frac{c_4}{V(x,d(x,y))\phi^{\gamma}(d(x,y))}\left(\int_0^{\phi^\gamma(d(x,y))}e^{-t}t\, \d t\right)
+c_4\int_{\phi^\gamma(d(x,y))}^{\infty}\frac{e^{-t}}{V(x,(\phi^{\gamma})^{-1}(t))}\, \d t.
\end{align*}
Since
\begin{equation}\label{eq:scale-comp}
\int_0^{\phi^\gamma(d(x,y))}e^{-t}t\, \d t\asymp \phi^{2\gamma}(d(x,y))\wedge 1,
\end{equation}
we have \eqref{eq:res-lower-1} and \eqref{eq:res-lower-2}.

(3) Suppose that \eqref{e:on-d-2} holds.
We first show the upper bound of $q(t,x,y)$.
By definition,
$$q(t,x,y)=\int_0^{\phi(d(x,y))}p(s,x,y)\pi_t(s)\,{\rm d}s
+\int_{\phi(d(x,y))}^{\infty}p(s,x,y)\pi_t(s)\,{\rm d}s=I_1+I_2.$$
Then, by {\rm (WUHK)} and  \eqref{eq:sub-upper},
$$I_1\leq c_1\int_0^{\phi(d(x,y))}\frac{s}{V(x,d(x,y))\phi(d(x,y))}\frac{t}{s^{1+\gamma}}\,{\rm d}s
= \frac{c_2t}{V(x,d(x,y))\phi^{\gamma}(d(x,y))}$$
and
$$I_2\leq c_3\int_{\phi(d(x,y))}^{\infty}\frac{1}{V(x,\phi^{-1}(s))}\frac{t}{s^{1+\gamma}}\,{\rm d}s
\leq \frac{c_4 t}{V(x,d(x,y))\phi^{\gamma}(d(x,y))}.$$
The last inequality above follows by the same calculation as \eqref{eq:int-sum-1}.
Hence
$$q(t,x,y)\leq \frac{c_5t}{V(x,d(x,y))\phi^{\gamma}(d(x,y))}.$$
Following the calculation in the proof of (1),
we also have
$$q(t,x,y)\leq \frac{c_6}{V(x,(\phi^{\gamma})^{-1}(t))}$$
so  that \eqref{eq:wuhk-sub} follows.
The upper bounds of $u_1^{\gamma}(x,y)$ follow by the same calculations as in (2).
\end{proof}

Suppose that the process $X$ satisfies one of the conditions in Definition \ref{def:heat-kernel-2}.
For $\gamma\in (0,1]$, let
$$
I^{\gamma}(x)=\int_1^{\infty}\frac{1}{V(x,(\phi^{\gamma})^{-1}(t))}\,\d t, \quad x\in M.
$$
Then, by Remark \ref{rem:global}(ii) and Lemma \ref{lem:heat-sub},
the process $X^{\gamma}$ is recurrent
if the process $X$ satisfies  {\rm (NDLHK)} and $I^{\gamma}(x)=\infty$ for any $x\in M$;
$X^{\gamma}$ is transient if $X$ satisfies {\rm (WUHK)} and $I^{\gamma}(x)<\infty$ for any $x\in M$.
The next lemma provides
the Green function (or $0$-order resolvent) estimates of the process $X^{\gamma}$.

\begin{lem}\label{lem:green}
Suppose that the process $X$ satisfies Assumption {\bf(H)}.
Then for any $\gamma\in (0,1]$,
the following estimates hold.
\begin{itemize}
\item[{\rm(1)}]
Under {\rm (NDLHK)},
\begin{equation}\label{eq:green-lower-1}
u^{\gamma}_0(x,y)
\succeq
\int_{\phi^{\gamma}(d(x,y))}^{\infty} \frac{1}{V(x,(\phi^{\gamma})^{-1}(t))}\,{\rm d}t, \quad x,y\in M_0.
\end{equation}
\item[{\rm(2)}]
Under {\rm (WUHK)},
\begin{equation}\label{eq:green-lower-2}
u^{\gamma}_0(x,y)
\preceq
\int_{\phi^{\gamma}(d(x,y))}^{\infty} \frac{1}{V(x,(\phi^{\gamma})^{-1}(t))}\,{\rm d}t \quad x,y\in M_0.
\end{equation}
\end{itemize}
\end{lem}

We omit the proof of Lemma \ref{lem:green}
because it is similar to that of Lemma \ref{lem:heat-sub}.

\subsection{Hausdorff dimensions
of level sets}
In this subsection, we will determine the Hausdorff dimensions
of the level sets for the process $X$.
First, we recall the definition of the Hausdorff dimension.
Let $\varphi$ be a continuous Hausdorff function of finite order such that $\varphi(0)=0$
(see Definition \ref{def:hausdorff}).
Let ${\cal H}^{\varphi}$ denote the associated Hausdorff measure on the metric measure space $M$.
If $\varphi(t)=t^p$ for some $p>0$, then we write ${\cal H}^p$ for ${\cal H}^{\varphi}$.
For a subset $A$ of $M$, let ${\rm dim}_{\cal H}(A)$ denote its Hausdorff dimension, i.e.,
$${\rm dim}_{\cal H}(A)=\inf\left\{s>0: {\cal H}^s(A)=0\right\}
=\sup\left\{s>0: {\cal H}^s(A)=\infty\right\}.$$

For any fixed $a\in M$, let
\begin{equation}\label{eq:gam}
\gamma_a(s)=\inf\left\{\gamma>0:
\int_0^1\frac{((\phi^{\gamma})^{-1}(t))^s}{V(a,(\phi^{\gamma})^{-1}(t))}\,{\rm d}t<\infty\right\},\quad s\ge0.
\end{equation}
Then, the main result of this part can be stated as follows.

\begin{thm}\label{thm:level}
Suppose that the process $X$ satisfies Assumption {\bf(H)} and {\rm (ODHK)}. We have the following two statements.
\begin{enumerate}
\item[{\rm (1)}]
Let $a\in M$. If $0<\gamma_a(0)\leq 1$, then
\begin{equation}\label{e:upper}
{\rm dim}_{\cal H}\{s>0: X_s=a\}\leq 1-\gamma_a(0), \quad \text{$P_x$-a.s.\ for any $x\in M_0$.}
\end{equation}
On the other hand, if $\gamma_a(0)>1$, then $\{s>0: X_s=a\}=\emptyset$, $P_x$-a.s.\ for any $x\in M_0$.
\item[{\rm (2)}]
Suppose  that $0<\gamma_a(0)<1$ for any $a\in M$.
Then $\N=\emptyset$ and thus $M_0=M$.
Moreover, if the process $X$ also satisfies {\rm (NDLHK)} and $I^{1}(a)=\infty$ for any $a\in M$,
then
\begin{equation}\label{e:lower}
{\rm dim}_{\cal H}\{s>0 : X_s=a\}=1-\gamma_a(0), \quad \text{$P_x$-a.s.\ for any $x\in M$.}
\end{equation}
\end{enumerate}
\end{thm}

We will prove Theorem \ref{thm:level} by following the argument of  \cite[Theorem 1]{H71}
(see also the proof of \cite[Theorem 2.1]{KS15}).
To do so, we  need two lemmas.
\begin{lem}\label{lem:gamma}
Let $a\in M$.
Then the function $s\mapsto \gamma_a(s)$ is
nonincreasing  and Lipschitz continuous on $[0,\infty)$.
Moreover, there exists a constant $s_0>0$ such that $\gamma_a(s)=0$ for any $s\geq s_0$
and $\gamma_a(s_1)>\gamma_a(s_2)>0$ if $0\leq s_1<s_2<s_0$.
\end{lem}

\begin{proof}
We split the proof into four steps.

(i) We show that the function $s\mapsto \gamma_a(s)$ is nonincreasing.
By the change of variables formula with $u=t^{1/\gamma}$, for any $\gamma>0$,
we have
\begin{equation}\label{eq:change}
\int_0^1\frac{\{(\phi^{\gamma})^{-1}(t)\}^s}{V(a,(\phi^{\gamma})^{-1}(t))}\,{\rm d}t
=\gamma \int_0^1\frac{(\phi^{-1}(t))^s}{V(a,\phi^{-1}(t))}t^{\gamma-1}\,{\rm d}t.
\end{equation}
Hence, if $s_2>s_1\geq 0$, then $\gamma_a(s_2)\leq \gamma_a(s_1)$ because
$$\int_0^1\frac{(\phi^{-1}(t))^{s_2}}{V(a,\phi^{-1}(t))}t^{\gamma-1}\,{\rm d}t
\leq  \int_0^1\frac{(\phi^{-1}(t))^{s_1}}{V(a,\phi^{-1}(t))}t^{\gamma-1}\,{\rm d}t, $$
thanks to the fact that $\phi$ is increasing on $[0,1]$ with $\phi(0)=0$ and $\phi(1)=1$.

(ii) We prove that there exists a constant $s_0\in (0,\infty)$ such that
$\gamma_a(s)>0$ for $s\in [0,s_0)$ and $\gamma_a(s)=0$ for $s\geq s_0$.
By \eqref{eq:sca-phi-inv} and \eqref{eq:v-inverse},
there exist positive constants $c_i$ and  $\eta_i$ $(1 \leq i\leq 4)$
such that
\begin{equation}\label{eq:s-inv}
c_1t^{\eta_1}\leq \phi^{-1}(t)\leq c_2t^{\eta_2}, \quad  0\le t\le 1
\end{equation}
and
$$
c_3t^{\eta_3}\leq V(x,\phi^{-1}(t))\leq c_4t^{\eta_4}, \quad  x\in M, \ 0\le t\le 1.
$$
Here, the constants $c_3,c_4$ may depend on $x\in M$.
Hence, if we define
$$s_0=\inf\left\{s>0 : \int_0^1\frac{(\phi^{-1}(t))^s}{V(a,\phi^{-1}(t))t}\,\d t<\infty\right\},$$
then $s_0\in (0,\infty)$.
For any $s>s_0$, we have
\begin{equation}\label{eq:int-finite-s}
\int_0^1\frac{(\phi^{-1}(t))^s}{V(a,\phi^{-1}(t))t}\,\d t<\infty
\end{equation}
so that $\gamma_a(s)=0$.

We now show that $\gamma_a(s_0)=0$ by contradiction.
Assume that $\gamma_a(s_0)>0$.
Then for any $\gamma\in (0,\gamma_a(s_0)
)$ and $s>s_0$,
we obtain by \eqref{eq:s-inv},
\begin{equation*}
\begin{split}
\infty
=\int_0^1\frac{(\phi^{-1}(t))^{s_0}}{V(a,\phi^{-1}(t))}t^{\gamma-1}\,{\rm d}t
&=\int_0^1\frac{(\phi^{-1}(t))^{s}}{V(a,\phi^{-1}(t))t}t^{\gamma}(\phi^{-1}(t))^{s_0-s}\,{\rm d}t\\
&\leq \frac{1}{c_1^{s-s_0}}\int_0^1\frac{(\phi^{-1}(t))^{s}}{V(a,\phi^{-1}(t))t}t^{\gamma-(s-s_0)\eta_1}\,{\rm d}t.
\end{split}
\end{equation*}
In particular, if we take $s>s_0$ so that $(s-s_0)\eta_1<\gamma<\gamma_a(s_0)
$,
then, by \eqref{eq:int-finite-s},
$$\int_0^1\frac{(\phi^{-1}(t))^{s}}{V(a,\phi^{-1}(t))t}t^{\gamma-(s-s_0)\eta_1}\,{\rm d}t
\leq \int_0^1\frac{(\phi^{-1}(t))^{s}}{V(a,\phi^{-1}(t))t}\,{\rm d}t<\infty.$$
Since the two inequalities above yield a contradiction,
we have $\gamma_a(s_0)=0$ as desired.

We also prove that $\gamma_a(s)>0$ for any $s\in [0,s_0)$ by contradiction.
Assume that $\gamma_a(s)=0$ for some $s\in [0,s_0)$.
Then, for any $s_1\in (s,s_0)$,
we obtain by \eqref{eq:s-inv},
$$\int_0^1\frac{(\phi^{-1}(t))^{s_1}}{V(a,\phi^{-1}(t))t}\,\d t
=\int_0^1\frac{(\phi^{-1}(t))^{s}}{V(a,\phi^{-1}(t))t}(\phi^{-1}(t))^{s_1-s}\,\d t
\leq c_2^{s_1-s}\int_0^1\frac{(\phi^{-1}(t))^{s}}{V(a,\phi^{-1}(t))}t^{(s_1-s)\eta_2-1}\,\d t.$$
Since
$s_1<s_0$,
the left hand side above is divergent;
however, we have $(s_1-s)\eta_2>\gamma_a(s)(=0)$ so that
the right hand side is convergent by  \eqref{eq:change}.
We thus get
a contradiction so that $\gamma_a(s)>0$ for any $s\in [0,s_0)$.

(iii) We show that $\gamma_a(s_1)>\gamma_a(s_2)$ if $0\leq s_1<s_2\leq s_0$.
If  $\gamma_a(s_1)=\gamma_a(s_2)(>0)$
for some nonnegative constants $s_1$ and $s_2$ with $s_1<s_2<s_0$,
then for any $\gamma>0$, we have by \eqref{eq:s-inv},
\begin{equation}\label{eq:int-comp}
\begin{split}
\int_0^1\frac{(\phi^{-1}(t))^{s_1}}{V(a,\phi^{-1}(t))}t^{\gamma-1}\,{\rm d}t
&=\int_0^1\frac{(\phi^{-1}(t))^{s_2}}{V(a,\phi^{-1}(t))}(\phi^{-1}(t))^{-(s_2-s_1)}t^{\gamma-1}\,{\rm d}t\\
&\geq \frac{1}{c_2^{s_2-s_1}}\int_0^1\frac{(\phi^{-1}(t))^{s_2}}{V(a,\phi^{-1}(t))}t^{\gamma-(s_2-s_1)\eta_2-1}\,{\rm d}t.
\end{split}
\end{equation}
Let $\gamma>0$ satisfy
$0<\gamma-\gamma_a(s_2)=\gamma-\gamma_a(s_1)<(s_2-s_1)\eta_2$.
Then the left hand side of \eqref{eq:int-comp} is convergent
but the right hand side is divergent.
We thus get  a contradiction so that
$\gamma_a(s_1)>\gamma_a(s_2)$ if $0\leq s_1<s_2\leq s_0$.

(iv) We  prove that the function $s\mapsto \gamma_a(s)$ is Lipschitz continuous on $[0,\infty)$.
If $0\leq s_1<s_2\leq s_0$, then for any $\gamma>0$, we have by \eqref{eq:s-inv},
$$
\int_0^1\frac{(\phi^{-1}(t))^{s_1}}{V(a,\phi^{-1}(t))}t^{\gamma-1}\,{\rm d}t
\leq \frac{1}{c_1^{s_2-s_1}}\int_0^1\frac{(\phi^{-1}(t))^{s_2}}{V(a,\phi^{-1}(t))}t^{\gamma-(s_2-s_1)\eta_1-1}\,{\rm d}t.
$$
Note also that $\gamma_a(s_1)>0$ by (ii).
Hence if $0<\gamma<\gamma_a(s_1)$, then $\gamma_a(s_2)\geq \gamma-(s_2-s_1)\eta_1$.
In particular, since $\gamma_a(s_1)-(s_2-s_1)\eta_1\leq \gamma_a(s_2)$,
the function $\gamma_a(s)$ is Lipschitz continuous on $[0,s_0]$.
Since we know by (ii) that $\gamma_a(s)=0$ for $s\geq s_0$,
the function $\gamma_a(s)$ is Lipschitz continuous on $[0,\infty)$ as well.
\medskip

Putting
the arguments in (i)--(iv) together, 
 we arrive at the desired assertion.
\end{proof}

\begin{lem}\label{lem:cap-heat}
Let the process $X$ satisfy Assumption {\bf (H)}.
For every $a\in M$, if $u_1(a,a)<\infty$, then ${\rm Cap}(\{a\})=1/u_1(a,a)${\rm ;}
otherwise, ${\rm Cap}(\{a\})=0$. In particular, ${\rm Cap}(\{a\})>0$ if and only if $u_1(a,a)<\infty$.
Furthermore, if $X$ satisfies {\rm (ODHK)} as well, then, for each $a\in M$, ${\rm Cap}(\{a\})>0$ if and only if
$$\int_0^1 \frac{1}{V(a,\phi^{-1}(t))}\,\d t<\infty.$$
\end{lem}
\begin{proof} The first assertion is essentially taken from \cite[Example 2.1.2]{FOT11}, and we present the details here for the sake of completeness.
Fix $a\in M$, and let $\delta_a$ be the Dirac measure at $a$.
We first assume that $u_1(a,a)<\infty$.
Then, by \cite[Exercise 4.2.2]{FOT11},
the measure $\delta_a$ is of finite energy integral,
and the function $x\mapsto  u_1(x,a)$ is a quasi-continuous and excessive
 version
of the $1$-potential $U_1\delta_a$
of $\delta_a$.
Furthermore, by \cite[Lemma 2.2.6 and the subsequent comment]{FOT11},
the function $e_a(x)=u_1(x,a)/u_1(a,a)$ is a version of the 1-equilibrium potential
 of $\{a\}$.
Hence
$${\rm Cap}(\{a\})={\cal E}_1(e_a,e_a)=\frac{1}{u_1(a,a)}.$$

We next assume that $u_1(a,a)=\infty$.
Then, by \cite[Exercise 4.2.2]{FOT11},  the measure $\delta_a$ is not of finite energy integral.
Let us suppose that ${\rm Cap}(\{a\})>0$.
Then, according to \cite[Lemma 2.2.6 and the subsequent comment]{FOT11} again,
it follows that for some $c>0$,
the measure $c\delta_a$ would be the equilibrium potential  of $\{a\}$,
so that $\delta_a$ is of finite energy integral.
This is a contradiction, and so ${\rm Cap}(\{a\})=0$.

Let us prove the second assertion.
By {\rm (ODHK)},
$$u_1(a,a)\asymp \int_0^1\frac{1}{V(a,\phi^{-1}(t))}\,{\rm d}t
+\int_1^{\infty}\frac{e^{-t}}{V(a,\phi^{-1}(t))}\,{\rm d}t.$$
Note that the second term of the right hand side above is finite,
because the function $t\mapsto V(a,\phi^{-1}(t))$ is nondecreasing.
Then, the proof is complete by the first assertion.
\end{proof}

\begin{proof}[Proof of Theorem {\rm \ref{thm:level}}]
We first prove (1) under the condition that $0<\gamma_a(0)\leq 1$.
Here and in what follows, let ${\rm Cap}^{\gamma}$ denote
the 1-capacity relative to the subordinate process $X^{\gamma}$.
If $0<\gamma<\gamma_a(0)$, then
$$\int_0^1\frac{1}{V(a,(\phi^{\gamma})^{-1}(t))}\,{\rm d}t=\infty,$$
and so  ${\rm Cap}^{\gamma}(\{a\})=0$ by Lemma \ref{lem:cap-heat} applied to $X^{\gamma}$,
also thanks to Lemma \ref{lem:heat-sub}(1).
Therefore, the process $X^{\gamma}$ can not hit the point $a$ by \cite[Theorems 4.1.2 and 4.2.1 (ii)]{FOT11},
that is,
$$0=P_x\otimes P^{\gamma}(\text{$X_{\tau _t}=a$ for some $t>0$})
=E_x\left[P^{\gamma}(\text{$\tau_t\in\{s>0: X_s=a\}$ for some $t>0$})\right].$$
This implies that
$$P^{\gamma}(\text{$\tau_t\in\{s>0: X_s(\omega)=a\}$ for some $t>0$})=0,
\quad \text{$P_x$-a.s.\ $\omega\in \Omega$ for any $x\in M_0$}.$$
Then, by the Frostman lemma for the $\gamma$-stable subordinator
(see \cite[Section 3]{H71} or \cite[Lemma 2.1]{H74}),
$${\rm dim}_{\cal H}\{s>0: X_s=a\}\leq 1-\gamma, \quad \text{$P_x$-a.s.\ for any $x\in M_0$.}$$
Letting $\gamma\uparrow\gamma_a(0)$ along a sequence, we have \eqref{e:upper}.

If $\gamma_a(0)>1$, then
$$\int_0^1\frac{1}{V(a,\phi^{-1}(t))}\,{\rm d}t=\infty.$$
Hence,
by Lemma \ref{lem:cap-heat} applied to $X$, we have $\Capa(\{a\})=0$ and thus
the process $X$ can not hit the point $a$ by \cite[Theorems 4.1.2 and 4.2.1 (ii)]{FOT11} again.
The proof of (1) is complete.

We next prove (2).
Assume that $\gamma_a(0)<1$ for any $a\in M$.
Then for any  $\gamma\in (\gamma_a(0),1]$,
since
\begin{equation}\label{eq:a-integrable}
\int_0^1\frac{1}{V(a,(\phi^{\gamma})^{-1}(t))}\,{\rm d}t<\infty,
\end{equation}
we have $\Capa^{\gamma}(\{a\})>0$ by Lemma \ref{lem:cap-heat} applied to $X^{\gamma}$,
also due to Lemma \ref{lem:heat-sub}(1) again.
In particular, it follows by \cite[Theorems 4.1.3 and A.2.6 (i)]{FOT11}
that the point $a$ is regular relative to $X^{\gamma}$ for any $\gamma\in (\gamma_a(0),1]$,
i.e.,
\begin{equation}\label{eq:a-start}
\begin{split}
1
&=P_a^{\gamma}(\text{for any $\varepsilon>0$, there exists $t\in (0,\varepsilon)$ such that $X_{\tau_t}=a$})\\
&=P_a^{\gamma}(\text{$X_{\tau_t}=a$ for some $t>0$})
=E_a\left[P^{\gamma}(\text{$\tau_t\in\{s>0: X_s=a\}$ for some $t>0$})\right].
\end{split}
\end{equation}

On the other hand, since \eqref{eq:a-integrable} is valid with $\gamma=1$,
we have $\Capa(\{a\})>0$ for any $a\in M$,
which implies that  $\N=\emptyset$ and $P_x(\sigma_a<\infty)>0$ for any $x\in M$.
Furthermore,
the process $X$ is irreducible and recurrent
by Assumption {\bf (H)}, {\rm (NDLHK)} and $I^1(a)=\infty$ for any $a\in M$,
with the comment just before Lemma \ref{lem:green}.
Hence by \cite[Theorem 4.7.1 (iii) and Exercise 4.7.1]{FOT11},
we obtain  $P_x(\sigma_a<\infty)=1$ for any $x\in M$.
Note that $X_{\sigma_a}=a$ because $\{a\}$ is closed in $M$.
Therefore, by \eqref{eq:a-start} and the strong Markov property of the process $X$,
\begin{equation*}
\begin{split}
1&=P_x(\sigma_a<\infty)
=E_x\left[E_{X_{\sigma_a}}\left[P^{\gamma}(\text{$\tau_t\in\{s>0: X_s=a\}$ for some $t>0$})\right];\sigma_a<\infty\right]\\
&=E_x\left[P^{\gamma}(\text{$\tau_t\in\{s>0: X_s\circ\theta_{\sigma_a}=a\}$ for some $t>0$});\sigma_a<\infty\right]\\
&\leq E_x\left[P^{\gamma}(\text{$\tau_t\in\{s>0: X_s=a\}$ for some $t>0$})\right],
\end{split}
\end{equation*}
which yields
$$P^{\gamma}(\text{$\tau_t\in\{s>0: X_s(\omega)=a\}$ for some $t>0$})=1,
\quad \text{$P_x$-a.s.\ $\omega\in \Omega$ for any $x\in M$}.$$
By using \cite[Section 3]{H71} or \cite[Lemma 2.1]{H74} again,
$${\rm dim}_{\cal H}\{s>0 : X_s=a\}\geq 1-\gamma, \quad \text{$P_x$-a.s.\ for any $x\in M$.}$$
Letting $\gamma\downarrow\gamma_a(0)$ along a sequence, we have
$${\rm dim}_{\cal H}\{s>0 : X_s=a\}\geq 1-\gamma_a(0), \quad \text{$P_x$-a.s.\ for any $x\in M$.}$$
Combining this with \eqref{e:upper}, we get \eqref{e:lower}.
\end{proof}

\begin{exam}\label{exam:level}
Let the process $X$ satisfy Assumption {\bf (H)}, {\rm (ODHK)} and {\rm (NDLHK)}.
We impose the next conditions on the functions $V(x,r)$ and $\phi(r)$ {\rm:}
\begin{itemize}
\item
There exist positive constants $d_1$, $d_2$ and $c_i$, $1\leq i\leq 4,$ such that
$$c_1r^{d_1}\leq V(x,r)\leq c_2r^{d_1}, \quad x\in M, \ r\in (0,1)$$
and
$$c_3r^{d_2}\leq V(x,r)\leq c_4r^{d_2}, \quad x\in M, \ r\in [1,\infty).$$
\item
There exist positive constants $\alpha$, $\beta$, $c_i$,  $ 5\leq i\leq 8,$ such that
$$c_5r^{\alpha}\leq \phi(r)\leq c_6r^{\alpha}, \quad r\in (0,1)$$
and
$$c_7r^{\beta}\leq \phi(r)\leq c_8r^{\beta}, \quad r\in [1,\infty).$$
\end{itemize}
Then for any $a\in M$, $\gamma_a(s)=(d_1-s)/\alpha$ for any $s\in [0,d_1]$, and
$\gamma_a(0)\le 1$ if and only if $0<d_1\le \alpha$.
We also see that  $I^1(a)=\infty$ for any $a\in M$ if and only if $0<d_2\leq \beta$.

By the calculation above and Theorem $\ref{thm:level}$, we have
the following{\rm :} if $0<d_1\le \alpha$ and $0<d_2\leq \beta$,
then
$${\rm dim}_{\cal H}\{s>0 : X_s=a\}=1-\frac{d_1}{\alpha}, \quad \text{$P_x$-a.s.\ for any $x\in M$.}$$
If $d_1>\alpha$, then $\{s>0 : X_s=a\}=\emptyset$, $P_x$-a.s.\ for any $x\in M$.
\end{exam}

\subsection{Hausdorff dimensions
of inverse images}

In this subsection, we determine the Hausdorff dimensions
of the inverse images
for the process $X$.
For this purpose, we make a stronger assumption on the volume
function.
\begin{assum}\label{assum:volume}\it
There exists a strictly increasing function $V(r)$ on $[0,\infty)$ so that $V(0)=0$
and that there are some positive constants $c_1$ and $c_2$ so that for all $x\in M$ and $r\geq 0$,
$$c_1V(r)\leq V(x,r)\leq c_2V(r).$$
\end{assum}

Note that under the assumption above,
the value $\gamma_a(u)$ defined by \eqref{eq:gam} is independent of the choice of $a\in M$.
Hence we write $\gamma(u)$ for $\gamma_a(u)$. In other words,
$$\gamma(s)=\inf\left\{\gamma>0:
\int_0^1\frac{((\phi^{\gamma})^{-1}(t))^s}{V((\phi^{\gamma})^{-1}(t))}\,{\rm d}t<\infty\right\},\quad s\ge0.$$
We also define
$$s_0=\inf\left\{s>0 : \int_0^1\frac{(\phi^{-1}(t))^s}{V(\phi^{-1}(t))t}\,{\rm d}t<\infty\right\}.$$
Then, by the proof of Lemma \ref{lem:gamma},
the function $s\mapsto \gamma(s)$ is
Lipschitz continuous on $[0,\infty)$ and $s_0$ defined above is positive; moreover, $\gamma(s)$ is
strictly decreasing on $[0,s_0]$ such that  $\gamma(s)=0$ for $s\geq s_0$.

We also introduce the next assumption on $M$ in order
for the validity of Proposition \ref{prop:frost-2} below.
\begin{assum}\label{assum:compact}
Any closed ball in $M$ is compact.
\end{assum}

\begin{thm}\label{thm:upper}
Let $F$ be a Borel subset of $M$ such that  $s_F={\rm dim}_{\cal H}(F)>0$.
Suppose that the process
$X$ satisfies Assumption {\bf(H)}, and that 
 Assumption {\rm \ref{assum:volume}} holds.
Suppose also that for any $s\geq 0$ with $\gamma(s)>0$ and for any $\gamma\in (0,\gamma(s))$,
there exists a constant $c_1>0$ such that
for any $T\in (0,1/2)$,
\begin{equation}\label{assum:regularity}
\int_T^1\frac{(\phi^{-1}(u))^s}{V(\phi^{-1}(u))}u^{\gamma-1}\,\d u
\leq \frac{c_1(\phi^{-1}(T))^s}{V(\phi^{-1}(T))}T^{\gamma}
\end{equation}
\begin{itemize}
\item[{\rm (1)}]
Under {\rm(NDLHK)},  if $0\leq \gamma(s_F)\leq 1$, then
$${\rm dim}_{\cal H}\{t>0: X_t\in F\}\leq 1-\gamma(s_F),
\quad \text{$P_x$-a.s.\ for any $x\in M_0$.}$$
On the other hand, if $\gamma(s_F)> 1$, then $\{t>0 :X_t\in F\}=\emptyset$,
$P_x$-a.s.\ for any $x\in M_0$.
\item[{\rm(2)}] Suppose further that
$M$ satisfies Assumption {\rm \ref{assum:compact}}, and
the process $X$ satisfies
Assumption {\bf(H)},  {\rm (NDLHK)} and {\rm (WUHK)} with $M\setminus \N$ replaced by $M$.
If $\gamma(s_F)>0$ and
\begin{equation}\label{eq:int-rec}
\int_1^{\infty}\frac{1}{V(\phi^{-1}(t))}\,{\rm d}t=\infty,
\end{equation}
then
$${\rm dim}_{\cal H}\{t>0: X_t\in F\}\geq 1-\gamma(s_F), \quad
\text{$P_x$-a.s.\ for any  $x\in M$}.$$
\end{itemize}
\end{thm}

To prove Theorem \ref{thm:upper}, we need the following lemma.

\begin{lem}\label{lem:borel}
Suppose that the process $X$ satisfies Assumption {\bf (H)}.
If $A$ is a subset of $M$, then $A^r$ {\rm (}relative to $X${\rm )} is Borel measurable.
\end{lem}

\begin{proof}
Under Assumption {\bf(H)}, if $B$ is a universally measurable subset of $M$, then
$\mu(B)=0$ implies $U_{\lambda}(x,B)=0$ for any $x\in M_0$.
Namely, $\mu$ is a reference measure for the process $X$
in the sense of \cite[Definition 1.1 in Chapter V]{BG68} or \cite[p.112]{CW05}.
Hence, the assertion follows from \cite[Corollary 1.14 in Chapter V]{BG68} or
\cite[p.115, Exercise 3]{CW05}.
\end{proof}

We also need the notation for the energy of a Borel measure.
Let $\psi:[0,\infty)\to [0,\infty)$ be a Borel measurable function.
For a Borel measure $\nu$ on $M$, define
$$I^{\psi}(\nu)
 =\iint_{M\times M}\frac{1}{\psi(d(x,y))}\,\nu({\rm d}x)\,\nu({\rm d}y).
$$
Then, $I^{\psi}(\nu)$ is called the $\psi$-energy of $\nu$.
If $\psi(t)=t^s$ for some $s>0$, then we write $I^{\psi}$ as $I^s$.

\begin{proof}[Proof of Theorem {\rm\ref{thm:upper}}]
We first prove (1) under the condition that $\gamma(s_F)\leq 1$.
Let $F$ be a Borel subset of $M$.
Without loss of generality, we assume that $\gamma(s_F)>0$.
Then, by Lemma \ref{lem:gamma} and its proof,
there exists $\delta\in (0,s_F/2)$ such that $\gamma(u)>0$ for any $u\in (s_F,s_F+\delta)$.
If we fix $u\in (s_F,s_F+\delta)$, then
$\gamma(u)<\gamma(s)<\gamma(s_F)$ for any $s\in (s_F,u)$,
thanks to Lemma \ref{lem:gamma} again.
Therefore, for any $C>0$, there exists $T_0\in (0,1/2
)$ such that
$$C \leq \int_{T_0}^1\frac{(\phi^{-1}(t))^s}{V(\phi^{-1}(t))}t^{\gamma(u)-1}\,d t.$$
In particular, for any  $x,y\in M$ with $d(x,y)\le \phi^{-1}(T_0)$,
it follows by \eqref{assum:regularity} that
$$C\leq \int_{\phi(d(x,y))}^1\frac{(\phi^{-1}(t))^s}{V(\phi^{-1}(t))}t^{\gamma(u)-1}\,d t
\leq \frac{c_1d(x,y)^s\phi^{\gamma(u)}(d(x,y))}{V(d(x,y))}.$$
This implies that for any compact subset $K$ of $F$, there exists a constant $C_0:=C_0(K)>0$ such that
\begin{equation}\label{eq:compare}
\frac{C_0}{d(x,y)^s}\leq \frac{c_2\phi^{\gamma(u)}(d(x,y))}{V(d(x,y))}, \quad x,y\in K.
\end{equation}

Let $X^{\gamma(u)}$ be the $\gamma(u)$-stable subordinate process of the process $X$,
and $(\E^{\gamma(u)},\F^{\gamma(u)})$ the associated regular Dirichlet form.
We now assume that there exists a finite and nontrivial Borel measure $\nu$ on $M$ such that
it is compactly supported in $K$
and charges no set of zero capacity relative to $(\E^{\gamma(u)},\F^{\gamma(u)})$.
Then for any $s\in (s_F,u)$, since ${\cal H}^s(K)=0$,
Proposition \ref{prop:frost-1} yields $I^s(\nu)=\infty$.
Combining this with \eqref{eq:compare}, we obtain
\begin{equation}\label{eq:energy-0}
\iint_{K\times K}\frac{\phi^{\gamma(u)}(d(x,y))}{V(d(x,y))}\,\nu({\rm d}x)\,\nu({\rm d}y)=\infty.
\end{equation}
Let $u_1^{\gamma(u)}(x,y)$ be the $1$-resolvent kernel for $X^{\gamma(u)}$.
According to Lemma \ref{lem:heat-sub}(2), under {\rm(NDLHK)},
\begin{equation}\label{eq:energy-1}
\iint_{K\times K}\frac{\phi^{\gamma(u)}(d(x,y))}{V(d(x,y))}\,\nu({\rm d}x)\nu({\rm d}y)
\leq c_3 \iint_{K\times K}u_1^{\gamma(u)}(x,y)\,\nu({\rm d}x)\nu({\rm d}y)
= c_3 \int_K U_1^{\gamma(u)}\nu\,{\rm d}\nu,
\end{equation}
where the constant $c_3>0$ may depend on the set $K$.
In particular, \eqref{eq:energy-0}
and \eqref{eq:energy-1} yield
\begin{equation}\label{eq:potential-1}
\int_K U_1^{\gamma(u)}\nu\,{\rm d}\nu=\infty.
\end{equation}

Let $\nu_K^{\gamma(u)}$ be the equilibrium measure of $K$
relative to $(\E^{\gamma(u)},\F^{\gamma(u)})$.
Since $K$ is compact, we have
$${\Capa}^{\gamma(u)}(K)=\nu_K^{\gamma(u)}(K)
=\int_K U_1^{\gamma(u)}\nu_K^{\gamma(u)}\,{\rm d}\nu_K^{\gamma(u)}<\infty.
$$
On the other hand, if $\nu_K^{\gamma(u)}$ is nontrivial,
then we have $\nu_K^{\gamma(u)}(K)=\infty$
by \eqref{eq:potential-1} with $\nu=\nu_K^{\gamma(u)}$.
This  is a contradiction so that we get  ${\Capa}^{\gamma(u)}(K)=\nu_K^{\gamma(u)}(K)=0$.
By the regularity of the capacity (\cite[(2.1.6)]{FOT11}),
we further obtain ${\Capa}^{\gamma(u)}(F)=0$.
This and \cite[Theorem 4.2.1 (ii)]{FOT11} yield
$$
0
=P_x\otimes P^{\gamma(u)}(\text{$X_t^{\gamma(u)}\in F$ for some $t>0$})\
=E_x\left[P^{\gamma(u)}(\text{$\tau_t \in \{t>0: X_t\in F\}$ for some $t>0$})\right]
$$
and thus
$$P^{\gamma(u)}(\text{$\tau_t \in \{t>0:X_t(\omega)\in F\}$ for some $t>0$})=0,
\quad \text{$P_x$-a.s.\  $\omega\in \Omega$ for any $x\in M_0$.}$$
Then, by \cite[Section 3]{H71} or \cite[Lemma 2.1]{H74} again, we have
$$\dim\left\{t>0:X_t\in F\right\}\leq 1-\gamma(u), \quad \text{$P_x$-a.s.\ for any $x\in M_0$.}$$
Letting $u\downarrow s_F$ along a sequence,
we arrive at the assertion (1) provided that $\gamma(s_F)\leq 1$.

If $\gamma(s_F)>1$, then, by the proof of Lemma \ref{lem:gamma},
we can take $u>s_F$ so that $\gamma(u)=1$.
Hence the same argument as before implies that $\Capa(F)=0$, and thus
$P_x(\{t>0: X_t\in F\}=\emptyset)=1$ for any $x\in M_0$.
The proof of (1) is complete.

We next prove (2).
Without loss of generality, we assume that $s_F>0$ and $0<\gamma(s_F)<1$.
Then, by Lemma \ref{lem:gamma},
there exists a constant $\varepsilon>0$ such that for all $s\in (s_F-\varepsilon, s_F)$,
$\gamma(s_F)<\gamma(s)<1$.
We now fix such $s\in (s_F-\varepsilon,s_F)$.
Then
the regularity of the Hausdorff measure
yields ${\cal H}^s(K)>0$ for some compact subset $K$ of $F$.
Under Assumption \ref{assum:compact}, we can further use Proposition \ref{prop:frost-2} to show that
there exists a finite and nontrivial Borel measure $\nu_K^s$ on $M$ such that
$\supp[\nu_K^s]\subset K$ and $I^s(\nu_K^s)<\infty$.

On the other hand,
Lemma \ref{lem:gamma} implies again that  $\gamma(s)<\gamma(v)<1$
for any $v\in (s_F-\varepsilon, s)$.
Then for any $v\in (s_F-\varepsilon, s)$ and  $T\in (0,1/2)$,
$$\infty>\int_0^1\frac{(\phi^{-1}(t))^s}{V(\phi^{-1}(t))}t^{\gamma(v)-1}\,{\rm d}t
\geq \int_T^{2T}\frac{(\phi^{-1}(t))^s}{V(\phi^{-1}(t))}t^{\gamma(v)-1}\,{\rm d}t
\geq \frac{c_1 (\phi^{-1}(T))^s T^{\gamma(v)}}{V(\phi^{-1}(T))},$$
which implies that for some $c_2>0$,
\begin{equation}\label{eq:upper-0}
\frac{\phi^{\gamma(v)}(d(x,y))}{V(d(x,y))}\leq\frac{c_2}{d(x,y)^s}, \quad x,y\in K.
\end{equation}

Let $X^{\gamma(v)}$ be the $\gamma(v)$-stable subordinate process of the process $X$.
Since $\gamma(v)<\gamma(0)$, it follows by
Lemma \ref{lem:heat-sub}(3) and \eqref{assum:regularity} that
under {\rm(WUHK)},
there exists a constant $c_3>0$ such that for all $x,y\in K$,
$$\frac{\phi^{\gamma(v)}(d(x,y))}{V(d(x,y))}
\geq  c_3u_1^{\gamma(v)}(x,y).$$
Combining this with  \eqref{eq:upper-0}, we have for some $c_4>0$,
$$\frac{1}{d(x,y)^s}\geq c_4 u_1^{\gamma(v)}(x,y),\quad  x,y\in K.$$
Therefore,
$$\infty>I^s(\nu_K^s)
=\iint_{K\times K}\frac{1}{d(x,y)^s}\,\nu_K^s({\rm d}x)\nu_K^s({\rm d}y)
\geq c_4
\iint_{K\times K}u_1^{\gamma(v)}(x,y)
\,\nu_K^s({\rm d}x)\nu_K^s({\rm d}y).$$
Then, by \cite[Exercise 4.2.2]{FOT11}, the measure $\nu_K^s$ is of finite energy integral
relative to $X^{\gamma(v)}$.
Moreover, since $\nu_K^{s}$ is nontrivial,
${\Capa}^{\gamma(v)}(F)\geq \Capa^{\gamma(v)}(K)>0$ thanks to \cite[Theorem 2.2.3]{FOT11}.
In particular, for all $s\in (s_F-\varepsilon, s_F)$, ${\Capa}^{\gamma(s)}(F)>0$.

We now follow the argument of \cite[Theorem 1]{L95}.
Let $\sigma_F$ be the hitting time of $X^{\gamma(s)}$ to $F$, i.e.,
$\sigma_F=\inf\{t>0 :X_t^{\gamma(s)}\in F\}$.
Define
$$F_{\gamma(s)}=\left\{x\in M : P_x\otimes P^{\gamma(s)}(\sigma_F<\infty)=1\right\},
\quad F_{\gamma(s)}^r=\left\{x\in M: P_x\otimes P^{\gamma(s)}(\sigma_F=0)=1\right\}.$$
Namely, $F_{\gamma(s)}^r$ is the totality of regular points of $F$
relative to the process $X^{\gamma(s)}$.
By Lemma \ref{lem:borel} applied to the process $X^{\gamma(s)}$,  $F_{\gamma(s)}^r$ is a Borel subset of $M$.
Since $F\setminus F_{\gamma(s)}^r$ is exceptional
(see \cite[Theorem 4.1.3 and Theorem A.2.6 (i)]{FOT11}),
it follows from \cite[Theorem 4.2.1 (ii)]{FOT11} that ${\Capa}^{\gamma(s)}(F\setminus F_{\gamma(s)}^r)=0$.
This and ${\Capa}^{\gamma(s)}(F)>0$ yield
${\Capa}^{\gamma(s)}(F_{\gamma(s)}^r)>0$  and so $F_{\gamma(s)}^r\ne \emptyset$.
On the other hand,
since it follows from \cite[Proposition 2.8 (p.73) and Proposition 1.4 (p.197)]{BG68}
that the function $g(x):=P_x\otimes P^{\gamma(s)}(\sigma_F<\infty)$ is Borel measurable and excessive,
the set $F_{\gamma(s)}$ is also Borel measurable.
As $F_{\gamma(s)}^r\subset F_{\gamma(s)}$ by definition,
$F\setminus F_{\gamma(s)}$ is also exceptional and
${\Capa}^{\gamma(s)}(F_{\gamma(s)})\geq {\Capa}^{\gamma(s)}(F_{\gamma(s)}^r)>0$.

Since ${\Capa}^{\gamma(s)}(F_{\gamma(s)})>0$,
the regularity of the capacity (\cite[(2.1.6)]{FOT11}) implies that
there exists a compact subset $K$ of $F_{\gamma(s)}$ such that
${\Capa}^{\gamma(s)}(K)>0$.
Then, by \cite[Theorem 2.2 (i)]{O02}, there exists a constant $c_5>0$ such that
${\Capa}(K)\geq c_5{\Capa}^{\gamma(s)}(K)>0$.
We here note that $X$ is irreducible and recurrent
by Assumption ({\bf H}), {\rm (NDLHK)} and \eqref{eq:int-rec} with the comment just before Lemma \ref{lem:green}.
Hence, by \cite[Theorem 4.7.1 (iii) and Exercise 4.7.1]{FOT11},
we have $P_x(\sigma_K<\infty)=1$ for any $x\in M$.
Noting that $X_{\sigma_K}\in K$ and  $K\subset F_{\gamma(s)}$,
we further obtain by the strong Markov property of $X$,
\begin{equation}\label{eq:cpt}
\begin{split}
1=P_x(\sigma_K<\infty)
&=E_x\left[P_{X_{\sigma_K}}\otimes P^{\gamma(s)}(\text{$X_{\tau_t}\in F$ for some $t>0$})\right]\\
&=E_x\left[E_{X_{\sigma_K}}\left[P^{\gamma(s)}\left(\text{$\tau_t\in \left\{u>0 : X_u\in F\right\}$ for some $t>0$}\right)\right]\right]\\
&=E_x\left[P^{\gamma(s)}\left(\text{$\tau_t \in \left\{u>0 : X_u\circ\theta_{\sigma_K}\in F\right\}$ for some $t>0$}\right)\right]\\
&\leq E_x\left[P^{\gamma(s)}\left(\text{$\tau_t \in \left\{u>0 : X_u\in F\right\}$ for some $t>0$}\right)\right].
\end{split}
\end{equation}
Thus
$$P^{\gamma(s)}\left(\text{$\tau_t \in \left\{u>0 : X_u\in F\right\}$ for some $t>0$}\right)=1,
\quad \text{$P_x$-a.s.\ for any $x\in M$.}$$
Then, by \cite[Section 3]{H71} or \cite[Lemma 2.1]{H74} again,
$$\dim\left\{t>0: X_t\in F\right\}\geq 1-\gamma(s),
\quad \text{$P_x$-a.s.\  for any $x\in M$.}$$
Letting $s\uparrow s_F$ along a sequence, we have by Lemma \ref{lem:gamma},
$$\dim\left\{t>0: X_t\in F\right\}\geq 1-\gamma(s_F),
\quad \text{$P_x$-a.s.\  for any $x\in M$.}$$
Combining this with Theorem \ref{thm:upper}, we complete the proof.
\end{proof}

\begin{exam}\label{exam:leve2}
Suppose that the process $X$ satisfies Assumption {\bf (H)}, {\rm (WUHK)} and {\rm(HR)}.
Then, by Remark $\ref{rem:heat-kernel}$ {\rm (iii)},
$X$ satisfies {\rm (NDLHK)} as well,
and all of $M_0$ can be replaced with $M$ under
these conditions.
We now impose the same conditions on the functions $V(x,r)$ and $\phi(r)$ as in Example $\ref{exam:level}$.
Then $\gamma(s)=(d_1-s)/\alpha$ for $s\in [0,d_1]$,
and \eqref{assum:regularity} is fulfilled.
We also see that \eqref{eq:int-rec} is valid if and only if $0<d_2\leq \beta$.

Let $F\subset M$ be a Borel set with $s_F=\dim_{{\cal H}}(F)>0$.
Then $0\leq \gamma(s_F)\leq 1$ if and only if $d_1-\alpha\le s_F\leq d_1$.
Therefore, if $F$ satisfies $s_F>0$, $d_1-\alpha\le s_F\leq d_1$ and $0<d_2\leq \beta$, then
$$\dim_{{\cal H}}\{t>0 : X_t\in F\}=1-\frac{d_1-s_F}{\alpha}, \quad \text{$P_x$-a.s.\ for any $x\in M$.}$$
\end{exam}

With Examples \ref{exam:level} and \ref{exam:leve2},
one can easily get the first assertion (1)
in Theorem \ref{Thm:main},
also thanks to Remark \ref{rem2:heat-kernel}.

\section{Hausdorff dimensions of collision time sets
}\label{section4}

\subsection{Resolvent of stable-subordinate direct-product processes}

For $i=1, 2$, let $X^i:=(\{X_t^i\}_{t\geq 0}, \{P_x\}_{x\in M})$
be a $\mu$-symmetric Hunt process on $M$
generated by a regular Dirichlet form $(\E^i,\F^i)$ on $L^2(M;\mu)$.
We assume that $X^1$ and $X^2$ are independent, and satisfy Assumption {\bf (H)}.
For each $i=1,2$, let $\N^i$ denote the corresponding properly exceptional set and $M_0^i:=M\setminus \N^i$,
and let $p^i(t,x,y)$ be the heat kernel of $X^i$.

For any $t\geq 0$ and $x=(x_1,x_2)\in M\times M$, define
$$X_t=(X_t^1,X_t^2), \quad P_x=P_{x_1}^1\otimes P_{x_2}^2.$$
Then, $X:=(\{X_t\}_{t\geq 0}, \{P_x\}_{x\in M\times M})$ is a $\mu\otimes\mu$-symmetric
Hunt process on $M\times M$.
Let $(\E,\F)$ be the associated  Dirichlet form on $L^2(M\times M;\mu\otimes\mu)$.
Then, by \cite[Theorems 3.1 and 5.1]{O97}, $(\E,\F)$ is regular and irreducible.
We also see by \cite[Theorem 4.3 (3)]{O97}
that
the set $M\times M\setminus(M_0^1\times M_0^2)$ has zero capacity relative to $(\E,\F)$.
By combining this with the relation
$$P_x(X_t\in M\times M)=P_{x_1}(X_t^1\in M)P_{x_2}(X_t^2\in M)=1,
\quad t\geq 0, \ x=(x_1,x_2)\in M_0^1\times M_0^2,$$
$(\E,\F)$ is also conservative by \cite[Exercise 4.5.1]{FOT11}.
The heat kernel of $X$ is given by
$$p(t,x,y)=p^1(t,x_1,y_1)p^2(t,x_2,y_2), \quad t\geq 0, \ x=(x_1,x_2), y=(y_1,y_2)\in M_0^1\times M_0^2.$$

For $\gamma\in (0,1)$,
let $X^{\gamma}:=((X_t^{\gamma})_{t\geq 0}, \{P_x^{\gamma}\}_{x\in M\times M})$
be a subordinate process of $X$
with respect to the $\gamma$-stable subordinator
$S^{\gamma}=(\{\tau_t\}_{t\geq 0},P^{\gamma})$, that is, for any $t\geq 0$ and $x\in M\times M$,
$$X_t^{\gamma}=X_{\tau_t}=(X_{\tau_t}^1,X_{\tau_t}^2), \quad
P_x^{\gamma}=P_x\otimes P^{\gamma}.$$
Let $(\E^{\gamma},\F^{\gamma})$ be the associated  Dirichlet form on $L^2(M\times M;\mu\otimes\mu)$.
Then, by \cite[Theorems 2.1(ii) and 3.1(i)(ii)]{O02},
$(\E^{\gamma},\F^{\gamma})$ is also regular, irreducible and conservative.
The heat kernel of $X^{\gamma}$ is given by
$$q^{\gamma}(t,x,y)=\int_0^{\infty}p(s,x,y)\pi_t(s)\,{\rm d}s,
\quad  t>0, \ x,y\in M_0^1\times M_0^2,$$
where $\pi_t(s)$ is the density function of $S^\gamma_t$.

For $\lambda\geq 0$, let $u_{\lambda}^{\gamma}(x,y)$
be the $\lambda$-resolvent density of $X^{\gamma}$,
i.e., for $x,y\in M_0^1\times M_0^2$,
\begin{equation}\label{eq:res}
u_{\lambda}^{\gamma}(x,y)
=\int_0^{\infty}e^{-\lambda t}q^{\gamma}(t,x,y)\,{\rm d}t
=\int_0^{\infty}e^{-\lambda t} \int_0^{\infty}p(s,x,y)\pi_t(s)\,{\rm d}s \,{\rm d}t.
\end{equation}

In the following, we will assume that
the processes
$X^1$ and $X^2$ satisfy
the common one of the conditions in
Definition \ref{def:heat-kernel-2}.
Under this assumption, we use the notations $\phi_i$, $\alpha_{i1}$ and $\alpha_{i2}$, respectively,
to denote the corresponding scaling function $\phi$ and the associated indices $\alpha_1$, $\alpha_2$.
For $x=(x_1,x_2)\in M\times M$ and $y=(y_1,y_2)\in M\times M$,  let
$$\phi_d(x,y)=\phi_1(d(x_1,y_1))\vee \phi_2(d(x_2,y_2))$$
and $\phi_d^{\gamma}(x,y):=(\phi_d(x,y))^{\gamma}$.
It is clear that
$$\frac{1}{2}\left\{
\phi_1(d(x_1,y_1))+\phi_2(d(x_2,y_2))\right\}
\leq \phi_d(x,y)\leq \phi_1(d(x_1,y_1))+\phi_2(d(x_2,y_2)).$$

We first show the lower bound for the resolvent density of the process $X^{\gamma}$.
\begin{lem}\label{lem:d-res-lower}
Suppose that the independent processes $X^1$ and $X^2$ satisfy Assumption {\bf (H)} and  {\rm (NDLHK)}.
Then, for any $\gamma\in (0,1]$,
there exist positive  constants $c_1$ and $c_2$ such that for any $x,y\in M_0^1\times M_0^2$,
$$
u_1^{\gamma}(x,y)\geq
c_1\int_{\phi_d^{\gamma}(x,y)}^{\infty}
\frac{e^{-t}}{V(x_1,(\phi_1^{\gamma})^{-1}(t))V(x_2,(\phi_2^{\gamma})^{-1}(t))}\,{\rm d}t, \quad \phi_d(x,y)\leq 1
$$
and
$$
u_1^{\gamma}(x,y)\geq
\frac{c_2}{V(x_1,\phi_1^{-1}(\phi_d(x,y)))
V(x_2,\phi_2^{-1}(\phi_d(x,y)))\phi_d^{\gamma}(x,y)}, \quad \phi_d(x,y)\geq 1.
$$
\end{lem}

\begin{proof}
Without loss of generality,
we assume that the processes $X^1$ and $X^2$ satisfy {\rm (NDLHK)} with the constant $c_0=1$ involved in.
Then, by \eqref{eq:sub-lower} and \eqref{eq:res},
\begin{equation}\label{eq:d-res-lower}
\begin{split}
u_1^{\gamma}(x,y)
&\geq c_1\int_0^{\infty}
e^{-t}t\left(\int_{t^{1/\gamma}\vee \phi_d(x,y)}^{\infty}
\frac{1}{V(x_1,\phi_1^{-1}(s))V(x_2,\phi_2^{-1}(s))s^{1+\gamma}}\,{\rm d}s\right)\,{\rm d}t\\
&=c_1\left(\int_0^{\phi_d^{\gamma}(x,y)}e^{-t}t\,{\rm d}t\right)
\left(\int_{\phi_d(x,y)}^{\infty}\frac{1}{V(x_1,\phi_1^{-1}(s))V(x_2,\phi_2^{-1}(s))s^{1+\gamma}}\,{\rm d}s\right)\\
&\quad +c_1\int_{\phi_d^{\gamma}(x,y)}^{\infty}
e^{-t}t\left(\int_{t^{1/\gamma}}^{\infty}\frac{1}{V(x_1,\phi_1^{-1}(s))V(x_2,\phi_2^{-1}(s))s^{1+\gamma}}\,{\rm d}s\right)\,{\rm d}t
=:{\rm (I)}+{\rm (II)}.
\end{split}
\end{equation}
Following \eqref{eq:int-sum-1} and \eqref{eq:scale-comp}, we have
$$
{\rm (I)}\asymp
\frac{\phi_d^{2\gamma}(x,y)\wedge 1}
{V(x_1,\phi_1^{-1}(\phi_d(x,y)))V(x_2,\phi_2^{-1}(\phi_d(x,y)))\phi_d^{\gamma}(x,y)}
$$
and
$${\rm (II)}\asymp
\int_{\phi_d^{\gamma}(x,y)}^{\infty}
\frac{e^{-t}}{V(x_1,\phi_1^{-1}(t^{1/\gamma}))V(x_2,\phi_2^{-1}(t^{1/\gamma}))}\,{\rm d}t
=\int_{\phi_d^{\gamma}(x,y)}^{\infty}
\frac{e^{-t}}{V(x_1,(\phi_1^{\gamma})^{-1}(t))V(x_2,(\phi_2^{\gamma})^{-1}(t))}\,{\rm d}t.$$
Hence, the proof is complete.
\end{proof}

We next show the upper bound of the resolvent of $X^{\gamma}$.
\begin{lem}\label{lem:prod-res}
Suppose that the independent processes $X^1$ and $X^2$ satisfy Assumption {\bf (H)} and {\rm (WUHK)}.
For a fixed constant $\gamma\in (0,1]$,
if there exists a constant $c_1>0$ such that for $i=1,2$,
\begin{equation}\label{eq:prod-res-cond}
\int_0^T\frac{t^{\gamma}}{V(w,\phi_i^{-1}(t))}\,{\rm d}t
\leq \frac{c_1T^{1+\gamma}}{V(w,\phi_i^{-1}(T))}, \quad w\in M, \ T\in (0,1],
\end{equation}
then there exists a constant $c_2>0$ such that for any $x,y\in M_0^1\times M_0^2$,
\begin{equation}\label{eq:prod-res-upper-1}
u_1^{\gamma}(x,y)\leq	
c_2\int_{\phi_d^{\gamma}(x,y)}^2
\frac{1}{V(x_1,(\phi_1^{\gamma})^{-1}(t))V(x_2,(\phi_2^{\gamma})^{-1}(t))}\,{\rm d}t,
\quad \phi_d(x,y)\le 1.
\end{equation}
\end{lem}

\begin{proof}
For any $x,y\in M_0^1\times M_0^2$, write
\begin{equation}\label{eq:j1-j3}
\begin{split}
 u_1^{\gamma}(x,y)
&=\int_0^{\infty}
e^{-t}\left(\int_0^{\infty}p(s,x,y)\pi_t(s)\,{\rm d}s\right)\,{\rm d}t\\
&=\int_0^{\infty}
e^{-t}\left(\int_{\phi_d(x,y)}^{\infty}p(s,x,y)\pi_t(s)\,{\rm d}s\right)\,{\rm d}t
+\int_0^{\infty}
e^{-t}\left(\int_0^{\phi_d(x,y)}p(s,x,y)\pi_t(s)\,{\rm d}s\right)\,{\rm d}t\\
&=:J_1+J_2.
\end{split}
\end{equation}
Then, by (WUHK) and \eqref{eq:sub-upper} with the Fubini theorem,
\begin{equation*}
\begin{split}
J_1
&\leq c_1\int_0^{\infty}e^{-t}t
\left(\int_{\phi_d(x,y)}^{\infty}\frac{e^{-t/s^{\gamma}}}{V(x_1,\phi_1^{-1}(s))V(x_2,\phi_2^{-1}(s))s^{1+\gamma}}\,{\rm d}s\right)\,{\rm d}t\\
&=c_1\int_{\phi_d(x,y)}^{\infty}\frac{1}{V(x_1,\phi_1^{-1}(s))V(x_2,\phi_2^{-1}(s))s^{1+\gamma}}
\left(\int_0^{\infty}e^{-t(1+1/s^{\gamma})}t\,{\rm d}t\right)\,{\rm d}s\\
&=c_1\int_{\phi_d(x,y)}^{\infty}
\frac{1}{V(x_1,\phi_1^{-1}(s))V(x_2,\phi_2^{-1}(s))}\frac{s^{\gamma-1}}{(1+s^{\gamma})^2}\,{\rm d}s.
\end{split}
\end{equation*}
Since $\gamma>0$, we have
$$\int_{2^{1/\gamma}}^{\infty}\frac{1}{V(x_1,\phi_1^{-1}(s))V(x_2,\phi_2^{-1}(s))}\frac{s^{\gamma-1}}{(1+s^{\gamma})^2}\,{\rm d}s
\leq \int_{2^{1/\gamma}}^{\infty}\frac{1}{V(x_1,\phi_1^{-1}(s))V(x_2,\phi_2^{-1}(s))s^{1+\gamma}}\,{\rm d}s<\infty.$$
If $\phi_d(x,y)\le 1$, then
\begin{equation*}
\begin{split}
&\int_{\phi_d(x,y)}^{2^{1/\gamma}} \frac{1}{V(x_1,\phi_1^{-1}(s))V(x_2,\phi_2^{-1}(s))}\frac{s^{\gamma-1}}{(1+s^{\gamma})^2}\,{\rm d}s
\asymp \int_{\phi_d(x,y)}^{2^{1/\gamma}} \frac{s^{\gamma-1}}{V(x_1,\phi_1^{-1}(s))V(x_2,\phi_2^{-1}(s))}\,{\rm d}s\\
&=\frac{1}{\gamma}\int_{\phi_d^{\gamma}(x,y)}^2 \frac{1}{V(x_1,(\phi_1^{\gamma})^{-1}(s))V(x_2,(\phi_2^{\gamma})^{-1}(s))}\,{\rm d}s.
\end{split}
\end{equation*}
Therefore, there exists a constant $c_2>0$ such that
\begin{equation}\label{eq:j1}
J_1\leq c_2\int_{\phi_d^{\gamma}(x,y)}^2
\frac{1}{V(x_1,(\phi_1^{\gamma})^{-1}(s))V(x_2,(\phi_2^{\gamma})^{-1}(s))}\,{\rm d}s, \quad \phi_d(x,y)\le 1.
\end{equation}

To prove the upper bound of $J_2$, we assume that $\phi_1(d(x_1,y_1))\leq \phi_2(d(x_2,y_2))\leq 1$.
By \eqref{eq:prod-res-cond},
$$
\int_0^{\phi_2(d(x_2,y_2))}\frac{s^{\gamma}}{V(x_1,\phi_1^{-1}(s))}\,{\rm d}s
\leq \frac{c_3\phi_2^{\gamma+1}(d(x_2,y_2))}{V(x_1,\phi^{-1}(\phi_2(d(x_2,y_2))))}.
$$
Then, by (WUHK) and \eqref{eq:sub-upper} with the Fubini theorem again,
\begin{equation*}
\begin{split}
J_2
&\leq \frac{c_4}{V(x_2,d(x_2,y_2))\phi_2(d(x_2,y_2))}
\int_0^{\infty}e^{-t}t
\left(\int_0^{\phi_2(d(x_2,y_2))}\frac{e^{-t/s^{\gamma}}}{V(x_1,\phi_1^{-1}(s))s^{\gamma}}\,{\rm d}s\right)
\, {\rm d}t\\
&=\frac{c_4}{V(x_2,d(x_2,y_2))\phi_2(d(x_2,y_2))}
\int_0^{\phi_2(d(x_2,y_2))}\frac{s^{\gamma}}{V(x_1,\phi_1^{-1}(s))(1+s^{\gamma})^2}\,{\rm d}s\\
&\leq \frac{c_5 \phi_2^{\gamma}(d(x_2,y_2))}
{V(x_1,\phi_1^{-1}(\phi_2(d(x_2,y_2))))V(x_2,d(x_2,y_2))}.
\end{split}
\end{equation*}
A similar bound as above is valid even for $\phi_2(d(x_2,y_2))\leq \phi_1(d(x_1,y_1))\leq 1$,
and thus
$$
J_2\leq \frac{c_6 \phi_d^{\gamma}(x,y)}
{V(x_1,\phi_1^{-1}(\phi_d(x,y)))V(x_2,\phi_2^{-1}(\phi_d(x,y)))},
\quad \phi_d(x,y)\leq 1.
$$
Combining this with \eqref{eq:j1}, we have \eqref{eq:prod-res-upper-1}.
\end{proof}

Before the proof of the Green function estimates of $X^{\gamma}$,
we give a criterion for recurrence or transience.
For $\gamma\in (0,1]$, let
$$
J^{\gamma}(x)=\int_1^{\infty}\frac{1}{V(x_1,(\phi_1^{\gamma})^{-1}(t))V(x_2,(\phi_2^{\gamma})^{-1}(t))}\,{\rm d}t,
 \quad x\in M\times M.
$$
Then, by the change of variables formula with $s=t^{1/\gamma}$, we have
$$J^{\gamma}(x)=\gamma\int_1^{\infty}\frac{t^{\gamma-1}}{V(x_1,\phi_1^{-1}(t))V(x_2,\phi_2^{-1}(t))}\,{\rm d}t.$$
\begin{lem}\label{lem:trans}
Suppose that the independent processes $X^1$ and $X^2$ satisfy Assumption {\bf (H)},
and let $\gamma\in (0,1]$.
\begin{enumerate}
\item[{\rm (1)}]
If $X^1$ and $X^2$ satisfy {\rm (NDLHK)} and $J^{\gamma}(x)=\infty$ for any $x\in M\times M$,
then $X^{\gamma}$ is recurrent.
\item[{\rm (2)}]
If $X^1$ and $X^2$ satisfy {\rm (WUHK)} and $J^{\gamma}(x)<\infty$
for any $x\in M\times M$, then $X^{\gamma}$ is transient.
\end{enumerate}
\end{lem}

\begin{proof}
We first prove (1).
Suppose that the processes $X^1$ and $X^2$ satisfy {\rm (NDLHK)}
and $J^{\gamma}(x)=\infty$ for any $x\in M\times M$.
We can then follow the calculations of \eqref{eq:d-res-lower} and \eqref{eq:int-sum-1}
to show that,
for any $x,y\in M_0^1\times M_0^2$,
\begin{equation*}
\begin{split}
\int_1^{\infty} q^{\gamma}(t,x,y)\,\d t
&\geq c_1\int_{1\vee \phi_d^{\gamma}(x,y)}^{\infty}
t\left(\int_{t^{1/\gamma}}^{\infty}
\frac{1}{V(x_1,\phi_1^{-1}(s))V(x_2,\phi_2^{-1}(s))s^{1+\gamma}}\,\d s\right)\,\d t\\
&\geq c_2\int_{1\vee \phi_d^{\gamma}(x,y)}^{\infty}
\frac{1}{V(x_1,\phi_1^{-1}(t^{1/\gamma}))V(x_2,\phi_2^{-1}(t^{1/\gamma}))}\d t\\
&=c_2\int_{1\vee \phi_d^{\gamma}(x,y)}^{\infty}
\frac{1}{V(x_1,(\phi_1^{\gamma})^{-1}(t))V(x_2,(\phi_2^{\gamma})^{-1}(t))}\d t=\infty.
\end{split}
\end{equation*}
Hence by Remark \ref{rem:global} (ii),
$X^{\gamma}$ is recurrent.

We next prove (2).
Suppose that the processes $X^1$ and $X^2$ satisfy  {\rm (WUHK)} and $J^{\gamma}(x)<\infty$
for any $x\in M\times M$.
Then for any  $x\in M_0^1\times M_0^2$,
we follow the calculation as in the proof of Lemma \ref{lem:prod-res} to see that
\begin{equation*}
\begin{split}
\int_1^{\infty} \sup_{y\in M_0^1\times M_0^2}q^{\gamma}(t,x,y)\,\d t
&=\int_1^{\infty}\sup_{y\in M_0^1\times M_0^2}\left(\int_0^{\infty}p(s,x,y)\,\pi_t(s)\d s\right)\d t\\
&\leq \int_1^{\infty}\left\{\int_0^{\infty}\left(\sup_{y\in M_0^1\times M_0^2}p(s,x,y)\right)\pi_t(s)\,\d s\right\}\d t\\
&\leq c_3\int_1^{\infty}t
\left(\int_0^{\infty}\frac{e^{-t/s^{\gamma}}}{V(x_1,\phi_1^{-1}(s))V(x_2,\phi_2^{-1}(s))s^{1+\gamma}}\,\d s\right)\,\d t
=:c_3 I.
\end{split}
\end{equation*}
Then, by the Fubini theorem,
\begin{equation*}
\begin{split}
I&=\int_0^{\infty}\frac{1}{V(x_1,\phi_1^{-1}(s))V(x_2,\phi_2^{-1}(s))s^{1+\gamma}}
\left(\int_1^{\infty}e^{-t/s^{\gamma}}t\,\d t\right)\,\d s\\
&=\int_0^{\infty}\frac{e^{-1/s^{\gamma}}(s^\gamma+1)}{V(x_1,\phi_1^{-1}(s))V(x_2,\phi_2^{-1}(s))s}\,\d s\\
&\leq 2\int_0^1\frac{e^{-1/s^{\gamma}}}{V(x_1,\phi_1^{-1}(s))V(x_2,\phi_2^{-1}(s))s}\,\d s
+2\int_1^{\infty}\frac{s^{\gamma-1}}{V(x_1,\phi_1^{-1}(s))V(x_2,\phi_2^{-1}(s))}\,\d s.
\end{split}
\end{equation*}
The first term above is convergent by \eqref{eq:v-inverse} with $\phi=\phi_1$ and $\phi=\phi_2$,
and so is the second one by assumption.
Hence by Remark \ref{rem:global} (i),
$X^{\gamma}$ is transient.
\end{proof}

\begin{rem}\rm
If Assumption \ref{assum:volume}  is imposed on the volume function $V$,
then the function $J^{\gamma}(x)$ in Lemma \ref{lem:trans} is replaced by the integral
\begin{equation}\label{eq:j-int}
J^{\gamma}=\int_1^{\infty}\frac{1}{V((\phi_1^{\gamma})^{-1}(t))V((\phi_2^{\gamma})^{-1}(t))}\, \d t.
\end{equation}
In particular, by the proof of Lemma \ref{lem:trans} and Remark \ref{rem:global} (iii),
we see that if
the
independent  processes $X^1$ and $X^2$ satisfy {\rm (ODHK)} and $J^{\gamma}<\infty$,
then $X^{\gamma}$ is transient.
\end{rem}

By following the proofs of Lemmas \ref{lem:d-res-lower} and \ref{lem:prod-res},
we also get the Green function estimates.
\begin{lem}\label{lem:green-1}
Suppose that the independent processes $X^1$ and $X^2$ satisfy Assumption {\bf (H)}.
\begin{itemize}
\item[{\rm (1)}]
Let $X^1$ and $X^2$ satisfy {\rm (NDLHK)}.
Then there exists a constant $c_1>0$ such that
$$
u_0^{\gamma}(x,y)
\geq c_1\int_{\phi_d^{\gamma}(x,y)}^{\infty}
\frac{1}{V(x_1,(\phi_1^{\gamma})^{-1}(t))V(x_2,(\phi_2^{\gamma})^{-1}(t))}\,{\rm d}t,
\quad x,y\in M_0^1\times M_0^2.
$$
\item[{\rm (2)}]
Let $X^1$ and $X^2$ satisfy {\rm (WUHK)}.
If \eqref{eq:prod-res-cond} holds and $J^{\gamma}<\infty$,
then there exists a constant $c_2>0$ such that
for any $x,y\in M_0^1\times M_0^2$ with $\phi_d(x,y)\leq 1$,
\begin{equation}\label{eq:g-upper}
u_0^{\gamma}(x,y)
\leq c_2\int_{\phi_d^{\gamma}(x,y)}^{\infty}
\frac{1}{V(x_1,(\phi_1^{\gamma})^{-1}(t))V(x_2,(\phi_2^{\gamma})^{-1}(t))}\,{\rm d}t.
\end{equation}
Assume in addition that there exists a constant $c_3>0$ such that for $i=1,2$,
\begin{equation}\label{eq:prod-res-cond-1}
\int_0^T\frac{t^{\gamma}}{V(w,\phi_i^{-1}(t))}\,{\rm d}t
\leq \frac{c_3T^{1+\gamma}}{V(w,\phi_i^{-1}(T))}, \quad w\in M, \ T\in [1,\infty).
\end{equation}
Then \eqref{eq:g-upper} is valid
for any $x,y\in M_0^1\times M_0^2$ with $\phi_d(x,y)\geq 1$ as well.
\end{itemize}
\end{lem}

\begin{proof}
We prove \eqref{eq:g-upper} for any $x,y\in M_0^1\times M_0^2$ with $\phi_d(x,y)\geq 1$ only
because the rest of the assertions follows
in the same way as Lemmas \ref{lem:d-res-lower} and \ref{lem:prod-res}.

Let
\begin{equation*}
\begin{split}
u_0^{\gamma}(x,y)
&=\int_0^{\infty}
\left(\int_0^{\infty}p(s,x,y)\pi_t(s)\,{\rm d}s\right)\,{\rm d}t\\
&=\int_0^{\infty}
\left(\int_{\phi_d(x,y)}^{\infty}p(s,x,y)\pi_t(s)\,{\rm d}s\right)\,{\rm d}t
+\int_0^{\infty}
\left(\int_0^{\phi_d(x,y)}p(s,x,y)\pi_t(s)\,{\rm d}s\right)\,{\rm d}t \\
&=:J_1+J_2.
\end{split}
\end{equation*}
Then, by following \eqref{eq:j1} and the change of the variables formula with $t=s^{\gamma}$, we get
\begin{equation}\label{eq:g-j1}
\begin{split}
J_1
\leq c_1\int_{\phi_d(x,y)}^{\infty}
\frac{s^{\gamma-1}}{V(x_1,\phi_1^{-1}(s))V(x_2,\phi_2^{-1}(s))}\,{\rm d}s
=\frac{c_1}{\gamma}\int_{\phi_d^{\gamma}(x,y)}^{\infty}
\frac{1}{V(x_1,(\phi_1^{\gamma})^{-1}(s))V(x_2,(\phi_2^{\gamma})^{-1}(s))}\,{\rm d}s.
\end{split}
\end{equation}

To prove the upper bound of $J_2$,
we assume that $\phi_d(x,y)\geq 1$ and $\phi_1(d(x_1,y_1))\leq \phi_2(d(x_2,y_2))$.
By \eqref{eq:prod-res-cond-1},
$$
\int_0^{\phi_2(d(x_2,y_2))}\frac{s^{\gamma}}{V(x_1,\phi_1^{-1}(s))}\,{\rm d}s
\leq \frac{c_2\phi_2^{\gamma+1}(d(x_2,y_2))}{V(x_1,\phi^{-1}(\phi_2(d(x_2,y_2))))}.
$$
Then, by (WUHK) and \eqref{eq:sub-upper} with the Fubini theorem,
\begin{equation*}
\begin{split}
J_2
&\leq \frac{c_3}{V(x_2,d(x_2,y_2))\phi_2(d(x_2,y_2))}
\int_0^{\infty}t
\left(\int_0^{\phi_2(d(x_2,y_2))}\frac{e^{-t/s^{\gamma}}}{V(x_1,\phi_1^{-1}(s))s^{\gamma}}\,{\rm d}s\right)
\, {\rm d}t\\
&=\frac{c_3}{V(x_2,d(x_2,y_2))\phi_2(d(x_2,y_2))}
\int_0^{\phi_2(d(x_2,y_2))}\frac{s^{\gamma}}{V(x_1,\phi_1^{-1}(s))}\,{\rm d}s\\
&\leq \frac{c_4 \phi_2^{\gamma}(d(x_2,y_2))}
{V(x_1,\phi_1^{-1}(\phi_2(d(x_2,y_2))))V(x_2,d(x_2,y_2))}.
\end{split}
\end{equation*}
A similar bound as above is valid also for $\phi_2(d(x_2,y_2))\le \phi_1(d(x_1,y_1))$,
and thus
$$
J_2\leq \frac{c_5 \phi_d^{\gamma}(x,y)}
{V(x_1,\phi_1^{-1}(\phi_d(x,y)))V(x_2,\phi_2^{-1}(\phi_d(x,y)))},
\quad \phi_d(x,y)\ge 1.
$$
Combining this with \eqref{eq:g-j1}, we arrive at the desired assertion.
\end{proof}

\subsection{Hausdorff dimensions of collision time sets}\label{subsect:collision}
In this subsection, we will determine the Hausdorff dimensions
of
collision time sets
of
two independent processes $X^1$ and $X^2$ on a given set in terms of the associated scale functions.
In what follows, we impose Assumption {\rm \ref{assum:volume}} on $M$.
Define $\phi(t)=\phi_1(t)\vee \phi_2(t)$
so that $\phi^{-1}(t)=\phi_1^{-1}(t)\wedge \phi_2^{-1}(t)$.
If we let $\phi^{\gamma}(t)=\phi(t)^{\gamma}$ and $\phi_i^{\gamma}(t)=\phi_i(t)^{\gamma}$,
then $(\phi^{\gamma})^{-1}(t)=(\phi_1^{\gamma})^{-1}(t)\wedge (\phi_2^{\gamma})^{-1}(t)$.
For $s>0$,
let
$$\gamma(s)
=\inf\left\{\gamma>0 :
\int_0^1\frac{((\phi^{\gamma})^{-1}(t))^s}
{V((\phi_1^{\gamma})^{-1}(t))V((\phi_2^{\gamma})^{-1}(t))}\,{\rm d}t<\infty\right\}.$$
We also let
$$s_0=\inf\left\{s>0 : \int_0^1\frac{((\phi
)^{-1}(t))^s}
{V(\phi_1^{-1}(t))V(\phi_2^{-1}(t))t}\,{\rm d}t<\infty\right\}.$$
Then, by the proof of Lemma \ref{lem:gamma},
the function $s\mapsto \gamma(s)$ is Lipschitz continuous on $[0,\infty)$, and  $s_0$ defined above is positive; moreover, $\gamma(s)$ is
positive and strictly decreasing on $[0,s_0)$
and $\gamma(s)=0$ for $s\ge s_0$.

\begin{thm}\label{thm:d-upper}
Let Assumption {\rm \ref{assum:volume}} hold.
Suppose that  the independent processes $X^1$ and $X^2$ satisfy Assumption {\bf(H)} and {\rm (NDLHK)}.
Let $F\subset M$ be a Borel set with $s_F=\dim_{{\cal H}}(F)>0$.
Assume that for any $s\in [0,s_0)$ and $\gamma\in (0,\gamma(s))$,
there exist constants $c_1>0$ and  $T_0\in (0,1)$ such that for any $T\in (0,T_0)$,
\begin{equation}\label{eq:d-upper-1}
\int_T^1\frac{(\phi^{-1}(t))^s}
{V(\phi_1^{-1}(t))V(\phi_2^{-1}(t))}t^{\gamma-1}\,{\rm d}t
\leq \frac{c_1(\phi^{-1}(T))^s}
{V(\phi_1^{-1}(T))V(\phi_2^{-1}(T))}T^{\gamma}.
\end{equation}
\begin{enumerate}
\item[{\rm (1)}]
If $\gamma(s_F)\leq 1$, then
\begin{equation*}
\dim_{{\cal H}}\{v>0 : X_v^1=X_v^2\in F\}\leq 1-\gamma(s_F),
\quad \text{$P_x$-a.s. for any $x\in M_0^1\times M_0^2$.}
\end{equation*}
On the other hand,
if $\gamma(s_F)>1$, then $\{v>0 : X_v^1=X_v^2\in F\}=\emptyset$,
$P_x$-a.s.\ for any $x\in M_0^1\times M_0^2$.
\item[{\rm (2)}]

Suppose further
that $M$ satisfies Assumption {\rm \ref{assum:compact}}, and
that the processes
$X^1$ and $X^2$ satisfy Assumption {\bf (H)}, {\rm (NDLHK)} and {\rm (WUHK)}
with $M_0^1$ and $M_0^2$ replaced by $M$.
If $J^1=\infty$,  $0\le \gamma(s_F)<1$
and \eqref{eq:prod-res-cond} holds for any $\gamma\in (\gamma(s_F),1]$,
then
\begin{equation}\label{eq:d-upper}
\dim_{{\cal H}}\{v>0 : X_v^1=X_v^2\in F\}\ge 1-\gamma(s_F),
\quad \text{$P_x$-a.s. for any $x\in M\times M$.}
\end{equation}
\end{enumerate}
\end{thm}

\begin{proof}
Let $F\subset M$ be a Borel set with $s_F=\dim_{{\cal H}}(F)>0$.
We first prove (1).
Let us now assume that $\gamma(s_F)\leq 1$.
Without loss of generality, we may and do assume that $\gamma(s_F)>0$.
Then, by the proof of  Lemma \ref{lem:gamma},
there exists $\delta\in (0,s_F/2)$ such that $\gamma(u)>0$ for any $u\in (s_F,s_F+\delta)$.
If we fix  $u\in (s_F,s_F+\delta)$, then for any $s\in (s_F,u)$,
$\gamma(u)<\gamma(s)<\gamma(s_F)$.
Therefore, it follows
by \eqref{eq:d-upper-1}
that for any $C_0>0$, there exists $T_0\in (0,1)$ such that for any $T\in (0,T_0)$,
\begin{equation}\label{eq:d-upper-2}
C_0
\leq \int_T^1
\frac{((\phi^{\gamma(u)})^{-1}(t))^s}{V((\phi_1^{\gamma(u)})^{-1}(t))V((\phi_2^{\gamma(u)})^{-1}(t))}\,{\rm d}t
\leq \frac{c_1((\phi^{\gamma(u)})^{-1}(T))^sT}{V((\phi_1^{\gamma(u)})^{-1}(T))V((\phi_2^{\gamma(u)})^{-1}(T))}.
\end{equation}
Here $c_1$ is a positive constant depending on the choices of $u\in(s_F,s_F+\delta)$ and $s\in (s_F,u)$.
Note that for any $x,z\in M\times M$,
\begin{equation}\label{eq:p-comp}
\phi_d(x,z)=\phi_1(d(x_1,z_1))\vee\phi_2(d(x_2,z_2))\leq  (\phi_1\vee\phi_2)(d(x,z))
=\phi(d(x,z)).
\end{equation}
Hence, if $\phi^{\gamma(u)}(d(x,z))\leq T_0$,
then Lemma \ref{lem:d-res-lower} and \eqref{eq:d-upper-2} with $T=\phi_d^{\gamma(u)}(x,z)$ yield
\begin{equation}\label{eq:d-upper-3}
\begin{split}
\frac{1}{d(x,z)^s}
&\leq \frac{c_2\phi_d^{\gamma(u)}(x,z)}{V((\phi_1^{\gamma(u)})^{-1}(\phi_d(x,z)))V((\phi_2^{\gamma(u)})^{-1}(\phi_d(x,z)))}\\
&\leq c_3
\int_{\phi_d^{\gamma(u)}(x,z)}^1\frac{1}{V((\phi_1^{\gamma(u)})^{-1}(t))V((\phi_2^{\gamma(u)})^{-1}(t))}\,{\rm d}t
\leq c_4 u_1^{\gamma(u)}(x,z).
\end{split}
\end{equation}

Recall that
$$
{\rm diag}(F)=\{y=(y_1,y_2)\in M\times M : y_1=y_2\in F\}.
$$
Let $K$ be a compact subset of $M\times M$ such that $K\subset {\rm diag}(F)$.
If there exists no finite and nontrivial Borel measure on $M\times M$
compactly supported in  $K$,
then the equilibrium measure of $K$ for $X^{\gamma}$ is trivial and thus
$$
{\Capa}^{\gamma(u)}(K)=0.
$$

We now assume that
there exists a finite and nontrivial Borel measure $\nu$ on $M\times M$
such that it is compactly supported in $K$
and charges no set of zero capacity relative to $(\E^{\gamma(u)},\F^{\gamma(u)})$.
Then, by \eqref{eq:d-upper-3},
\begin{equation}\label{eq:d-pot}
\iint_{\phi^{\gamma(u)}(d(x,z))\leq T_0}\frac{1}{d(x,z)^s}\,\nu({\rm d}x)\,\nu({\rm d}z)
\leq c_4
\iint_{\phi^{\gamma(u)}(d(x,z))\leq T_0}u_1^{\gamma(u)}(x,z)\,\nu({\rm d}x)\,\nu({\rm d}z).
\end{equation}
On the other hand, by Lemma \ref{lem:d-res-lower},
there exists $c_5:=c_5(T_0)>0$ such that
for any $x, z\in K$ with $\phi^{\gamma(u)}(x,z)\geq T_0$,
we obtain $u_1^{\gamma(u)}(x,z)\geq c_5$.
Hence, by \eqref{eq:d-pot},
\begin{equation}\label{eq:d-upper-4}
I^s(\nu)=\int_K \int_K\frac{1}{d(x,z)^s}\,\nu({\rm d}x)\,\nu({\rm d}z)
\leq c_6
\int_K \int_K u_1^{\gamma(u)}(x,z)\,\nu({\rm d}x)\,\nu({\rm d}z).
\end{equation}
We also note that ${\cal H}^s(K)=0$ because
$$s>s_F=\dim_{{\cal H}}({\rm diag}(F))\ge \dim_{{\cal H}}(K).$$
Then Proposition \ref{prop:frost-1} yields $I^s(\nu)=\infty$.
Therefore, by \eqref{eq:d-upper-4},
\begin{equation}\label{eq:d-upper-5}
\int_K \int_K u_1^{\gamma(u)}(x,z)\,\nu({\rm d}x)\,\nu({\rm d}z)=\infty.
\end{equation}

Let $\nu_K^{\gamma(u)}$ be the equilibrium measure of $K$ for $X^{\gamma(u)}$.
Since this measure is of finite energy integral relative to $(\E^{\gamma(u)},\F^{\gamma(u)})$,
it charges no set of zero capacity (\cite[Theorems 2.1.5(ii) and 2.2.3]{FOT11}).
If we assume that $\nu_K^{\gamma(u)}$ is nontrivial,
then \eqref{eq:d-upper-5} with $\nu=\nu_K^{\gamma(u)}$ gives
$${\Capa}^{\gamma(u)}(K)
=\int_K \int_K u_1^{\gamma(u)}(x,z)
\,\nu_K^{\gamma(u)}({\rm d}x)\,\nu_K^{\gamma(u)}({\rm d}z)=\infty,$$
which contradicts ${\Capa}^{\gamma(u)}(K)<\infty$.
Therefore, $\nu_K^{\gamma(u)}$ is trivial and
$${\Capa}^{\gamma(u)}(K)=\nu_K^{\gamma(u)}(K)=0.$$

By the argument above, we  have ${\Capa}^{\gamma(u)}(K)=0$
for any compact subset $K$ of ${\rm diag}(F)$.
Since it follows from the regularity of the capacity (\cite[(2.1.6)]{FOT11})
that
$${\Capa}^{\gamma(u)}({\rm diag}(F))
=\sup\{{\Capa}^{\gamma(u)}(K) : \text{$K$ is a compact subset of ${\rm diag}(F)$}\}=0,$$
we obtain, by \cite[Theorem 4.2.1 (ii)]{FOT11},
\begin{equation*}
\begin{split}
0
&=P_x\otimes P^{\gamma(u)}(\text{$X_t^{\gamma(u)}\in {\rm diag}(F)$ for some $t>0$})\\
&=\int_{\Omega}\left[P^{\gamma(u)}(\text{$X_{\tau_t}(\omega)\in {\rm diag}(F)$ for some $t>0$})\right]
\,P_x({\rm d}\omega).
\end{split}
\end{equation*}
Namely, for any $x\in M_0^1\times M_0^2$, we have for $P_x$-a.s.\ $\omega\in \Omega$,
$$P^{\gamma(u)}
(\text{$\tau_t\in \{v>0 : X_v^1(\omega)=X_v^2(\omega)\in F\}$ for some $t>0$})=0.$$
Then, by \cite[Section 3]{H71} or \cite[Lemma 2.1]{H74} again,
we get
$$\dim_{{\cal H}}\{v>0 : X_v^1=X_v^2\in F\}\leq 1-\gamma(u), \quad \text{$P_x$-a.s.\ for $x\in M_0^1\times M_0^2$.}$$
Letting $s\downarrow s_F$ and then $u\downarrow s_F$
along some sequences, we arrive at \eqref{eq:d-upper}.

If $\gamma(s_F)>1$, then, by the proof of Lemma \ref{lem:gamma} again,
there exists a constant $u>s_F$ such that $\gamma(u)=1$.
Then the same argument as above yields $\Capa({\rm diag}(F))=0$ and thus
$$
P_x(\text{$X_v\in {\rm diag}(F)$ for some $v>0$})
=P_x(\text{$X_v^1=X_v^2\in F$ for some $v>0$})=0, \quad x\in M_0^1\times M_0^2.
$$

We next prove (2).
We assume that  $0\le \gamma(s_F)<1$.
Then, by Lemma \ref{lem:gamma},
there exists a constant $\varepsilon>0$ such that for all $s\in (s_F-\varepsilon, s_F)$,
$\gamma(s_F)<\gamma(s)<1$.
We now fix such $s\in (s_F-\varepsilon,s_F)$.
Since $s_F=\dim_{{\cal H}}({\rm diag}(F))$,
the regularity of the Hausdorff measure
yields ${\cal H}^s(K)>0$ for some compact subset $K$ of ${\rm diag}(F)$.
We also note that any closed ball in $M\times M$ is compact by Assumption \ref{assum:compact}.
Hence, as a consequence of Proposition \ref{prop:frost-2},
there exists a finite and nontrivial Borel measure $\nu_K^s$ on $M\times M$ such that
$\supp[\nu_K^s]\subset K$ and $I^s(\nu_K^s)<\infty$.

On the other hand,
by the proof of Lemma \ref{lem:gamma} again, we have $\gamma(s)<\gamma(v)<1$
for any $v\in (s_F-\varepsilon, s)$.
Then, for any $v\in (s_F-\varepsilon, s)$ and  $T\in (0,1/2)$,
\begin{equation*}
\begin{split}
\infty>\int_0^1\frac{(\phi^{-1}(t))^s}{V(\phi_1^{-1}(t))V(\phi_2^{-1}(t))}t^{\gamma(v)-1}\,{\rm d}t
&\ge \int_T^{2T}\frac{(\phi^{-1}(t))^s}{V(\phi_1^{-1}(t))V(\phi_2^{-1}(t))}t^{\gamma(v)-1}\,{\rm d}t\\
&\ge \frac{c_1 (\phi^{-1}(T))^s T^{\gamma(v)}}{V(\phi_1^{-1}(T))V(\phi_2^{-1}(T))},
\end{split}
\end{equation*}
which implies that for some $c_2>0$,
\begin{equation}\label{eq:upper-1}
\frac{\phi^{\gamma(v)}(x,y)}{V(\phi_1^{-1}(\phi(x,y)))V(\phi_2^{-1}(\phi(x,y)))}
\leq\frac{c_2}{d(x,y)^s}, \quad x,y\in K.
\end{equation}

Let $X^{\gamma(v)}$ be the $\gamma(v)$-stable subordinate process of the process $X$.
Since $\gamma(v)<\gamma(0)$, it follows by
Lemma \ref{lem:prod-res} and \eqref{eq:d-upper-1} with \eqref{eq:prod-res-cond} that
under {\rm(WUHK)} pointwisely,
there exists a constant $c_3>0$ such that for all $x,y\in K$,
\begin{equation}\label{eq:r-upper}
\frac{\phi^{\gamma(v)}(x,y)}{V(\phi_1^{-1}(\phi(x,y)))V(\phi_2^{-1}(\phi(x,y)))}
\geq  c_3u_1^{\gamma(v)}(x,y).
\end{equation}
Here we note that $x_1=x_2$ and $y_1=y_2$
and thus
$$\phi_d(x,y)=\phi(d(x_1,y_1))=\phi(d(x_2,y_2))\simeq \phi(d(x,y)).$$
Combining \eqref{eq:r-upper}
with \eqref{eq:upper-1}, we have for some $c_4>0$,
\begin{equation}\label{eq:res-dist}
\frac{1}{d(x,y)^s}\geq c_4 u_1^{\gamma(v)}(x,y),\quad  x,y\in K.
\end{equation}
Therefore,
$$\infty>I^s(\nu_K^s)
=\iint_{K\times K}\frac{1}{d(x,y)^s}\,\nu_K^s({\rm d}x)\nu_K^s({\rm d}y)
\geq c_4
\iint_{K\times K}u_1^{\gamma(v)}(x,y)
\,\nu_K^s({\rm d}x)\nu_K^s({\rm d}y).$$
Note that the last integral above is well-defined
because $u_1^{\gamma}(x,y)$ is defined for any $x,y\in M$ by assumption.
Then, by \cite[Exercise 4.2.2]{FOT11}, the measure $\nu_K^s$ is of finite energy integral
relative to $X^{\gamma(v)}$.
Moreover, since $\nu_K^{s}$ is nontrivial,
${\Capa}^{\gamma(v)}({\rm diag}(F))\geq \Capa^{\gamma(v)}(K)>0$ thanks to \cite[Theorem 2.2.3]{FOT11}.
In particular, for all $s\in (s_F-\varepsilon, s_F)$, ${\Capa}^{\gamma(s)}({\rm diag}(F))>0$.

We now follow the argument of \cite[Theorem 1]{L95} again.
Let $\sigma_{{\rm diag}(F)}$ be the hitting time of $X^{\gamma(s)}$ to ${\rm diag}(F)$, i.e.,
$\sigma_{{\rm diag}(F)}=\inf\{t>0 :X_t^{\gamma(s)}\in {\rm diag}(F)\}$.
Define
$$({\rm diag}(F))_{\gamma(s)}=\left\{x\in M\times M : P_x\otimes P^{\gamma(s)}(\sigma_{{\rm diag}(F)}<\infty)=1\right\}$$
and
$$({\rm diag}(F))_{\gamma(s)}^r=\left\{x\in M\times M: P_x\otimes P^{\gamma(s)}(\sigma_{{\rm diag}(F)}=0)=1\right\}.$$
Then, by following the proof of Theorem \ref{thm:upper}(2),
we get
$${\Capa}^{\gamma(s)}(({\rm diag}(F))_{\gamma(s)})
\geq {\Capa}^{\gamma(s)}(({\rm diag}(F))_{\gamma(s)}^r)>0.$$
In particular, since $J^1=\infty$ and {\rm (NDLHK)} hold,
$X$ is recurrent.
We can then have an inequality corresponding to \eqref{eq:cpt}
with some compact set $K\subset ({\rm diag}(F))_{\gamma(s)}$.
We further follow the proof of Theorem \ref{thm:upper}(2) to obtain
$$\dim\left\{t>0: X_t\in F\right\}\geq 1-\gamma(s_F),
\quad \text{$P_x$-a.s.\  for any $x\in M\times M$.}$$
The proof is complete.
\end{proof}

\begin{cor}
Keep the same condition in Theorem {\rm \ref{thm:d-upper}}.
Suppose that the independent processes $X^1$ and $X^2$ satisfy {\rm (WUHK)} and {\rm (HR)}.
 If $J^1=\infty$, $\gamma(s_F)\in [0,1]$
and \eqref{eq:prod-res-cond} hold, then
$$\dim_{{\cal H}}\{v>0 : X_v^1=X_v^2\in F\}=1-\gamma(s_F),
\quad \text{$P_x$-a.s.\ for any $x\in M\times M$.}$$
\end{cor}

The assumption in Theorem \ref{thm:d-upper} (2) implies that the process $X$ is recurrent.
For instance, the assumption is fulfilled by a class of
$\alpha$-stable-like symmetric jump processes
on
ultra-metric spaces with any $\alpha>0$ (see, e.g., \cite{BGHH21, G22} for details).
On the other hand, it is natural to allow $X$ to be transient.
For instance, if $X^1$ and $X^2$ are independent symmetric stable process on ${\mathbb R}$
with index $\alpha\in (1,2)$,
then their direct product process is transient.
Here we utilize two types of the Wiener tests in Propositions \ref{prop:rec-set} and \ref{prop:reg-set}.
The price is to assume that the collision place $F$ is closed,
and  to make the next assumption on $M$ in addition to Assumption \ref{assum:compact}.
\begin{assum}\label{assum:comp}
$M$ is connected.
\end{assum}

Note that under Assumption \ref{assum:comp}, $M\times M$ is also connected.

\begin{thm}\label{thm:d-lower}
Suppose that Assumptions {\rm \ref{assum:volume}}, {\rm \ref{assum:compact}}
and {\rm \ref{assum:comp}} hold.
Let the processes $X^1$ and $X^2$ satisfy Assumption  {\bf(H)}, {\rm (WUHK)} and {\rm(HR)}, so that $X^1$ and $X^2$ are independent.
Let $F\subset M$ be an $(s_F,t_F)$-set for some $s_F\in (0,s_0)$ and $t_F>0$
with $\gamma(s_F)<1$.
Assume the following conditions on $X^1$ and $X^2${\rm :}
\begin{itemize}
\item For any $\gamma\in (\gamma(s_F),1]$,
\eqref{eq:prod-res-cond} holds, and there exists a constant $c_1>0$ such that
for any
$T\in (0,1/2)$,
\begin{equation}\label{eq:reg-set-0}
\int_0^T\frac{((\phi^{\gamma})^{-1}(t))^{s_F}}{V((\phi_1^{\gamma})^{-1}(t))V((\phi_2^{\gamma})^{-1}(t))}\,{\rm d}t
\leq c_1((\phi^{\gamma})^{-1}(T))^{s_F}
\int_T^1\frac{1}{V((\phi_1^{\gamma})^{-1}(t))V((\phi_2^{\gamma})^{-1}(t))}\,{\rm d}t.
\end{equation}
\item $J^1<\infty$, and \eqref{eq:prod-res-cond-1} holds with $\gamma=1$.
Furthermore, there exists a constant $c_2>0$ such that for any
$T\geq 1$,
\begin{equation}\label{eq:rec-set-0}
\int_1^T\frac{(\phi^{-1}(t))^{t_F}}{V(\phi_1^{-1}(t))V(\phi_2^{-1}(t))}\,{\rm d}t
\leq c_2(\phi^{-1}(T))^{t_F}
\int_T^{\infty}\frac{1}{V(\phi_1^{-1}(t))V(\phi_2^{-1}(t))}\,{\rm d}t.
\end{equation}
\end{itemize}
Then
\begin{equation}\label{eq:d-lower}
\dim_{{\cal H}}\{v>0 : X_v^1=X_v^2\}\geq 1-\gamma(s_F), \quad \text{$P_x$-a.s. for any $x\in M\times M$.}
\end{equation}
\end{thm}

\begin{proof}
For $\gamma\in (0,1]$, we use the same notations $\sigma_{{\rm diag}(F)}$,
$({\rm diag}(F))_{\gamma(s)}$ and
$({\rm diag}(F))_{\gamma(s)}^r$ as in the proof of Theorem \ref{thm:d-upper} (2).
For any $s\in (0,s_F)$ with $\gamma(s_F)<\gamma(s)<1$,
$$\int_0^1\frac{1}{V((\phi_1^{\gamma(s)})^{-1}(t))V((\phi_2^{\gamma(s)})^{-1}(t))}\,\d t=\infty,$$
and $J^{\gamma(s)}\leq J^1<\infty$ by assumption.
Since \eqref{eq:prod-res-cond} and \eqref{eq:reg-set-0} are also valid by assumption,
we apply Proposition \ref{prop:reg-set} for $X^{\gamma(s)}$ and thus
$${\rm diag}(F)=({\rm diag}(F))_{\gamma(s)}^r\subset ({\rm diag}(F))_{\gamma(s)}.$$
Then, for any $y\in {\rm diag}(F)$,
\begin{equation}\label{eq:reg}
1=P_y^{\gamma(s)}(\sigma_{{\rm diag}(F)}<\infty)
=E_y\left[P^{{\gamma(s)}}
\left(\text{$\tau_t\in \left\{v>0 : X_v^1=X_v^2\in F\right\}$ for some $t>0$}\right)\right].
\end{equation}

Note that $\eqref{eq:prod-res-cond}$ with $\gamma=1$ is valid
by assumption.
Since \eqref{eq:prod-res-cond-1} and \eqref{eq:rec-set-0} are 
also valid by assumption,
we apply Proposition \ref{prop:rec-set} with $\gamma=1$
to show that $P_x(\sigma_{{\rm diag}(F)}<\infty)=1$ for any $x\in M\times M$.
We also see that $X_{{\rm diag}(F)}\in {\rm diag}(F)$ because
${\rm diag}(F)$ is closed.
Therefore, by \eqref{eq:reg} and the strong Markov property of the process $X$,
we obtain for any $x\in M\times M$,
\begin{equation*}
\begin{split}
1&=P_x(\sigma_{{\rm diag}(F)}<\infty)\\
&=E_x\left[E_{X_{\sigma_{{\rm diag}(F)}}}\left[P^{{\gamma(s)}}
\left(\text{$\tau_t\in \left\{v>0 : X_v^1=X_v^2\in F \right\}$ for some $t>0$}\right)
\right];\sigma_{{\rm diag}(F)}<\infty\right]\\
&=E_x\left[P^{{\gamma(s)}}
\left(\text{$\tau_t\in
\left\{v>0 : X_v^1\circ\theta_{\sigma_{{\rm diag}(F)}}=X_v^2\circ\theta_{\sigma_{{\rm diag}(F)}}\in F \right\}$
for some $t>0$}\right)\right]\\
&\leq E_x\left[P^{{\gamma(s)}}
\left(\text{$\tau_t\in \left\{v>0 : X_v^1=X_v^2\in F \right\}$ for some $t>0$}\right)\right].
\end{split}
\end{equation*}
Then, by \cite[Section 3]{H71}  and \cite[Lemma 2.1]{H74},
$${\rm dim}_{{\cal H}}\{v>0: X_v^1=X_v^2\in F\}\geq 1-\gamma(s), \quad \text{$P_x$-a.s.\ for any $x\in M\times M$.}$$
Letting $s\uparrow s_F$
along a sequence, we get
$${\rm dim}_{{\cal H}}\{v>0: X_v^1=X_v^2\in F\}\geq 1-\gamma(s_F), \quad \text{$P_x$-a.s.\ for any $x\in M\times M$.}$$
The proof is complete.
\end{proof}

Recall that {\rm (WUHK)} and {\rm (HR)} implies {\rm (NDLHK)} (Remark \ref{rem:heat-kernel}).
Then, by Theorems \ref{thm:d-upper} (1) and \ref{thm:d-lower}, we have
\begin{cor}
Under the full conditions of Theorems {\rm \ref{thm:d-upper}} {\rm (1)} and {\rm \ref{thm:d-lower}},
$$
{\rm dim}_{{\cal H}}\{v>0: X_v^1=X_v^2\in F\}=1-\gamma(s_F), \quad \text{$P_x$-a.s. for any $x\in M\times M$.}
$$
\end{cor}

\begin{exam}\label{exam3}
Let $M$ satisfy Assumptions {\rm \ref{assum:compact}} and {\rm \ref{assum:comp}}.
Suppose that the independent processes $X^1$ and $X^2$ satisfy Assumption {\bf(H)},  {\rm (WUHK)} and {\rm(HR)}.
We impose the following conditions on $V(x,r)$ and $\phi_j(r) \ (j=1,2)${\rm :}
\begin{itemize}
\item
There exist positive constants $d_1$, $d_2$ and $c_i \ (1\leq i\leq 4)$ such that
$$c_1r^{d_1}\leq V(x,r)\leq c_2r^{d_1}, \quad x\in M, \ r\in (0,1)$$
and
$$c_3r^{d_2}\leq V(x,r)\leq c_4r^{d_2}, \quad x\in M, \ r\in [1,\infty).$$
\item
There exist positive constants $\alpha_{11}$, $\alpha_{21}$, $\alpha_{12}$, $\alpha_{22}$ and  $c_i \ (5\leq i\leq 8)$ such that
$$c_5r^{\alpha_{j1}}\leq \phi_j(r)\leq c_6r^{\alpha_{j1}}, \quad r\in (0,1)$$
and
$$c_7r^{\alpha_{j2}}\leq \phi_j(r)\leq c_8r^{\alpha_{j2}}, \quad r\in [1,\infty).$$
\end{itemize}

For simplicity, we assume that $\alpha_{11}\leq \alpha_{21}$ and $\alpha_{12}\leq \alpha_{22}$.
Then, by calculations, we have
$$s_0=\alpha_{11}\left(\frac{d_1}{\alpha_{11}}+\frac{d_1}{\alpha_{21}}\right),
\quad \gamma(s)=\frac{d_1-s}{\alpha_{11}}+\frac{d_1}{\alpha_{21}}, \ 0\leq s\le s_0$$
so that
$$0<\gamma(s)<1 \iff s_0-\alpha_{11}<s<s_0.$$
In particular, \eqref{eq:d-upper-1} holds for any $s\in [0,s_0)$ and $\gamma\in (0,\gamma(s))$.
\begin{enumerate}
\item[{\rm(i)}]
Let $F\subset M$ be a Borel subset with $s_F=\dim_{{\cal H}}(F)>0$.
\begin{itemize}
\item
Assume that $d_1<\alpha_{21}$. Then $s_0-\alpha_{11}<d_1$,
and $0<\gamma(s_F)<1$ for any $s_F\in (s_0-\alpha_{11},d_1]$.
In particular, if $d_1<2\alpha_{11}$,
then $s_0-\alpha_{11}<\alpha_{11}(1+d_1/\alpha_{21})$,
and $\gamma(s_F)\ge (d_1/\alpha_{11})-1$  for any $s_F\in (s_0-\alpha_{11},\alpha_{11}(1+d_1/\alpha_{21})]$.
Hence we see that if
$d_1<(2\alpha_{11})\wedge \alpha_{21}$ and
\begin{equation}\label{eq:s-f}
s_0-\alpha_{11}<s_F
\leq d_1\wedge \left\{\alpha_{11}\left(1+\frac{d_1}{\alpha_{21}}\right)\right\},
\end{equation}
then \eqref{eq:prod-res-cond} holds
for any $\gamma\in (\gamma(s_F),1]$.
\item
$J^1=\infty$ if and only if
\begin{equation}\label{eq:t-f-1}
\frac{d_2}{\alpha_{12}}+\frac{d_2}{\alpha_{22}}\le 1.
\end{equation}
\end{itemize}
By the calculations above, we have the following{\rm :}
Suppose that $0<d_1<(2\alpha_{11})\wedge \alpha_{21}$,
and \eqref{eq:t-f-1} 
holds.
If $F$ satisfies  \eqref{eq:s-f}, then
\begin{equation}\label{eq:exam-1}
{\rm dim}_{\cal H}\{s>0 : X_s^1=X_s^2\in F\}
=1-\left(\frac{d_1-s_F}{\alpha_{11}}+\frac{d_1}{\alpha_{21}}\right),
\quad \text{$P_x$-a.s.\ for any $x\in M\times M$.}
\end{equation}
We now assume in addition that
$d_1=d_2=d$,
$\alpha_{11}=\alpha_{21}=\alpha$ and $\alpha_{12}=\alpha_{22}=\beta$.
Under this condition, if $d<\alpha$, $d\leq \beta/2$  and $2d-\alpha<s_F\leq d$ hold,
then
\begin{equation}\label{eq:exam-2}
{\rm dim}_{\cal H}\{s>0 : X_s^1=X_s^2\in F\}
=1-\frac{2d-s_F}{\alpha},
\quad \text{$P_x$-a.s.\ for any $x\in M\times M$.}
\end{equation}
In particular, since $s_M=d$, we have
$$
{\rm dim}_{\cal H}\{s>0 : X_s^1=X_s^2\}
=1-\frac{d}{\alpha},
\quad \text{$P_x$-a.s.\ for any $x\in M\times M$.}
$$

\item[{\rm(ii)}]
Let $F\subset M$ be an $(s_F,t_F)$-set with some constants $s_F\in (0,s_0)$ and $t_F>0$
so that $\dim_{\cal H}(F)=s_F$.
Then, by calculations, we can see that
\begin{itemize}
\item
If $d_1<(2\alpha_{11})\wedge \alpha_{21}$
and \eqref{eq:s-f} hold,
then \eqref{eq:prod-res-cond} is fulfilled
for any $\gamma\in (\gamma(s_F),1]$.
Under the current setting, \eqref{eq:reg-set-0} also holds for any $\gamma\in (\gamma(s_F),\gamma(0))$.
Therefore, if
\begin{equation}\label{eq:t-f-0}
\frac{d_1}{\alpha_{11}}+\frac{d_1}{\alpha_{21}}>1,
\end{equation}then \eqref{eq:reg-set-0} is true for any $\gamma\in (\gamma(s_F),1]$.
\item
$J^1<\infty$ if and only if
\begin{equation}\label{eq:t-f-a}
\frac{d_2}{\alpha_{12}}+\frac{d_2}{\alpha_{22}}>1.
\end{equation}
If we assume in addition that $d_2<\alpha_{12}$, then
$$d_2>\alpha_{22}\left(\frac{d_2}{\alpha_{12}}+\frac{d_2}{\alpha_{22}}-1\right).$$
Therefore, \eqref{eq:rec-set-0} with $\gamma=1$ holds if and only if
\begin{equation}\label{eq:t-f-b}
\alpha_{22}\left(\frac{d_2}{\alpha_{12}}+\frac{d_2}{\alpha_{22}}-1\right)<t_F\leq d_2.
\end{equation}
\end{itemize}

Hence we have the following{\rm :}
Suppose that $d_1<(2\alpha_{11})\wedge \alpha_{21}$, $d_2<\alpha_{12}$, \eqref{eq:t-f-0}
and \eqref{eq:t-f-a} hold.
If $F$ satisfies  \eqref{eq:s-f} and \eqref{eq:t-f-b},
then \eqref{eq:exam-1} holds.
We assume in addition that
$d_1=d_2=d$,
$\alpha_{11}=\alpha_{21}=\alpha$ and $\alpha_{12}=\alpha_{22}=\beta$.
Under this condition, if
$(\alpha\vee \beta)/2<d<\alpha\wedge \beta$,
$2d-\alpha<s_F\leq d$ and $2d-\beta<t_F\leq d$,
then \eqref{eq:exam-2} holds.
\end{enumerate}
\end{exam}

It immediately follows from Example \ref{exam3}
and Remark  \ref{rem2:heat-kernel}
that
the second assertion (2)
in Theorem \ref{Thm:main} holds.

\appendix
\section{Hausdorff measure and dimension}
\subsection{Frostman lemma}\label{subsect:frostman}
Here we follow the arguments in \cite{F14,H01,H95}
to give a proof of the Frostman lemma
on the complete separable metric space.

\begin{defn}\label{def:hausdorff}
\begin{enumerate}
\item[{\rm (1)}]
A function $\varphi: [0,\infty) \to {\mathbb R}$ is called Hausdorff, if
the following three conditions are satisfied.
\begin{description}
\item[{\rm (i)}] $\varphi(t)>0$ for any $t>0$.
\item[{\rm (ii)}] If $t\geq s>0$, then $\varphi(t)\ge \varphi(s)$.
\item[{\rm (iii)}] $\varphi$ is right continuous.
\end{description}
\item[{\rm (2)}] A Hausdorff function $\varphi$ is of finite order,
if there exists a constant $\eta>0$ such that
$$\limsup_{t\rightarrow 0}\frac{\varphi(3t)}{\varphi(t)}\leq \eta.$$
\end{enumerate}
\end{defn}

Let $(M,d)$ be a complete separable metric space.
Let $\varphi$ be a continuous Hausdorff function of finite order such that $\varphi(0)=0$.
For any subset $F$ of $M$ and $\delta>0$, define
$${\cal H}_{\delta}^{\varphi}(F)
=\inf\left\{\sum_{n=1}^{\infty}\varphi({\rm diam}(U_n))
:
F\subset \bigcup_{n=1}^{\infty}U_n, \, U_n\subset M \text{ and } {\rm diam}(U_n)\leq \delta \hbox{ for all }n\geq 1 \right\}$$
and
$${\cal H}^{\varphi}(F):=\lim_{\delta\rightarrow 0}{\cal H}_{\delta}^{\varphi}(F).$$
Here ${\rm diam} (A)=\sup\left\{d(x,y) : x,y\in A\right\}$ for $A\subset M$.
Then, by \cite[Notes 4--6, 9]{H95},
${\cal H}^{\varphi}$ is an outer measure on $M$,
such that any Borel subset $B\subset M$ is measurable with respect to ${\cal H}^{\varphi}$ and
\begin{equation}\label{eq:regular}
\begin{split}
{\cal H}^{\varphi}(B)
&=\inf\{{\cal H^{\varphi}}(G): \text{$G$ is an open subset of $M$ and $G\supset B$}\}\\
&=\sup\{{\cal H^{\varphi}}(K): \text{$K$ is a compact subset of $M$ and $K\subset B$}\}.
\end{split}
\end{equation}

In the rest of this part, we always assume that
$\varphi$ is a continuous Hausdorff function of finite order such that $\varphi(0)=0$.
For any Borel measure $\nu$ on $M$,
define the $\varphi$-energy of $\nu$ as
$$I^{\varphi}(\nu)=\iint_{M\times M}\frac{1}{\varphi(d(x,y))}\,\nu({\rm d}x)\,\nu({\rm d}y).$$
For $x\in M$ and $r>0$, let $B(x,r)$ denote the closed ball
with radius $r$ centered at $x$, i.e., $B(x,r)=\{y\in M : d(x,y)\leq r\}$.

We first present a condition for the Hausdorff measure of a Borel set
being infinite in terms of the $\varphi$-energy.
\begin{prop}\label{prop:frost-1}
Let $F$ be a Borel subset of $M$.
If there exists a finite and nontrivial Borel measure $\nu$ on $M$ such that
$\supp[\nu]\subset F$ and $I^{\varphi}(\nu)<\infty$, then ${\cal H}^{\varphi}(F)=\infty$.
\end{prop}

To obtain Proposition \ref{prop:frost-1},
we follow the proof of \cite[Proposition 4.9]{F14} to show

\begin{lem}\label{lem:limsup}
Let $\nu$ be a Borel measure on $M$.
Suppose that for some $F\in {\cal B}(M)$ and $c>0$,
\begin{equation}\label{eq:limsup}
\limsup_{r\rightarrow 0}\frac{\nu(B(x,r))}{\varphi(r)}\leq c, \quad x\in F.
\end{equation}
Then $\nu(F)\leq c{\cal H}^{\varphi}(F)$.
\end{lem}

\begin{proof}
Let $\nu$ be a Borel measure on $M$, and let $\nu^*$ be the associated outer measure.
Then any Borel subset $B\subset M$ is $\nu^*$-measurable and $\nu^*(B)=\nu(B)$.
Suppose that  \eqref{eq:limsup} holds for some $F\in {\cal B}(M)$ and $c>0$.
For $n,m\in {\mathbb N}$, define a Borel subset
$$F_{n,m}=\left\{x\in
F: \text{$\nu(B(x,r))\leq \left(c+\frac{1}{n}\right)\varphi(r)$
for any $r\in \left(0,\frac{1}{m}\right]$}\right\},$$
so that $F=\cap_{n=1}^{\infty}\cup_{m=1}^{\infty}F_{n,m}$.
For $m\in {\mathbb N}$, let $\{U_k\}_{k=1}^{\infty}$ be a $(1/m)$-covering of $F$.
Namely, $\{U_k\}_{k=1}^{\infty}$  is a sequence of
subsets of $M$ such that
$$F\subset \bigcup_{k=1}^{\infty}U_k,
\quad {\rm diam}(U_k)\leq \frac{1}{m}, \quad  k=1,2,3,\dots$$
If $F_{n,m}\cap U_k\ne \emptyset$,
then, for any $x\in F_{n,m}\cap U_k$,
$$
\nu^*(F_{n,m}\cap U_k)\leq \nu^*(U_k)\leq \nu^*(B(x,{\rm diam}(U_k)))
=\nu(B(x,{\rm diam}(U_k)))
\leq \left(c+\frac{1}{n}\right)\varphi\left({\rm diam}(U_k)\right).
$$
Therefore,
$$\nu(F_{n,m})=\nu(F_{n,m}\cap F)
=\nu^*(F_{n,m}\cap F)
\leq \sum_{k=1}^{\infty}\nu^*(F_{n,m}\cap U_k)
\leq \left(c+\frac{1}{n}\right)\sum_{k=1}^{\infty}\varphi\left({\rm diam}(U_k)\right).$$
Since the covering $\{U_k\}_{k=1}^{\infty}$ is taken arbitrary,
we have for any $n,m\in {\mathbb N}$,
$$\nu(F_{n,m})\leq \left(c+\frac{1}{n}\right){\cal H}_{1/m}^{\varphi}(F).$$
Letting $m\rightarrow \infty$ and then $n\rightarrow\infty$, we obtain
$\nu(F)\leq c{\cal H}^{\varphi}(F)$.
\end{proof}

\begin{proof}[Proof of Proposition {\rm \ref{prop:frost-1}}]
Let $F$ be a Borel subset of $M$.
Suppose that there exists a finite and nontrivial Borel measure $\nu$ on $M$ such that
$\supp[\nu]\subset F$ and $I^{\varphi}(\nu)<\infty$.
Let
$$F_1=\left\{x\in F: \limsup_{r\rightarrow 0}\frac{\nu(B(x,r))}{\varphi(r)}>0\right\}.$$
Then, for any $x\in F_1$, there exist $\varepsilon>0$ and
a decreasing positive sequence $\{r_n\}_{n=1}^{\infty}$
such that $r_n\downarrow 0$ as $n\rightarrow\infty$ and
$$\frac{\nu(B(x,r_n))}{\varphi(r_n)}\geq \varepsilon, \quad n=1,2,3,\dots$$
We also have $\nu(\{a\})=0$ for any $a\in M$ because $I^{\varphi}(\nu)<\infty$.
Hence, for each $r_n$, there exists $q_n\in (0,r_n)$ such that
\begin{equation}\label{eq:annuli}
\nu(B(x,r_n)\setminus B(x,q_n))\geq \frac{1}{4}\varepsilon\varphi(r_n).
\end{equation}
Moreover, we may and do assume that $q_n>r_{n+1}$ for all $n\ge1$
by taking subsequences of $\{r_n\}_{n=1}^{\infty}$ and $\{q_n\}_{n=1}^{\infty}$ respectively,
if necessary.
Under this assumption,
the annuli $B(x,r_n)\setminus B(x,q_n)$, $n=1,2,3,\dots$, are disjoint.

For any $x\in F_1$, it follows by \eqref{eq:annuli} that
$$\int_{B(x,r_n)\setminus B(x,q_n)}\frac{1}{\varphi(d(x,y))}\,\nu({\rm d}y)
\geq \frac{1}{\varphi(r_n)}\nu(B(x,r_n)\setminus B(x,q_n))\geq \frac{\varepsilon}{4},$$
and thus
$$
\int_M\frac{1}{\varphi(d(x,y))}\,\nu({\rm d}y)
\geq \sum_{n=1}^{\infty}\int_{B(x,r_n)\setminus B(x,q_n)}\frac{1}{\varphi(d(x,y))}\,\nu({\rm d}y)
=\infty.
$$
Since $I^{\varphi}(\nu)<\infty$ by assumption, we get $\nu(F_1)=0$.

For any $x\in F\setminus F_1$,
$$\lim_{r\rightarrow 0}\frac{\nu(B(x,r))}{\varphi(r)}=0.$$
Since $\nu(F_1)=0$, Lemma \ref{lem:limsup} implies that for any $c>0$,
$${\cal H}^{\varphi}(F)\geq {\cal H}^{\varphi}(F\setminus F_1)
\geq \frac{1}{c}\nu(F\setminus F_1)=\frac{1}{c}\nu(F).$$
Letting $c\rightarrow 0$, we have ${\cal H}^{\varphi}(F)=\infty$.
\end{proof}

In the following, we present a criterion for a Borel set
to be of zero Hausdorff measure in terms of the potential.
\begin{prop}\label{prop:frost-2}
Let $M$ satisfy Assumption {\rm \ref{assum:compact}}, and let
$F$ be a Borel subset of $M$ such that ${\cal H}^{\varphi}(F)>0$.
Then for any $\varepsilon\in (0,1)$, there exists a finite and nontrivial
Borel measure $\nu$ on $M$ such that $\supp[\nu]\subset F$ and
$I^{\varphi^{\varepsilon}}(\nu)<\infty$.
\end{prop}

The proof of Proposition \ref{prop:frost-2} needs three  lemmas.
The first two lemmas concern the upper bound of the Hausdorff measure.

\begin{lem}\label{lem:limsup-1}
Let $M$ satisfy Assumption {\rm \ref{assum:compact}}, and let
$\nu$ be a finite and nontrivial Borel measure on $M$.
If $A$ is a Borel subset of $M$ with $A\subset \supp[\nu]$, and
if $c$ is a  positive constant such that
\begin{equation}\label{eq:limsup-1}
\limsup_{r\rightarrow 0}\frac{\nu(B(x,r))}{\varphi(r)}>c, \quad x\in A,
\end{equation}
then ${\cal H}^{\varphi}(A)\leq (c_*/c)\nu(M)$.
Here $c_*$ is a positive constant that is independent
of the choices of $A$ and $c$.
\end{lem}

\begin{proof}
Suppose that \eqref{eq:limsup-1} holds for some $A\in {\cal B}(M)$ and $c>0$.
For $\delta>0$, let
$${\cal C}_{\delta}=\left\{B(x,r): x\in A, r\in (0,\delta], \nu(B(x,r))>c\varphi(r)\right\}.$$
Then for any $x\in A$, there exists $r_0\in (0,\delta]$ such that
$\nu(B(x,r_0))>c\varphi(r_0)$.
This yields $B(x,r_0)\in {\cal C}_{\delta}$ and thus
$A\subset \bigcup_{B\in {\cal C}_{\delta}}B$.
Moreover, since $\sup\{{\rm diam}(B) : B\in {\cal C}_{\delta}\}\leq 2\delta$
and $M$ satisfies Assumption \ref{assum:compact},
the covering lemma (see, e.g., \cite[Theorem 1.2]{H01})
implies that there exists a sequence of countable
disjoint sets
$\{B_n\}_{n=1}^{\infty}\subset {\cal C}_{\delta}$ such that
$\bigcup_{B\in {\cal C}_{\delta}}B\subset \bigcup_{n=1}^{\infty}5B_n$, where $5B(x,r)=B(x,5r)$.
Therefore,
\begin{equation}\label{eq:10-upper}
{\cal H}_{10\delta}^{\varphi}(A)\leq \sum_{n=1}^{\infty}\varphi({\rm diam}(5B_n))
=\sum_{n=1}^{\infty}\varphi(5\, {\rm diam}(B_n)).
\end{equation}

Since the Hausdorff function $\varphi$ is of finite order and $B_n\in {\cal C}_{\delta}$,
there exists $c_*>0$, which depends only on $\varphi$, such that for any $n\in {\mathbb N}$,
$$\varphi(5\, {\rm diam}(B_n))\leq c_*\varphi({\rm diam}(B_n)
)\leq \frac{c_*}{c}\nu(B_n).$$
Combining this with \eqref{eq:10-upper}
and noting that the sequence $\{B_n\}_{n=1}^{\infty}$ is disjoint,
we obtain
$${\cal H}_{10\delta}^{\varphi}(A)\leq
\frac{c_*}{c}\sum_{n=1}^{\infty}\nu(B_n)\leq \frac{c_*}{c}\nu(M).$$
Letting $\delta\downarrow 0$, we get ${\cal H}^{\varphi}(A)\leq (c_*/c)\nu(M)$.
\end{proof}

We refer to the next key lemma for the regularity of the Hausdorff measure.
\begin{lem}\label{lem:how}{\rm (\cite[Corollary 7]{H95})}
If $F$ is  Borel subset of $M$ such that ${\cal H}^{\varphi}(F)>0$,
then there exists a compact subset $K$ of $ F$ such that $0<{\cal H}^{\varphi}(K)<\infty$.
\end{lem}

\begin{lem}\label{lem:bound}
Let $M$ satisfy Assumption {\rm \ref{assum:compact}}, and let
$F$ be a Borel subset of $M$ such that ${\cal H}^{\varphi}(F)>0$.
Then there exist a constant $b>0$ and a compact subset $K$ of $F$ such that
${\cal H}^{\varphi}(K)>0$ and
$${\cal H}^{\varphi}(B(x,r)\cap K)\leq b\varphi(r), \quad x\in K, \ r>0.$$
\end{lem}

\begin{proof}
Let $F$ be a Borel subset of $M$ such that ${\cal H}^{\varphi}(F)>0$.
Then, by Lemma \ref{lem:how}, there exists a compact subset $E$ of $F$
such that  $0<{\cal H}^{\varphi}(E)<\infty$.
Hence, if we define  $\nu(A)={\cal H}^{\varphi}(A\cap E)$ for $A\in {\cal B}(M)$,
then $\nu$ is a finite and nontrivial Borel measure on $M$ such that $\supp[\nu]=E$.

Let $c_*>0$ be the same constant as in Lemma \ref{lem:limsup-1} and
$$E_1=\left\{x\in E : \limsup_{r\rightarrow 0}\frac{\nu(B(x,r))}{\varphi(r)}>2c_*\right\}.$$
Since Lemma \ref{lem:limsup-1} yields
$${\cal H}^{\varphi}(E_1)\leq \frac{1}{2}\nu(M)=\frac{1}{2}{\cal H}^{\varphi}(E),$$
we have
$${\cal H}^{\varphi}(E\setminus E_1)\geq {\cal H}^{\varphi}(E)-{\cal H}^{\varphi}(E_1)
\geq \frac{1}{2}{\cal H}^{\varphi}(E)>0$$
and thus $0<{\cal H}^{\varphi}(E\setminus E_1)<\infty$.

Define
$$h_n(x)=\sup_{0<r\leq 1/n}\frac{\nu(B(x,r))}{\varphi(r)},
\quad x\in E\setminus E_1, \ n\in {\mathbb N}$$
and
$$h(x)=\limsup_{r\rightarrow 0}\frac{\nu(B(x,r))}{\varphi(r)}, \quad x\in E\setminus E_1.$$
Then $h_n(x)\rightarrow h(x) $ as  $n\rightarrow\infty$ for any $x\in E\setminus E_1$.
Hence, by the Egorov theorem and \eqref{eq:regular},
there exists a compact subset $K$ of $E\setminus E_1$ such that
${\cal H}^{\varphi}(K)>0$ and
\begin{equation}\label{eq:uniform-conv}
\sup_{x\in K}|h_n(x)-h(x)|\rightarrow 0 ,\quad n\rightarrow\infty.
\end{equation}

Since $h(x)\leq 2c_*$ for any $x\in E\setminus E_1$,
\eqref{eq:uniform-conv} implies that for some $r_0>0$,
\begin{equation}\label{eq:bound-1}
\frac{\nu(B(x,r))}{\varphi(r)}\leq 4c_*, \quad x\in K, \ 0<r\leq r_0.
\end{equation}
As the function $\varphi$ is nondecreasing, we also have
\begin{equation}\label{eq:bound-2}
\frac{\nu(B(x,r))}{\varphi(r)}\leq \frac{{\cal H}^{\varphi}(E)}{\varphi(r_0)}=:c_1, \quad
x\in K,\ r\geq r_0.
\end{equation}
Hence if we let $b=(4c_*)\vee c_1$,
then \eqref{eq:bound-1} and \eqref{eq:bound-2} yield
$\nu(B(x,r))\leq b\varphi(r)$ for any $x\in K$ and $r>0$.
Moreover, by noting that $K\subset E$, we obtain
$${\cal H}^{\varphi}(B(x,r)\cap K)\leq {\cal H}^{\varphi}(B(x,r)\cap E)
=\nu(B(x,r))\leq b\varphi(r),
\quad x\in K, \ r>0.$$
The proof is complete.
\end{proof}

\begin{proof}[Proof of Proposition {\rm \ref{prop:frost-2}}]
Let $F$ be a Borel subset of $M$ such that ${\cal H}^{\varphi}(F)>0$.
Then, by Lemma \ref{lem:bound}, there exist a constant $b>0$ and
a compact subset $K$ of  $F$ such that $0<{\cal H}^{\varphi}(K)<\infty$ and
$$
{\cal H}^{\varphi}(B(x,r)\cap K)\leq b\varphi(r), \quad x\in K, \, r>0.
$$
Define $\nu(A)={\cal H}^{\varphi}(A\cap K)$ for $A\in {\cal B}(M)$. Then $\nu$ is a finite and nontrivial Borel measure on $M$ such that
$\supp[\nu]\subset K$,
and
\begin{equation}\label{eq:p:measure}
\nu(B(x,r))={\cal H}^{\varphi}(B(x,r)\cap K)\leq b\varphi(r), \quad x\in K, \ r>0.
\end{equation}

Fix a point $x$ in $K$ and let $m(r)=\nu(B(x,r))$ for $r>0$.
Then, for any $\varepsilon\in (0,1)$,
\begin{equation*}
\begin{split}
&\int_{M}\frac{1}{\varphi(d(x,y))^{\varepsilon}}\,\nu(\d y)
=\int_{d(x,y)\leq 1}\frac{1}{\varphi(d(x,y))^{\varepsilon}}\,\nu(\d y)
+\int_{d(x,y)>1}\frac{1}{\varphi(d(x,y))^{\varepsilon}}\,\nu(\d y)\\
&=\int_0^1\frac{1}{\varphi(r)^{\varepsilon}}\,\d m(r)
+\int_{d(x,y)>1}\frac{1}{\varphi(d(x,y))^{\varepsilon}}\,\nu(\d y)
\leq \int_0^1\frac{1}{\varphi(r)^{\varepsilon}}\,\d m(r)+\frac{{\cal H}^{\varphi}(K)}{\varphi(1)^{\varepsilon}}.
\end{split}
\end{equation*}
Moreover, we obtain by \eqref{eq:p:measure} that
\begin{equation*}
\begin{split}
\int_0^1\frac{1}{\varphi(r)^{\varepsilon}}\,{\rm d}m(r)
\leq b^{\varepsilon}\int_0^1\frac{1}{m(r)^{\varepsilon}}\,{\rm d}m(r)
\leq \frac{b^{\varepsilon}m(1)^{1-\varepsilon}}{1-\varepsilon}<\infty.
\end{split}
\end{equation*}
Therefore, there exists a constant $c_1=c_1(\varepsilon,K)>0$ such that
$$I^{{\varphi}^{\varepsilon}}(\nu)
=\int_M \int_M\frac{1}{\varphi(d(x,y))^{\varepsilon}}\,\nu(\d x)\nu(\d y)
\leq c_1\nu(K)<\infty.$$
The proof is complete.
\end{proof}

\subsection{Locally $s$-set and $s$-measure}
Let $(M,d)$ be a locally compact separable metric space.
In this subsection, for $x\in M$ and $r>0$, we still use the notation $B(x,r)$ for the closed ball
with radius $r$ centered at $x$, i.e., $B(x,r)=\{y\in M : d(x,y)\leq r\}$.
We recall the notions of locally $s$-sets and $s$-measures.
\begin{defn}\label{def:s-set}
Let $s$ and $t$ be positive constants.
\begin{itemize}
\item[{\rm (i)}]
A subset $F$ of $M$ is called a locally $s$-set,
if $F$ is a closed set and there exists a positive Borel measure $\eta$ on $M$
such that $\supp[\eta]\subset F$, and, for some positive constants $r_0$, $c_1(F)$ and $c_2(F)$,
$$c_1(F)r^s\leq \eta(B(x,r))\leq c_2(F)r^s, \quad x\in F, \ r\in (0,r_0).$$
The measure $\eta$ is called the locally $s$-measure of $F$.
\item[{\rm (ii)}]
A subset $F$ of $M$ is called a globally  $t$-set,
if $F$ is a closed set and there exists a positive Borel measure $\eta$ on $M$
such that $\supp[\eta]\subset F$, and, for some positive constants $r_0$, $c_3(F)$ and $c_4(F)$,
$$c_3(F)r^t\leq \eta(B(x,r))\leq c_4(F)r^t, \quad x\in F, \ r\geq r_0.$$
The measure $\eta$ is called the globally $t$-measure of $F$.
\item[{\rm (iii)}]
A subset $F$ of $M$ is called an $(s,t)$-set,
if $F$ is a locally $s$-set and globally $t$-set
such that the corresponding locally $s$-measure and globally $t$-measure are the same.
In particular, an $(s,s)$-set is called the $s$-set.
\end{itemize}
\end{defn}

Let $F\subset M$ be a locally $s$-set.
We then have $\dim_{{\cal H}}(F)=s$ by following the argument of \cite[8.7 (p.61--62)]{H01}.
For $x\in M$ and $r>0$,
if $B(x,r)\cap F\ne \emptyset$, then, for any $y\in B(x,r)\cap F$ and $z\in B(x,r)$,
$$d(y,z)\leq d(y,x)+d(x,z)\le
2r,$$
that is, $B(x,r)\subset B(y,2r)$.
This implies that
$$
\eta(B(x,r))\leq 2^s c_2(F) r^s, \quad x\in M, \ 0<r\leq r_0/2.
$$

Let $F$ be a globally $t$-set with $t$-measure $\eta$.
Then we also have
\begin{equation}\label{eq:s-1}
\eta(B(x,r))\leq 2^t c_4(F) r^t, \quad x\in M, \ r\geq r_0.
\end{equation}
This fact is already observed in \cite[Section 1]{MS09}.

Fix $x\in M$ and $x_0\in F$.
If $r\geq 2d(x,x_0)$, then, for any $y\in B(x_0,r/2)$,
$$d(x,y)\leq d(x,x_0)+d(x_0,y)\le
{r}/{2}+{r}/{2}=r,$$
that is, $B(x_0,r/2)\subset B(x,r)$.
Hence, if we define $d_F(x):=\inf_{y\in F}d(x,y) \ (x\in M)$, then
\begin{equation}\label{eq:s-2}
\eta(B(x,r))\geq \left({c_3(F)}/{2^t}\right) r^t, \quad x\in M, \ r\geq 2(d_F(x)\vee r_0).
\end{equation}

\section{Wiener tests}
In this appendix, we establish the Wiener tests
for recurrence and regularity of the stable subordinate process of
the direct product process.
Hereafter, $(M,d)$ is a locally compact separable metric space and
$\mu$ is a positive Radon measure on $M$ with full support.

\subsection{Transience and regularity}
Let $X:=(\Omega, {\cal F}, \{{\cal F}_t\}_{t\geq 0}, \{X_t\}_{t\geq 0},\{P_x\}_{x\in M}, \{\theta_t\}_{t\geq 0})$
be a $\mu$-symmetric Hunt process on $M$, where $\{{\cal F}_t\}_{t\geq 0}$ is a minimum completed admissible filtration,
and $\theta_t: \Omega \to \Omega$ is the shift of paths such that $X_t\circ\theta_s=X_{t+s}$
for $s,t\geq 0$. In this subsection, we will present equivalent conditions
for the transience and regularity of sets relative to the process $X$.

Let
$${\cal F}_{\infty}=\sigma\left(\bigcup_{t\geq 0}{\cal F}_t\right),$$ and
define the tail $\sigma$-field ${\cal T}$ by
$${\cal T}=\bigcap_{t>0}\sigma\left(\bigcup_{u\geq t}\sigma(X_s:s\in [t,u])\right).$$
We say that ${\cal T}$ is trivial, if for any $A \in {\cal T}$, $P_x(A)=1$ for any $x\in M$
or $P_x(A)=0$ for any $x\in M$.
For $B\in {\cal B}(M)$,
let $\sigma_B=\inf\{t>0 : X_t\in B\}$ be the first hitting time of $X$ to $B$,
and  let $ L_B=\sup\{t>0:X_t\in B\}$ be the last exit time of $X$ from $B$.
Then, $\{L_B<\infty\}\in {\cal T}$.

Below, for $x\in M$, $n\ge1$ and $B\in {\cal B}(M)$, define
$$B_n^{x,\lambda}=\left\{y\in B : \lambda^n\le d(x,y)\le \lambda^{n+1}\right\}$$
when $\lambda>1$, and
$$B_n^{x,\lambda}=\left\{y\in B : \lambda^{n+1}\le d(x,y)\le \lambda^n\right\}$$ when $0<\lambda\le 1$.

We first give equivalent conditions for the transience.
\begin{lem}\label{lem:0-1}
Assume that the process $X$ is conservative and transient, and that ${\cal T}$ is trivial.
Then, for any $x\in M$, $\lambda>1$ and $B\in {\cal B}(M)$,
the following assertions are equivalent to each other.
\begin{enumerate}
\item[{\rm (i)}] $P_x(L_B<\infty)=1$.
\item[{\rm (ii)}] $P_x(L_B<\infty)>0$.
\item[{\rm (iii)}] $P_x\left(\liminf_{n\rightarrow\infty}\{\sigma_{B_n^{x,\lambda}}=\infty\}\right)=1$.
\end{enumerate}
\end{lem}

\begin{proof}
Since ${\cal T}$ is trivial and $\{L_B<\infty\}\in {\cal T}$,
we obtain the equivalence between  (i) and (ii).

We now prove the equivalence between  (i) and (iii).
We simply write $B_n$ for $B_n^{x,\lambda}$.
Suppose first that (i) holds.
Then for $P_x$-a.s.\ $\omega\in \Omega$,
we have $X_t(\omega)\notin B$ for all $t> L_B(\omega)$.
Since $X$ is conservative, it follows by \cite[p.95, Corollary]{CW05} that
$$P_x\left(\text{$X_{t-}\in M$ and $X_t\in M$ for any $t>0$}\right)=1.$$
Then, by (i),
$$P_x\left(\sup_{s\in [0,L_B]}d(x,X_s)<\infty\right)=1,$$
which implies that
$$P_x\left(\text{$\sigma_{B_n}=\infty$ for all $n> \frac{\log \sup_{s\in [0,L_B]}d(x ,X_s)}{\log\lambda}$}\right)=1.$$
Therefore, (iii) follows.

Suppose next that (iii) holds.
Since $X$ is transient, we see that
$$P_x\left(\lim_{t\rightarrow\infty}d(x,X_t)=\infty\right)=1.$$
Then
$$1=P_x\left(\left\{\lim_{t\rightarrow\infty}d(x,X_t)=\infty\right\}
\cap \liminf_{n\rightarrow\infty}\left\{\sigma_{B_n}=\infty\right\}\right)
\leq P_x(L_B<\infty),$$
which yields (i).
\end{proof}

We next show the equivalent conditions for the regularity of points.

\begin{lem}\label{lem:0-1-reg}
Assume that the process $X$ is transient and that the single point set $\{x\}$ is polar relative to the process $X$.
If the Blumenthal
zero-one law holds for  the process $X$,
then, for any $\lambda\in (0,1)$ and $B\in {\cal B}(M)$,
the following three conditions are equivalent to each other.
\begin{enumerate}
\item[{\rm (i)}] $x$ is regular for $B$, that is, $P_x(\sigma_B=0)=1$.
\item[{\rm (ii)}] $P_x(\sigma_B=0)>0$.
\item[{\rm (iii)}] $P_x\left(\limsup_{n\rightarrow\infty}\{\sigma_{B_n^{x,\lambda}}<\infty\}\right)=1$.
\end{enumerate}
\end{lem}

\begin{proof}
The equivalence between (i) and (ii) follows by the Blumenthal zero-one law.

We now prove the equivalence between (i) and (iii).
Assume first that  (i) is valid.
Then for $P_x$-a.s.\ $\omega\in \Omega$,
there exists a sequence $\{t_n(\omega)\}$ such that $t_n(\omega)\downarrow 0$ as $n\rightarrow\infty$
and $X_{t_n(\omega)}(\omega)\in B$ for all $n\geq 1$.
Since $d(x,X_{t_n(\omega)}(\omega))\rightarrow 0$ as $n\rightarrow\infty$,
we have (iii).

Assume next that (iii) is valid.
We simply write $B_n$ for $B_n^{x,\lambda}$.
Then, for $P_x$-a.s.\ $\omega\in \Omega$,
there exist some sequences $\{n_k(\omega)\}_{k=1}^{\infty}\subset {\mathbb N}$
and $\{t_k(\omega)\}_{k=1}^{\infty}\subset [0,\infty)$ such that
$X_{t_k(\omega)}(\omega)\in B_{n_k(\omega)}$ for all $k\geq 1$.
This yields
$$\lambda^{n_k(\omega)+1}\leq d(x,X_{t_k(\omega)}(\omega))\leq \lambda^{n_k(\omega)}, \quad k\geq 1,$$
and thus
$$
d(x,X_{t_k(\omega)}(\omega))\rightarrow 0, \quad k\rightarrow\infty.
$$

On the other hand, since $X$ is transient, there exists $T(\omega)\in (0,\infty)$ such that
$d(x,X_t(\omega))>1$ for all $t>T(\omega)$,
which implies that $t_k(\omega)\in [0,T(\omega)]$ for all $k\geq 1$.
Therefore,  there exists a subsequence
$\{t_{k_l}(\omega)\}$ of $\{t_k(\omega)\}$ 
such that 
the monotone limit $t_0(\omega):=\lim_{l\rightarrow\infty}t_{k_l}(\omega)$ exists
in $[0,T(\omega)]$ and
\begin{equation}\label{eq:r-1}
\lim_{l\rightarrow
\infty}d(x,X_{t_{k_l}(\omega)}(\omega))=0.
\end{equation}
Moreover, as $\{x\}$ is polar relative to $X$,
it follows by \cite[p.95, Theorem 9]{CW05} that
$$P_x\left(\text{$d(x,X_{t-})>0$ and $d(x,X_t)>0$ for all $t>0$}\right)=1.$$
Taking \eqref{eq:r-1} into account, we have $t_0(\omega)=0$ and thus $P_x(\sigma_B=0)=1$.
\end{proof}

\subsection{Zero-one law for the tail event}\label{subsect:tail}
Let $X^1$ and $X^2$ be
the independent
$\mu$-symmetric Hunt processes on $M$, and let $X$ be
the direct product of $X^1$ and $X^2$ on $M\times M$.
For $\gamma\in (0,1]$, let $X^{\gamma}$ be the $\gamma$-subordinate process of $X$.
We will present the zero-one law for the tail event of $X^{\gamma}$.
Let $\{{\cal F}_t^{\gamma}\}_{t\geq 0}$ denote
the minimum completed admissible filtration of the process $X^{\gamma}$, and set
$${\cal F}_{\infty}^{\gamma}=\sigma\left(\bigcup_{t\geq 0}{\cal F}_t^{\gamma}\right).$$
Let ${\cal T}^{\gamma}$ be the tail $\sigma$-field of $X^{\gamma}$, i.e.,
$${\cal T}^{\gamma}
=\bigcap_{t>0}\sigma\left(\bigcup_{u\geq t}\sigma(X_s^{\gamma}:s\in [t,u])\right).$$
We then have

\begin{prop}\label{prop:tail}
Let $M$ satisfy Assumption {\rm \ref{assum:comp}}.
If the independent processes $X^1$ and $X^2$ satisfy Assumption {\bf (H)}, {\rm (WUHK)} and {\rm (HR)},
then, for any $A\in {\cal T}^{\gamma}$, $P_x^{\gamma}(A)=1$ for all $x\in M\times M$,
or $P_x^{\gamma}(A)=0$ for all $x\in M\times M$. In particular, ${\cal T}^{\gamma}$ is trivial.
Moreover, under the assumptions above,  the Blumenthal
zero-one law also holds true.
\end{prop}

Let us prove Proposition \ref{prop:tail}
by following the proof of \cite[Theorem 2.10]{KKW17}
(see also the references therein for the original proofs).
For $i=1,2$, let $p^i(s,x_i,y_i)$ be the heat kernel of the process $X^i$, and $q^{\gamma}(t,x,y)$ the heat kernel of $X^{\gamma}$, i.e.,
$$q^{\gamma}(t,x,y)=\int_0^{\infty}p^1(s,x_1,y_1)p^2(s,x_2,y_2)\pi_t(s)\,{\rm d}s.$$
We also let $q^{1,\gamma}(t,u,w)$ and $q^{2,\gamma}(t,u,w)$ be the heat kernels of
the subordinate processes of $X^1$ and $X^2$, respectively, i.e.,
$$q^{1,\gamma}(t,u,v)=\int_0^{\infty}p^1(s,u,v)\pi_t(s)\,{\rm d}s, \quad
q^{2,\gamma}(t,u,v)=\int_0^{\infty}p^2(s,u,v)\pi_t(s)\,{\rm d}s.$$
Since for $i=1,2$, any $t>0$ and $u\in M$,
\begin{equation}\label{eq:markov}
\int_M p^i(t,u,w)\, \mu({\rm d}w)\leq 1,
\end{equation}
we have  for any $A,B\in {\cal B}(M)$,
\begin{equation}\label{eq:sub1}
\begin{split}
&\iint_{A\times B}q^{\gamma}(t,x,y)\,\mu({\rm d}y)
=\int_0^{\infty}\left(\int_A p^1(s,x_1,y_1)\,\mu({\rm d}y_1)\right)\left(\int_B p^2(s,x_2,y_2)\,\mu({\rm d}y_2)\right)\pi_t(s)\,{\rm d}s\\
&\leq \left\{\int_0^{\infty}\left(\int_A p^1(s,x_1,y_1)\,\mu({\rm d}y_1)\right)\pi_t(s)\,{\rm d}s\right\}
\wedge \left\{\int_0^{\infty}\left(\int_B p^2(s,x_2,y_2)\,\mu({\rm d}y_2)\right)\pi_t(s)\,{\rm d}s\right\}\\
&=\left(\int_A q^{1,\gamma}(t,x_1,y_1)\,\mu({\rm d}y_1)\right)\wedge \left(\int_B q^{2,\gamma}(t,x_2,y_2)\,\mu({\rm d}y_2)\right).
\end{split}
\end{equation}

For $x=(x_1,x_2)\in M\times M$ and $r>0$,
let $B(x,r)$ be an open ball with radius $r$
centered at $x$ with respect to the product metric, i.e.,
$$B(x,r)=\left\{y=(y_1,y_2)\in M\times M : d(x_1,y_1)+d(x_2,y_2)<r \right\}.$$
Let $\tau_{B(x,r)}=\inf\{t>0 : X_t^{\gamma}\not\in B(x,r)\}$ be the exit time from $B(x,r)$ of the process $X^{\gamma}$.
\begin{lem}\label{lem:exit}
If the independent processes $X^1$ and $X^2$ satisfy {\rm (WUHK)},
then there exists a constant $c_1>0$ such that for any $x\in M_0^1\times M_0^2$, $t\geq 0$ and $r>0$,
$$P_x^{\gamma}(\tau_{B(x,r)}\leq t)\leq c_1t\left(\frac{1}{\phi_1^{\gamma}(r)}+\frac{1}{\phi_2^{\gamma}(r)}\right).$$
\end{lem}

\begin{proof}
Suppose that the processes $X^1$ and $X^2$ satisfy {\rm (WUHK)}.
For $x=(x_1,x_2)\in M$ and $r>0$, we write  $\tau=\tau_{B(x,r)}$ for simplicity.
Then
\begin{equation*}
\begin{split}
P_x^{\gamma}(\tau\leq t)
&=P_x^{\gamma}(\tau\leq t, X_{2t}^{\gamma}\in B(x,r/2))+P_x^{\gamma}(\tau\leq t, X_{2t}^{\gamma}\notin B(x,r/2))\\
&\leq P_x^{\gamma}(\tau\leq t, d(x,X_{2t}^{\gamma})\leq r/2)+P_x^{\gamma}(d(x,X_{2t}^{\gamma})\geq r/2).
\end{split}
\end{equation*}

Since $X^1$ and $X^2$ satisfy {\rm (WUHK)}, it follows by \eqref{eq:wuhk-sub} that
there exist positive constants $c_1$ and $c_2$ such that for each $i=1,2$,
and for any $t> 0$ and $r>0$,
$$\int_{d(x_i,w)\geq r/4}q^{i,\gamma}(2t,x_i,w)\,\mu({\rm d}w)
\leq c_1\int_{d(x_i,w)\geq r/4}\frac{t}{V(x_i,d(x_i,w))\phi_i^{\gamma}(d(x_i,w))}\,\mu({\rm d}w)
\leq \frac{c_2 t}{\phi_i^{\gamma}(r)}.
$$
Note that if $d(x,y)\geq r/2$, then $d(x_1,y_1)\geq r/4$ or $d(x_2,y_2)\geq r/4$.
Therefore, by \eqref{eq:sub1}, there exists a constant $c_3>0$ such that
for any $x\in M_0^1\times M_0^2$, $t>0$ and $r>0$,
\begin{equation}\label{e:ine4}
\begin{split}
&P_x^{\gamma}(d(x,X_{2t}^{\gamma})\geq r/2)
=\int_{d(x,y)\geq r/2}q^{\gamma}(2t,x,y)\,\mu({\rm d}y)\\
&\leq \int_{d(x_1,y_1)\geq r/4}q^{1,\gamma}(2t,x_1,y_1)\,\mu({\rm d}y_1)
+\int_{d(x_2,y_2)\geq r/4}q^{2,\gamma}(2t,x_2,y_2)\,\mu({\rm d}y_2)\\
&\leq c_3t\left(\frac{1}{\phi_1^{\gamma}(r)}+\frac{1}{\phi_2^{\gamma}(r)}\right).
\end{split}
\end{equation}
Then,  by the triangle inequality and the strong Markov property,
we get
\begin{equation*}
\begin{split}
P_x^{\gamma}(\tau\leq t, d(x,X_{2t}^{\gamma})\leq r/2)
&\leq E_x^{\gamma}\left[P_{X_{\tau}}^{\gamma}\left(d(X_{2t-s}^{\gamma},X_0^{\gamma})\geq r/2\right)\mid_{s=\tau};\tau\leq t \right]\\
&\leq \sup_{s\in [0,t], d(x,z)\geq r}P_z^{\gamma}\left(d(X_{2t-s}^{\gamma},z)\geq r/2\right)
\leq c_3t\left(\frac{1}{\phi_1^{\gamma}(r)}+\frac{1}{\phi_2^{\gamma}(r)}\right),
\end{split}
\end{equation*}
where the last inequality follows from the argument of \eqref{e:ine4}.
We thus complete the proof.
\end{proof}

\begin{proof}[Proof of Proposition {\rm \ref{prop:tail}}]
We split the proof into three parts.

(i) Throughout the proof,
we will fix $\varepsilon>0$ small enough.
 By Lemma \ref{lem:exit} and \eqref{eq:sca-phi},
there exist positive constants $c_1$ and $c_2$ so that
for all $x\in M\times M$,
$t_0>0$ and $c_{\star}\ge1$,
\begin{equation}\label{eq:exit-bdd}
\begin{split}
&P_x^{\gamma}\left(
\sup_{s\leq t_0}d(X_s^{\gamma},x)>c_{\star}((\phi_1^{\gamma})^{-1}(t_0)\vee(\phi_2^{\gamma})^{-1}(t_0))\right)\\
&\leq c_1t_0\left(\frac{1}{\phi_1^{\gamma}(c_{\star}((\phi_1^{\gamma})^{-1}(t_0)\vee (\phi_2^{\gamma})^{-1}(t_0)))}
+\frac{1}{\phi_2^{\gamma}(c_{\star}((\phi_1^{\gamma})^{-1}(t_0)\vee (\phi_2^{\gamma})^{-1}(t_0)))}\right)\\
&\leq \frac{c_2}{c_{\star}^{\gamma \alpha_1}}.
\end{split}
\end{equation}
Take $c_{\star}$ large enough so that ${c_2}/{c_{\star}^{\gamma \alpha_1}}<\varepsilon.$

Let $c_*>1$ and $t_1>0$ be constants which will be fixed later in this order.
We first fix $c_*>1$.
Then, by (WUHK), \eqref{eq:wuhk-sub} and \eqref{eq:sca-phi},
\begin{equation}\label{eq:markov0}
\begin{split}
&\int_{d(x_1,y_1)\geq c_*(\phi_1^{\gamma})^{-1}(t_1)}
q^{1,\gamma}(t_1,x_1,y_1)\,\mu({\rm d}y_1)\\
&\leq c_3\int_{d(x_1,y_1)\geq c_* (\phi_1^{\gamma})^{-1}(t_1)}
\frac{t_1}{V(x_1,d(x_1,y_1))\phi_1^{\gamma}(d(x_1,y_1))}\,\mu({\rm d}y_1)\\
&\leq \frac{c_4t_1}{\phi_1^{\gamma}(c_*(\phi_1^{\gamma})^{-1}(t_1))}
\leq \frac{c_5}{c_*^{\gamma \alpha_1}}.
\end{split}
\end{equation}
Here the positive constants $c_3, c_4, c_5$ above
are independent of the choices of $c_*$, $t_1$, $x_1$ and $y_1$.
Similarly, we have
\begin{equation}\label{eq:markov1}
\int_{d(x_2,y_2)\geq c_*(\phi_2^{\gamma})^{-1}(t_1)}q^{2,\gamma}(t_1,x_2,y_2)\,\mu({\rm d}y_2)
\leq \frac{c_6}{c_*^{\gamma\alpha_1}}.
\end{equation}
Note that, by \eqref{eq:sub1},
\begin{equation*}
\begin{split}
&\int_{d(x_1,y_1)\geq c_*(\phi_1^{\gamma})^{-1}(t_1)}
q^{\gamma}(t_1,x,y)\,\mu({\rm d}y)
\leq \int_{d(x_1,y_1)\geq c_*(\phi_1^{\gamma})^{-1}(t_1)}
q^{1,\gamma}(t_1,x_1,y_1)\,\mu({\rm d}y_1)
\end{split}
\end{equation*}
and
\begin{equation*}
\begin{split}
&\int_{d(x_2,y_2)\geq c_*(\phi_2^{\gamma})^{-1}(t_1)}
q^{\gamma}(t_1,x,y)\,\mu({\rm d}y)
\leq \int_{d(x_2,y_2)\geq c_*(\phi_2^{\gamma})^{-1}(t_1)}
q^{2,\gamma}(t_1,x_2,y_2)\,\mu({\rm d}y_2).
\end{split}
\end{equation*}
Hence it follows  by \eqref{eq:markov0} and \eqref{eq:markov1}
that,
if we take  $c_*>1$ so large that
$(c_5+c_6)/c_*^{\gamma \alpha_1}<\varepsilon/4$, then  for any $x\in M\times M$ and $t_1>0$,
\begin{equation}\label{eq:b-1}
\begin{split}
&\int_{\{d(x_1,y_1)\geq c_*(\phi_1^{\gamma})^{-1}(t_1)\}\cup\{d(x_2,y_2)\geq c_*(\phi_2^{\gamma})^{-1}(t_1)\}}
q^{\gamma}(t_1,x,y)\,\mu({\rm d}y)\\
&\leq \int_{d(x_1,y_1)\geq c_*(\phi_1^{\gamma})^{-1}(t_1)}
q^{\gamma}(t_1,x,y)\,\mu({\rm d}y)
+\int_{d(x_2,y_2)\geq c_*(\phi_2^{\gamma})^{-1}(t_1)}
q^{\gamma}(t_1,x,y)\,\mu({\rm d}y)\\
&\leq \int_{d(x_1,y_1)\geq c_*(\phi_1^{\gamma})^{-1}(t_1)}
q^{1,\gamma}(t_1,x_1,y_1)\,\mu({\rm d}y_1)
+\int_{d(x_2,y_2)\geq c_*(\phi_2^{\gamma})^{-1}(t_1)}
q^{2,\gamma}(t_1,x_2,y_2)\,\mu({\rm d}y_2)\\
&\leq \frac{c_5+c_6}{c_*^{\gamma \alpha_1}}<\frac{\varepsilon}{4}.
\end{split}
\end{equation}
In the same way, we can take and fix $c_*>1$ so large that  for any $z\in M\times M$ and $t_1>0$,
\eqref{eq:b-1} holds and
\begin{equation}\label{eq:b-2}
\int_{\{d(z_1,y_1)\geq c_*(\phi_1^{\gamma})^{-1}(t_1)/2\}\cup\{d(z_2,y_2)\geq c_*(\phi_2^{\gamma})^{-1}(t_1)/2\}}
q^{\gamma}(t_1,z,y)\,\mu({\rm d}y)
<\frac{\varepsilon}{4}.
\end{equation}

Next we assume that
$d(x_1,z_1)+d(x_2,z_2)
\leq c_{\star}((\phi_1^{\gamma})^{-1}(t_0)\vee (\phi_2^{\gamma})^{-1}(t_0))$.
We now take  $t_1>0$ so large that
\begin{equation}\label{eq:t-1}
c_{\star}((\phi_1^{\gamma})^{-1}(t_0)\vee (\phi_2^{\gamma})^{-1}(t_0))
\leq \frac{c_*}{2}((\phi_1^{\gamma})^{-1}(t_1)\wedge (\phi_2^{\gamma})^{-1}(t_1)).
\end{equation}
For each $i=1,2$,
if $d(x_i,y_i)\geq c_*(\phi_i^{\gamma})^{-1}(t_1)$, then, by the triangle inequality
and \eqref{eq:t-1},
$$d(z_i,y_i)\geq d(x_i,y_i)-d(x_i,z_i)
\geq c_*(\phi_i^{\gamma})^{-1}(t_1)-c_{\star}((\phi_i^{\gamma})^{-1}(t_0)\vee \phi_i^{-1}(t_0))
\geq \frac{c_*}{2}(\phi_i^{\gamma})^{-1}(t_1).$$
Therefore, it follows by \eqref{eq:b-2} that
\begin{equation*}
\begin{split}
&\int_{\{d(x_1,y_1)\geq c_*\phi_1^{-1}(t_1)\}\cup\{d(x_2,y_2)\geq c_*\phi_2^{-1}(t_1)\}}
q^{\gamma}(t_1,z,y)\,\mu({\rm d}y)\\
&\leq \int_{\{d(z_1,y_1)\geq c_*\phi_1^{-1}(t_1)/2\}\cup\{d(z_2,y_2)\geq c_*\phi_2^{-1}(t_1)/2\}}
q^{\gamma}(t_1,z,y)\,\mu({\rm d}y)
<\frac{\varepsilon}{4}.
\end{split}
\end{equation*}
Combining this with \eqref{eq:b-1},
we obtain
\begin{equation}\label{eq:b-3}
\begin{split}
&\left|\int_{\{d(x_1,y_1)\geq c_*\phi_1^{-1}(t_1)\}\cup\{d(x_2,y_2)\geq c_*\phi_2^{-1}(t_1)\}}
\left(q^{\gamma}(t_1,x,y)-q^{\gamma}(t_1,z,y)\right)\mu({\rm d}y)\right|\\
&\leq \int_{\{d(x_1,y_1)\geq c_*\phi_1^{-1}(t_1)\}\cup\{d(x_2,y_2)\geq c_*\phi_2^{-1}(t_1)\}}
\left(q^{\gamma}(t_1,x,y)+q^{\gamma}(t_1,z,y)\right)\mu({\rm d}y)
<\frac{\varepsilon}{2}.
\end{split}
\end{equation}

Since the processes $X^1$ and $X^2$ satisfy (HR) by assumption, for each $i=1,2$, there exist constants
$\theta_i\in (0,1]$ and $C_i>0$ such that for any $t>0$ and $u,v,w\in M$,
\begin{equation}\label{eq:heat-conti}
|p^i(t,u,w)-p^i(t,v,w)|\leq \frac{C_i}{V(w,\phi_i^{-1}(t))}\left(\frac{d(u,v)}{\phi_i^{-1}(t)}\right)^{\theta_i}.
\end{equation}
Therefore, as in the proof of Lemma \ref{lem:heat-sub} (1),
we can show that
\begin{equation}\label{eq:h-sub}
\begin{split}
\int_0^{\infty}|p^i(s,u,w)-p^i(s,v,w)|\pi_t(s)\,{\rm d}s
&\leq \int_0^{\infty}\frac{C_i}{V(w,\phi_i^{-1}(s))}\left(\frac{d(u,v)}{\phi_i^{-1}(s)}\right)^{\theta_i}\pi_t(s)\,{\rm d}s\\
&\leq \frac{C_i'}{V(w,(\phi_i^{\gamma})^{-1}(t))}\left(\frac{d(u,v)}{(\phi_i^{\gamma})^{-1}(t)}\right)^{\theta_i}.
\end{split}
\end{equation}
Hence, if $d(u,v)\leq c_{\star}((\phi_1^{\gamma})^{-1}(t_0)\vee (\phi_2^{\gamma})^{-1}(t_0))$,
then there exist positive constants $c_7$ and $\eta$ such that
\begin{equation*}
\begin{split}
&\int_{d(u,w)\leq c_*(\phi_i^{\gamma})^{-1}(t_1)}\left(\int_0^{\infty}|p^i(s,u,w)-p^i(t,v,w)|\pi_t(s)\,{\rm d}s\right)\,\mu({\rm d}w)\\
&\leq C_i''
\left(\frac{c_{\star}((\phi_1^{\gamma})^{-1}(t_0)\vee (\phi_2^{\gamma})^{-1}(t_0))}{(\phi_i^{\gamma})^{-1}(t_1)}\right)^{\theta_i}
\frac{V(u,c_*(\phi_i^{\gamma})^{-1}(t_1))}{V(u,(\phi_i^{\gamma})^{-1}(t_1))}\\
&\leq c_7c_*^{\eta}\left(\frac{c_{\star}((\phi_1^{\gamma})^{-1}(t_0)\vee (\phi_2^{\gamma})^{-1}(t_0))}{(\phi_i^{\gamma})^{-1}(t_1)}\right)^{\theta_i},
\end{split}
\end{equation*}
where in the second
inequality we used \eqref{eq:v-inverse}.
In particular, if we take $t_1>0$ so large that
\begin{equation}\label{eq:t-2}
c_{\star}((\phi_1^{\gamma})^{-1}(t_0)\vee (\phi_2^{\gamma})^{-1}(t_0))
\leq \left\{\left(\frac{\varepsilon}{4c_7c_*^{\eta}}\right)^{1/\theta_1}\wedge
\left(\frac{\varepsilon}{4c_7c_*^{\eta}}\right)^{1/\theta_2}\right\}
((\phi_1^{\gamma})^{-1}(t_1)\wedge (\phi_2^{\gamma})^{-1}(t_1)),
\end{equation}
then
$$
\int_{d(u,w)\leq c_*(\phi_i^{\gamma})^{-1}(t_1)}
\left(\int_0^{\infty}|p^i(s,u,w)-p^i(s,v,w)|\pi_{t_1}(s)\,{\rm d}s\right)\,\mu({\rm d}w)<\frac{\varepsilon}{4}.
$$
Set
$$I_1(t,x,y,z)=\int_0^{\infty}p^1(s,x_1,y_1)(p^2(s,x_2,y_2)-p^2(s,z_2,y_2))\,\pi_t(s)\,{\rm d}s,$$
and
$$I_2(t,x,y,z)=\int_0^{\infty}p^2(s,z_2,y_2)(p^1(s,x_1,y_1)-p^1(s,z_1,y_1))\,\pi_t(s)\,{\rm d}s.$$
Then, for any $f\in {\cal B}_b(M\times M)$,  by \eqref{eq:markov} and the Fubini theorem,
\begin{equation}\label{eq:f-1}
\begin{split}
&\left|\int_{d(x_1,y_1)\leq c_*(\phi_1^{\gamma})^{-1}(t_1), d(x_2,y_2)\leq c_*(\phi_2^{\gamma})^{-1}(t_1)}
I_1(t_1,x,y,z)
f(y)\,\mu({\rm d}y)\right|\\
&\leq \|f\|_{\infty}
\int_{d(x_2,y_2)\leq c_*(\phi_2^{\gamma})^{-1}(t_1)}
\left(\int_0^{\infty}p^1(s,x_1,y_1)|p^2(s,x_2,y_2)-p^2(s,z_2,y_2)|\,\pi_{t_1}(s)\,{\rm d}s\right)
\,\mu({\rm d}y)\\
&\le \|f\|_{\infty}
\int_{d(x_2,y_2)\leq c_*(\phi_2^{\gamma})^{-1}(t_1)}
\left(\int_0^{\infty}|p^2(s,x_2,y_2)-p^2(s,z_2,y_2)|\,\pi_{t_1}(s)\,{\rm d}s\right)
\,\mu({\rm d}y_2)\\
&\leq \frac{\varepsilon}{4} \|f\|_{\infty}
\end{split}
\end{equation}
and
\begin{equation}\label{eq:f-2}
\left|\int_{d(x_1,y_1)\leq c_*(\phi_i^{\gamma})^{-1}(t_1), d(x_2,y_2)\leq c_*(\phi_2^{\gamma})^{-1}(t_1)}
I_2(t_1,x,y,z)
f(y)\,\mu({\rm d}y)\right|
\leq \frac{\varepsilon}{4}\|f\|_{\infty}.
\end{equation}
Note that
\begin{equation*}
\begin{split}
&p^1(s,x_1,y_1)p^2(s,x_2,y_2)-p^1(s,z_1,y_1)p^2(s,z_2,y_2)\\
&=p^1(s,x_1,y_1)(p^2(s,x_2,y_2)-p^2(s, z_2,y_2))
+p^2(s,z_2,y_2)(p^1(s,x_1,y_1)-p^2(s,z_1,y_1)).
\end{split}
\end{equation*}
Hence, by \eqref{eq:f-1} and \eqref{eq:f-2},
\begin{equation*}
\begin{split}
&\left|\int_{d(x_1,y_1)\leq c_*(\phi_1^{\gamma})^{-1}(t_1), d(x_2,y_2)\leq c_*(\phi_2^{\gamma})^{-1}(t_1)}
(q^{\gamma}(t_1,x,y)-q^{\gamma}(t_1,z,y))
f(y)\,\mu({\rm d}y)\right|\\
&=\left|\int_{d(x_1,y_1)\leq c_*(\phi_1^{\gamma})^{-1}(t_1), d(x_2,y_2)\leq c_*(\phi_2^{\gamma})^{-1}(t_1)}
(I_1(t_1,x,y,x)+I_2(t_1,x,y,z))
f(y)\,\mu({\rm d}y)\right|\\
&\leq \left|\int_{d(x_1,y_1)\leq c_*(\phi_1^{\gamma})^{-1}(t_1), d(x_2,y_2)\leq c_*(\phi_2^{\gamma})^{-1}(t_1)}
I_1(t_1,x,y,z)
f(y)\,\mu({\rm d}y)\right|\\
&\quad +\left|\int_{d(x_1,y_1)\leq c_*(\phi_1^{\gamma})^{-1}(t_1), d(x_2,y_2)\leq c_*(\phi_2^{\gamma})^{-1}(t_1)}
I_2(t_1,x,y,z)
f(y)\,\mu({\rm d}y)\right|\\
&\leq \frac{\varepsilon}{2}\|f\|_{\infty}.
\end{split}
\end{equation*}
Therefore, if we fix  $t_1>0$ so that \eqref{eq:t-1} and \eqref{eq:t-2} hold,
then, by \eqref{eq:b-3} and the inequality above,
\begin{equation}\label{eq:s-bdd}
\left|P_{t_1}^{\gamma}f(x)-P_{t_1}^{\gamma}f(z)\right|
=\left|\int_{M\times M}
\left(q^{\gamma}(t_1,x,y)-q^{\gamma}(t_1,z,y)\right)
f(y)\,\mu({\rm d}y)\right|
\leq \varepsilon\|f\|_{\infty}.
\end{equation}

(ii) Fix $x\in M\times M$ and $A\in {\cal T}^{\gamma}$.
Then, by the martingale convergence theorem,
we have as $t\rightarrow\infty$,
$$E_x^{\gamma}[{\bf 1}_A \mid {\cal F}_t^{\gamma}]
\rightarrow E_x^{\gamma}[{\bf 1}_A \mid{\cal F}_{\infty}^{\gamma}]={\bf 1}_A
\quad  \text{$P_x^{\gamma}$-a.s.\ and in $L^1(P_x^{\gamma})$}.$$
Namely, for any fixed $\varepsilon>0$, there exists $t_0>0$ such that
$$E_x^{\gamma}\left[\left|E_x^{\gamma}[{\bf 1}_A\mid {\cal F}_{t_0}^{\gamma}]-{\bf 1}_A\right|\right]<\varepsilon.$$
Hence, letting $Y=E_x^{\gamma}[{\bf 1}_A\mid {\cal F}_{t_0}^{\gamma}]$,
we obtain
\begin{equation}\label{eq:t-prob}
|P_x^{\gamma}(A)-E_x^{\gamma}[Y;A]|=|E_x^{\gamma}[({\bf 1}_A-Y);A]|
\leq E_x^{\gamma}[|{\bf 1}_A-Y|]=E_x^{\gamma}[|{\bf 1}_A-E_x^{\gamma}[{\bf 1}_A\mid {\cal F}_{t_0}^{\gamma}]|]
<\varepsilon
\end{equation}
and
\begin{equation}\label{eq:t-prob-1}
|P_x^{\gamma}(A)-E_x^{\gamma}[Y]|\leq E_x^{\gamma}[|{\bf 1}_A-Y|]
<\varepsilon.
\end{equation}

Let $t_0$ and $t_1$ be the positive constants
which are fixed in the argument in part (i).
Then, for $A\in {\cal T}^{\gamma}$, there exists an event $C\in {\cal F}_{\infty}^{\gamma}$
such that $A=C\circ\theta_{t_0+t_1}$.
Let $g(x)=P_x^{\gamma}(C)$ for $x\in M\times M$.
Since $Y$ is ${\cal F}_{t_0}^{\gamma}$-measurable and the Markov property yields
$$E_x^{\gamma}\left[{\bf 1}_C\circ\theta_{t_1}\right]
=E_x^{\gamma}\left[P_{X_{t_1}}^{\gamma}(C)\right]=P_{t_1}^{\gamma}g(x),$$
we have
\begin{equation}\label{eq:t-3}
E_x^{\gamma}\left[Y;A\right]=E_x^{\gamma}\left[Y;C\circ\theta_{t_0+t_1}\right]
=E_x^{\gamma}\left[YE_{X_{t_0}}^{\gamma}\left[{\bf 1}_C\circ\theta_{t_1}\right]\right]
=E_x^{\gamma}\left[Y P_{t_1}^{\gamma}g(X^\gamma_{t_0})\right]
\end{equation}
and
\begin{equation}\label{eq:t-4}
P_x^{\gamma}(A)=E_x^{\gamma}\left[P_{t_1}g(X^\gamma_{t_0})\right].
\end{equation}

Let
$$A_{t_0}=\left\{\omega\in \Omega:
d(X_{t_0}^{\gamma}(\omega),x)\leq c_{\star}((\phi_1^{\gamma})^{-1}(t_0)\vee (\phi_2^{\gamma})^{-1}(t_0))\right\}.$$
Since $\|g\|_{\infty}\leq 1$, we get, by \eqref{eq:s-bdd},
\begin{equation*}
\begin{split}
\left|E_x^{\gamma}[Y P_{t_1}^{\gamma}g(X_{t_0}^{\gamma});A_{t_0}]-P_{t_1}^{\gamma}g(x)E_x^{\gamma}[Y;A_{t_0}]\right|
\leq E_x^{\gamma}[Y|P_{t_1}^{\gamma}g(X_{t_0}^{\gamma})-P_{t_1}^{\gamma}g(x)|;A_{t_0}]\leq \varepsilon.
\end{split}
\end{equation*}
We also see, by \eqref{eq:exit-bdd}, that
$$\left|E_x^{\gamma}[Y P_{t_1}^{\gamma}g(X_{t_0}^{\gamma});A_{t_0}^c]-P_{t_1}^{\gamma}g(x)E_x^{\gamma}[Y;A_{t_0}^c]\right|
\leq 2P_x^{\gamma}(A_{t_0}^c)<2\varepsilon.$$
Therefore, it follows by \eqref{eq:t-3} that
\begin{equation*}
\begin{split}
&|E_x^{\gamma}[Y;A]-P_{t_1}^{\gamma}g(x)E_x^{\gamma}[Y]|
=|E_x^{\gamma}[Y P_{t_1}^{\gamma}g(X_{t_0}^{\gamma})]-P_{t_1}^{\gamma}g(x)E_x^{\gamma}[Y]|\\
&\leq \left|E_x^{\gamma}[Y P_{t_1}g(X_{t_0}^{\gamma});A_{t_0}]-P_{t_1}^{\gamma}g(x)E_x^{\gamma}[Y;A_{t_0}]\right|
+\left|E_x^{\gamma}[Y P_{t_1}^{\gamma}g(X_{t_0}^{\gamma});A_{t_0}^c]
-P_{t_1}^{\gamma}g(x)E_x^{\gamma}[Y;A_{t_0}^c]\right|\\
&<3\varepsilon.
\end{split}
\end{equation*}
Similarly, we have, by \eqref{eq:t-4},
$$|P_x^{\gamma}(A)-P_{t_1}^{\gamma}g(x)|
=|E_x^{\gamma}[P_{t_1}^{\gamma}g(X_{t_0}^{\gamma})]-P_{t_1}^{\gamma}g(x)|<3\varepsilon.$$
Combining two inequalities above with \eqref{eq:t-prob}, we obtain
\begin{equation*}
\begin{split}
&\left|P_x^{\gamma}(A)-P_x^{\gamma}(A)E_x^{\gamma}[Y]\right|\\
&\leq \left|P_x^{\gamma}(A)-E_x^{\gamma}[Y;A]\right|
+\left|E_x^{\gamma}[Y;A]-P_{t_1}^{\gamma}g(x)E_x^{\gamma}[Y]\right|
+\left|P_{t_1}^{\gamma}g(x)-P_x^{\gamma}(A)\right|E_x^{\gamma}[Y]\\
&\leq \varepsilon+3\varepsilon+3\varepsilon=7\varepsilon.
\end{split}
\end{equation*}
Then, by \eqref{eq:t-prob-1}, we further have
$$|P_x^{\gamma}(A)-P_x^{\gamma}(A)^2|
\leq |P_x^{\gamma}(A)-P_x^{\gamma}(A)E_x^{\gamma}[Y]|
+P_x^{\gamma}(A)|E_x^{\gamma}[Y]-P_x^{\gamma}(A)|
<7\varepsilon+\varepsilon=8\varepsilon.$$
Since $\varepsilon>0$ is arbitrary, we get $P_x^{\gamma}(A)=P_x^{\gamma}(A)^2$
for any $x\in M\times M$.

Fix $t>0$ and $h\in {\cal B}_b(M\times M)$.
Then, the same argument as for \eqref{eq:s-bdd} implies that
for any $\varepsilon>0$, there exists $\delta>0$ such that
if $x,z\in M\times M$ satisfy $d(x_1,z_1)+d(x_2,z_2)<\delta$, then
$$|P_t^{\gamma} h(x)-P_t^{\gamma} h(z)|<\varepsilon.$$
Namely, the function $P_t^{\gamma} h$ is uniformly continuous in $M\times M$.
Moreover, since \eqref{eq:t-4} yields
$$P_x^{\gamma}(A)=P_{t_0}^{\gamma}P_{t_1}^{\gamma}g(x),$$
the function $P_x^{\gamma}(A)$ is continuous in $x\in M\times M$.
We also note that $M\times M$ is connected because so is $M$ by assumption.
As $P_x^{\gamma}(A)=P_x^{\gamma}(A)^2$ for any $x\in M\times M$,
we get $P_x^{\gamma}(A)=1$ for all $x\in M\times M$, or $P_x^{\gamma}(A)=0$ for all $x\in M\times M$.

(iii) Since \eqref{eq:exit-bdd} holds for all $t_0>0$,  one can see from the arguments in part (i) (in particular \eqref{eq:s-bdd}) that the semigroup of the process $X^\gamma$ satisfies the Feller property,
i.e., the associated semigroup maps the set of bounded continuous functions into itself.
Then, according to \cite[p.57]{Bl57}, the Blumenthal
zero-one law holds as well.
The proof is complete.
\end{proof}

\subsection{Wiener test for recurrence}
In this subsection, we establish the Wiener test for the recurrence
relative to stable-subordinate direct-product processes
by using Proposition \ref{prop:tail}.
Let $X^1$ and $X^2$ be
two independent
$\mu$-symmetric Hunt processes on $M$ satisfying Assumption {\bf(H)},
{\rm (WUHK)} and {\rm (HR)},
and let $X$ be the direct product of $X^1$ and $X^2$ on $M\times M$.
For $\gamma\in (0,1]$, $X^{\gamma}$ denotes the $\gamma$-subordinate process of $X$.

Fix $x\in M\times M$ and $\lambda>1$.
For $B\in {\cal B}(M\times M)$,  define
\begin{equation}\label{eq:b-p}
B_n^{x,\lambda,\phi}
=\left\{y\in B: \lambda^n\leq \phi_d(x,y)\leq \lambda^{n+1}\right\},\quad n\ge1.
\end{equation}
Suppose that Assumptions \ref{assum:comp} is satisfied.
Then, by Proposition \ref{prop:tail},
one can apply the argument of Lemma \ref{lem:0-1} to the process $X^{\gamma}$ and obtain that,
if the process $X^{\gamma}$ is transient, then
\begin{equation}\label{eq:p-equ}
P_x^{\gamma}(L_B<\infty)=1\iff P_x^{\gamma}(L_B<\infty)>0
\iff P_x^{\gamma}\left(\liminf_{n\rightarrow\infty}\{\sigma_{B_n^{x,\lambda,\phi}}=\infty 
\}\right)=1.
\end{equation}
Furthermore, using this equivalence, we can show the Wiener test,
which is well known for the transient Brownian motion (see, e.g., \cite[p.67, Theorem 3.3]{PS78}),
for the stable-subordinate of direct product process on the metric measure space.

\begin{prop}\label{thm:wiener}
Let $M$ satisfy Assumptions {\rm \ref{assum:volume}}, {\rm \ref{assum:compact}} and {\rm \ref{assum:comp}}.
Suppose that the two independent
 processes $X^1$ and $X^2$ satisfy Assumption  {\bf (H)}, {\rm (WUHK)} and {\rm(HR)}.
Fix a constant $\lambda>1$ so that $\phi_i^{-1}(\lambda t)\geq 2\phi_i^{-1}(t) $ for $i=1,2$ and all $t>0$.
Assume that for some $\gamma\in (0,1]$,
$J^{\gamma}<\infty$, \eqref{eq:prod-res-cond} and \eqref{eq:prod-res-cond-1} hold.
Then, for any $x\in M\times M$ and $B\in {\cal B}(M\times M)$,
$$
P_x^{\gamma}(L_B=\infty)=1\iff\sum_{n=1}^{\infty}P_x^{\gamma}(\sigma_{B_n^{x,\lambda,\phi}}<\infty)=\infty.
$$
\end{prop}

\begin{proof}
We first note that, by Remark \ref{rem:heat-kernel} (iii),
under the assumption of this
proposition,
the processes $X^1$ and $X^2$ satisfy {\rm (NDLHK)}.
We also note that $X^{\gamma}$ is transient by Lemma \ref{lem:trans}.
Then, we take an approach similar to the proof of \cite[p.67, Theorem 3.3]{PS78}.
In what follows, we simply write $B_n$ for $B_n^{x,\lambda,\phi}$.

Let $m$, $n$ be positive integers such that $|m-n|>1$.
Without loss of generality, we suppose that $m>n+1$.
For any   $z\in B_m$ and $y\in B_n$,
if $\phi_1(d(x_1,z_1))\geq \phi_2(d(x_2,z_2))$, then
\begin{equation}\label{eq:p}
\phi_1(d(x_1,z_1))=\phi_d(x,z)\geq \lambda^m\geq \lambda^{n+2}
\geq \lambda \phi_d(x,y).
\end{equation}
Noting that
$\phi_1^{-1}(\lambda \phi_d(x,y))\geq 2\phi_1^{-1}(\phi_d(x,y))$
by assumption,
we have, by the triangle inequality,
\begin{equation*}
\begin{split}
d(z_1, y_1)\geq d(x_1,z_1)-d(y_1,x_1)
&\geq \phi_1^{-1} \left(\lambda \phi_d(x,y)\right)
-\phi_1^{-1}(\phi_d(x,y))
\geq \phi_1^{-1}(\phi_d(x,y)).
\end{split}
\end{equation*}
Therefore,
\begin{equation}\label{eq:p-0}
\phi_d(z,y)\geq  \phi_d(x,y).
\end{equation}
Since \eqref{eq:p} also implies that
$$d(x_1,z_1)\ge \phi_1^{-1}(\lambda \phi_1(d(x_1,y_1)))\ge 2\phi_1^{-1}(\phi_1(d(x_1,y_1)))=2d(x_1,y_1),$$
we have, by the triangle inequality,
$$d(z_1,y_1)\ge d(z_1,x_1)-d(x_1,y_1)\ge \frac{1}{2}d(x_1,z_1).$$
Hence, by \eqref{eq:sca-phi}, there exists a constant $c_1\in (0,1)$ such that
$$
\phi_d(z,y)\geq \phi_1(d(z_1,y_1))\geq \phi_1(d(x_1,z_1)/2)
\geq c_1\phi_1(d(x_1,z_1))=c_1\phi_d(x,z).
$$
Combining this with \eqref{eq:p-0}, we get
\begin{equation}\label{eq:p-1}
\phi_d(z,y)\geq   c_1(\phi_d(x,y)\vee \phi_d(x,z)).
\end{equation}
In the same way, one can see that the inequality above is valid also
when $\phi_1(d(x_1,z_1))\leq \phi_2(d(x_2,z_2))$.

Since \eqref{eq:prod-res-cond} and \eqref{eq:prod-res-cond-1} hold by assumption,
Lemma \ref{lem:green-1}
with \eqref{eq:v-inverse} and \eqref{eq:p-1} implies that
for any positive integers $m,n$ with $|m-n|>1$,
and for any  $z\in B_m$ and $y\in B_n$,
\begin{equation}\label{eq:green-comp}
\begin{split}
u_0^{\gamma}(z,y)
&\leq c_2\int_{\phi_d^{\gamma}(z,y)}^{\infty}\frac{1}{V((\phi_1^{\gamma})^{-1}(t))V((\phi_2^{\gamma})^{-1}(t))}\,{\rm d}t\\
&\leq c_2\int_{c_1^{\gamma}(\phi_d^{\gamma}(x,y)\vee \phi_d^{\gamma}(x,z))}^{\infty}
\frac{1}{V((\phi_1^{\gamma})^{-1}(t))V((\phi_2^{\gamma})^{-1}(t))}\,{\rm d}t\\
&\leq c_3 (u_0^{\gamma}(x,y)\wedge u_0^{\gamma}(x,z)).
\end{split}
\end{equation}
On the other hand,
since $B_n$ is compact by Assumption \ref{assum:compact},
there exists a positive Radon measure $\nu_n$ on $M$ such that
$\supp[\nu_n]\subset B_n$, and, for any $z\in B_m$,
$$
P_z^{\gamma}(\sigma_{B_n}<\infty)
=\int_{B_n}u_0^{\gamma}(z,y)\,\nu_n({\rm d}y)
\leq c_3\int_{B_n}u_0^{\gamma}(x,y)\,\nu_n({\rm d}y)
=c_3P_x^{\gamma}(\sigma_{B_n}<\infty).
$$
Note that if $\sigma_{B_m}<\infty$, then  $X_{\sigma_{B_m}}^{\gamma}\in B_m$
because $B_m$ is closed.
Therefore,  by the strong Markov property of the process $X^{\gamma}$,
\begin{equation}\label{eq:hit-1}
\begin{split}
P_x^{\gamma}(\sigma_{B_m}<\infty, \sigma_{B_n}\circ\theta_{\sigma_{B_m}}<\infty)
&=E_x^{\gamma}\left[P_{X_{\sigma_{B_m}}^{\gamma}}^{\gamma}(\sigma_{B_n}<\infty);\sigma_{B_m}<\infty\right]\\
&\leq c_3 P_x^{\gamma}(\sigma_{B_m}<\infty)P_x^{\gamma}(\sigma_{B_n}<\infty).
\end{split}
\end{equation}
By the same argument as before, we also see that for any $y\in B_n$,
$$
P_y^{\gamma}(\sigma_{B_m}<\infty)
\leq c_4 P_x^{\gamma}(\sigma_{B_m}<\infty)$$
and thus
\begin{equation}\label{eq:hit-2}
P_x^{\gamma}(\sigma_{B_n}<\infty, \sigma_{B_m}\circ\theta_{\sigma_{B_n}}<\infty)
\le c_4 P_x^{\gamma}(\sigma_{B_m}<\infty)P_x^{\gamma}(\sigma_{B_n}<\infty).
\end{equation}
Noting that
$$
\{\sigma_{B_m}<\infty, \sigma_{B_n}<\infty\}
=\{\sigma_{B_m}<\infty, \sigma_{B_n}\circ\theta_{\sigma_{B_m}}<\infty\}
\cup \{\sigma_{B_n}<\infty, \sigma_{B_m}\circ\theta_{\sigma_{B_n}}<\infty\},
$$
we  obtain, by \eqref{eq:hit-1} and \eqref{eq:hit-2},
$$
P_x^{\gamma}(\sigma_{B_m}<\infty, \sigma_{B_n}<\infty)
\leq
c_5 P_x^{\gamma}(\sigma_{B_m}<\infty)P_x^{\gamma}(\sigma_{B_n}<\infty).$$
Hence, by combining \eqref{eq:p-equ} with Lemma \ref{lem:b-c}
below,
we get the following equivalence:
$$P_x^{\gamma}(L_B=\infty)=1
\iff P_x^{\gamma}\left(\limsup_{n\rightarrow\infty}{\left\{\sigma_{B_n}<\infty\right\}}\right)>0
\iff \sum_{n=1}^{\infty}P_x^{\gamma}(\sigma_{B_n}<\infty)=\infty.$$
The proof is complete.
\end{proof}

Next we will apply Proposition \ref{thm:wiener} to derive a sufficient condition
for the recurrence of the subset of the diagonal set relative to the process $X^{\gamma}$.
Let
$${\rm diag}(M)=\{y=(y_1,y_2)\in M\times M : y_1=y_2\}$$
be the diagonal set in $M\times M$
with the relative topology.
Then
$$
{\rm diag}(M)\cap {\cal B}(M\times M)={\cal B}({\rm diag}(M)),
$$
where
$$
{\rm diag}(M)\cap {\cal B}(M\times M)
=\left\{{\rm diag}(M)\cap B : B\in {\cal B}(M\times M)\right\}.
$$
We also note that for any $B\in {\rm diag}(M)\cap {\cal B}(M\times M)$,
the set $A_B=\{w\in M : (w,w)\in B\}$ is a Borel subset of $M$
and $B=\{(w,w)\in M\times M : w\in A_B\}$.
On the contrary, for any $A\in {\cal B}(M)$,
$A=\{w\in M : (w,w)\in {\rm diag}(A)\}$, where ${\rm diag}(A)=\left\{(w,w)\in M\times M : w\in A\right\}$.
Hence, we have a one to one correspondence
between ${\rm diag}(M)\cap {\cal B}(M\times M)$ and ${\cal B}(M)$.
Moreover, if $\eta$ is a measure on ${\cal B}(M)$,
then we can associate a unique measure $\nu_{\eta}$ on ${\rm diag}(M)\cap {\cal B}(M\times M)$
such that $\nu_{\eta}({\rm diag}(M)\cap (C_1\times C_2))=\eta(C_1\cap C_2)$ for any $C_1,C_2\in {\cal B}(M)$.

In what follows, fix $x\in M\times M$ and $\lambda\geq 2\vee\phi_1(2d(x_1,x_2))$,
and let $F\subset M$ be an $(s_F,t_F)$-set with $(s_F,t_F)$-measure $\eta$ (see Definition \ref{def:s-set}(ii)).
For simplicity, we assume that $r_0=1$ in  Definition \ref{def:s-set}(ii).
Let $B_n=B_n^{x,\lambda,\phi}$ be as in \eqref{eq:b-p} with $B={\rm diag}(F)$.
Define $\phi(t):=\phi_1(t)\vee \phi_2(t)$. Then $\phi^{-1}(t)=\phi_1^{-1}(t)\wedge \phi_2^{-1}(t)$.
Hence, by \eqref{eq:s-1}, there exists a constant $c_1>0$ such that
for any $n\geq 1$,
\begin{equation}\label{eq:b-upper}
\nu_{\eta}(B_n)\leq c_1 (\phi^{-1}(\lambda^n))^{t_F}.
\end{equation}
We now discuss the lower bound of $\nu_{\eta}(B_n)$.
By definition,
\begin{equation*}
\begin{split}
\nu_{\eta}(B_n)
&=\eta\left(\left\{w\in M: \lambda^n\leq \phi_1(d(x_1,w))\vee\phi_2(d(x_2,w))\leq \lambda^{n+1}\right\}\right)\\
&=\eta\left(\left\{w\in M: \phi_1(d(x_1,w))\vee\phi_2(d(x_2,w))\leq \lambda^{n+1}\right\}\right)\\
&\quad -\eta\left(\left\{w\in M:\phi_1(d(x_1,w))\vee\phi_2(d(x_2,w))<\lambda^n\right\}\right)\\
&={\rm (I)}_n-{\rm (II)}_n.
\end{split}
\end{equation*}
By \eqref{eq:sca-phi-inv}, there exists  $\varepsilon\in (0,1)$ so small that for $i=1,2$ and any $r>0$,
$\phi_i^{-1}(\varepsilon r)/\phi_i^{-1}(r)\leq 1/2$.
Since $\lambda\geq \phi_1(2d(x_1,x_2))$,
we have
$$
{\rm (I)}_n
\ge \eta\left(\left\{w\in M:
(\phi_1\vee \phi_2)(d(x_2,w))\leq \varepsilon\lambda^{n+1}\right\}\right)
=\eta\left(\left\{w\in M:
\phi(d(x_2,w))\leq \varepsilon\lambda^{n+1}\right\}\right),
$$
where we used the fact that if $\phi_1(d(x_2,w))\le \varepsilon \lambda^{n+1}$, then
\begin{align*}
\phi_1(d(x_1,w))\le
&\phi_1(d(x_2,w)+d(x_1,x_2))\le \phi_1(\phi_1^{-1}(\varepsilon \lambda^{n+1})+\phi_1^{-1}(\lambda)/2)\\
\le& \phi_1(\phi_1^{-1}(\lambda^{n+1})/2+\phi_1^{-1}(\lambda)/2)\le \lambda^{n+1}.
\end{align*}
On the other hand, one can see that
there exists a constant $c_2>0$, which is independent of the choices of
$\lambda\geq 2\vee\phi_1(2d(x_1,x_2))$ and $\varepsilon>0$,
such that
$${\rm (II)}_n\le \eta\left(\left\{w\in M:
\phi(d(x_2,w))\leq c_2\lambda^n\right\}\right).$$
Since $\eta$ is an $(s_F,t_F)$-measure,
by
\eqref{eq:s-1} and \eqref{eq:s-2},
we can further take $\lambda\geq 2\vee\phi_1(2d(x_1,x_2))$ so large that
\begin{equation}\label{eq:sup-1}
\limsup_{n\rightarrow\infty}\frac{{\rm (II)}_n}{{\rm (I)}_n}<\frac{1}{2}.
\end{equation}
In particular, there exist $c_3>0$ and $n_0\geq 1$ such that
for all $n\geq n_0$,
\begin{equation}\label{eq:b-lower}
\nu_{\eta}(B_n)={\rm (I)}_n-{\rm (II)}_n
\geq {{\rm (I)}_n}/{2}
\geq c_3(\phi^{-1}(\lambda^n))^{t_F}.
\end{equation}

\begin{prop}\label{prop:rec-set}
Let $M$ satisfy Assumptions  {\rm \ref{assum:volume}}, {\rm \ref{assum:compact}} and {\rm \ref{assum:comp}}.
Suppose that the independent processes $X^1$ and $X^2$ satisfy Assumption  {\bf (H)},  {\rm (WUHK)} and {\rm(HR)}.
Let $F\subset M$ be an $(s_F,t_F)$-set for some positive constants $s_F$ and $t_F$
such that $\gamma(s_F)<1$.
Assume that the following conditions hold for some $\gamma\in (\gamma(s_F),1]${\rm :}
\begin{itemize}
\item
$J^{\gamma}<\infty$, \eqref{eq:prod-res-cond} and \eqref{eq:prod-res-cond-1} hold.
\item
There exists a constant $c_1>0$ such that for any  $T\geq 1$,
\begin{equation}\label{eq:rec-set-1}
\int_1^T\frac{((\phi^{\gamma})^{-1}(t))^{t_F}}{V((\phi_1^{\gamma})^{-1}(t))V((\phi_2^{\gamma})^{-1}(t))}\,{\rm d}t
\leq c_1((\phi^{\gamma})^{-1}(T))^{t_F}
\int_T^{\infty}\frac{1}{V((\phi_1^{\gamma})^{-1}(t))V((\phi_2^{\gamma})^{-1}(t))}\,{\rm d}t.
\end{equation}
\end{itemize}
Then for any  $x\in M\times M$,
$P_x^{\gamma}(L_{{\rm diag}(F)}=\infty)=1$.
\end{prop}

\begin{rem}\rm
We use Proposition \ref{prop:rec-set} with $\gamma=1$ only
for the proof of Theorem \ref{thm:d-lower}.
\end{rem}

\begin{proof}[Proof of  Proposition {\rm \ref{prop:rec-set}}]
Let $F\subset M$ be an $(s_F,t_F)$-set, and $\eta$ the corresponding $(s_F,t_F)$-measure.
We simply write $\nu$ for $\nu_{\eta}$.
Fix $x\in M\times M$, and $\lambda\geq 2\vee \phi_1(2d(x_1,x_2))$ so large that \eqref{eq:b-lower} holds.
Let $B={\rm diag}(F)$ and $B_n=B_n^{x,\lambda,\phi}$  as in \eqref{eq:b-p}.

Take $\varepsilon_0\in (0,1)$ so small  that for $i=1,2$ and for any $r>0$,
$\phi_i^{-1}(\varepsilon_0 r)/ \phi_i^{-1}(r)\leq 1/2$.
Let $z\in M\times M$.
We first assume that $\phi_d(z,x)\geq \lambda^{n+1}/\varepsilon_0$.
Then, for any $y\in B_n$,  it follows from the triangle inequality that
$\phi_d(z,y)\geq \lambda^{n+1}$ by taking $\varepsilon_0$ small enough if necessary.
This and Lemma \ref{lem:green-1} imply that
$$
u_0^{\gamma}(z,y)
\leq c_1\int_{\phi_d^{\gamma}(z,y)}^{\infty}
\frac{1}{V((\phi_1^{\gamma})^{-1}(t))V((\phi_2^{\gamma})^{-1}(t))}\,{\rm d}t
\leq c_1\int_{(\lambda^{n+1})^{\gamma}}^{\infty}
\frac{1}{V((\phi_1^{\gamma})^{-1}(t))V((\phi_2^{\gamma})^{-1}(t))}\,{\rm d}t.$$
Then, by \eqref{eq:b-upper},
\begin{equation}\label{eq:b-int-1}
\begin{split}
\int_{B_n}u_0^{\gamma}(z,y)\,\nu({\rm d}y)
&\leq c_1\nu(B_n)\int_{(\lambda^{n+1})^{\gamma}}^{\infty}
\frac{1}{V((\phi_1^{\gamma})^{-1}(t))V((\phi_2^{\gamma})^{-1}(t))}\,{\rm d}t\\
&\leq c_2(\phi^{-1}(\lambda^n))^{t_F}\int_{(\lambda^{n+1})^{\gamma}}^{\infty}
\frac{1}{V((\phi_1^{\gamma})^{-1}(t))V((\phi_2^{\gamma})^{-1}(t))}\,{\rm d}t.
\end{split}
\end{equation}

We next assume that $\phi_d(z,x)<\lambda^{n+1}/\varepsilon_0$.
Then, there exists a constant $c_3>0$ such that
$\phi_d(z,y)\leq c_3\lambda^{n+1}$ for any $y\in B_n$.
Hence, by Lemma \ref{lem:green-1} and the Fubini theorem,
\begin{equation}\label{eq:b-int-2}
\begin{split}
&\int_{B_n}u_0^{\gamma}(z,y)\,\nu({\rm d}y)
\leq \int_{\phi_d(z,y)\leq c_3\lambda^{n+1}}u_0^{\gamma}(z,y)\,\nu({\rm d}y)\\
&\leq c_4\int_{\phi_d(z,y)\leq c_3\lambda^{n+1}}
\left(\int_{\phi_d^{\gamma}(z,y)}^{\infty}
\frac{1}{V((\phi_1^{\gamma})^{-1}(t))V((\phi_2^{\gamma})^{-1}(t))}\,{\rm d}t\right)\,\nu({\rm d}y)\\
&=c_4\Biggl\{\int_0^{(c_3\lambda^{n+1})^{\gamma}}\frac{1}{V((\phi_1^{\gamma})^{-1}(t))V((\phi_2^{\gamma})^{-1}(t))}
\nu(\{y\in M\times M:\phi_d^{\gamma}(z,y)\leq t\})\,{\rm d}t\\
&\quad +\left(\int_{(c_3\lambda^{n+1})^{\gamma}}^{\infty}
\frac{1}{V((\phi_1^{\gamma})^{-1}(t))V((\phi_2^{\gamma})^{-1}(t))}\,{\rm d}t\right)
\nu(\{y\in M\times M:\phi_d(z,y)\leq c_3 \lambda^{n+1}\})
\Biggr\}\\
&=:c_4({\rm (I)}+{\rm (II)}).
\end{split}
\end{equation}
Since $\eta$ is an $(s_F,t_F)$-measure,
there exists a constant $c_5>0$ such that for any $r\geq 0$,
$$\nu(\{y\in M\times M: \phi_d(z,y)\leq r\})
\leq c_5\left((\phi^{-1}(r))^{s_F}{\bf 1}_{\{0<r<1\}}+(\phi^{-1}(r))^{t_F}{\bf 1}_{\{r\geq 1\}}\right).$$
Then, by this inequality and \eqref{eq:rec-set-1}
with \eqref{eq:sca-phi} and \eqref{eq:v-inverse}, we obtain
\begin{equation*}
\begin{split}
{\rm (I)}
&\leq
c_5\left(\int_0^1\frac{((\phi^{\gamma})^{-1}(t))^{s_F}}
{V((\phi_1^{\gamma})^{-1}(t))V((\phi_2^{\gamma})^{-1}(t))}\,{\rm d}t
+\int_1^{(c_3\lambda^{n+1})^{\gamma}}
\frac{((\phi^{\gamma})^{-1}(t))^{t_F}}{V((\phi_1^{\gamma})^{-1}(t))V((\phi_2^{\gamma})^{-1}(t))}\,{\rm d}t\right)\\
&\leq c_6 (\phi^{-1}(c_3\lambda^{n+1}))^{t_F} \int_{(c_3\lambda^{n+1})^{\gamma}}^{\infty}
\frac{1}{V((\phi_1^{\gamma})^{-1}(t))V((\phi_2^{\gamma})^{-1}(t))}\,{\rm d}t\\
&\leq c_7 (\phi^{-1}(\lambda^{n+1}))^{t_F} \int_{(\lambda^{n+1})^{\gamma}}^{\infty}
\frac{1}{V((\phi_1^{\gamma})^{-1}(t))V((\phi_2^{\gamma})^{-1}(t))}\,{\rm d}t
\end{split}
\end{equation*}
and
$$
{\rm (II)}\leq
c_8(\phi^{-1}(\lambda^{n+1}))^{t_F}
\int_{(\lambda^{n+1})^{\gamma}}^{\infty}
\frac{1}{V((\phi_1^{\gamma})^{-1}(t))V((\phi_2^{\gamma})^{-1}(t))}\,{\rm d}t.
$$
According to \eqref{eq:b-int-1} and \eqref{eq:b-int-2},
there exists a constant  $c_*>0$ so that for any $n\geq 1$,
\begin{equation}\label{eq:gamma}
\int_{B_n}u_0^{\gamma}(z,y)\,\nu({\rm d}y) \leq \Gamma_n(\lambda),\quad z\in M\times M,
\end{equation}
where
$$\Gamma_n(\lambda)
=c_*(\phi^{-1}(\lambda^{n+1}))^{t_F}
\int_{(\lambda^{n+1})^{\gamma}}^{\infty}
\frac{1}{V((\phi_1^{\gamma})^{-1}(t))V((\phi_2^{\gamma})^{-1}(t))}\,{\rm d}t.$$

For any $n\ge1$, define
$$\nu_n=\frac{1}{\Gamma_n(\lambda)}\nu|_{B_n}.$$
Then by \eqref{eq:gamma},
$$
\int_{M\times M}u_0^{\gamma}(z,y)\,\nu_n({\rm d}y)\leq 1, \quad  z\in M\times M
$$
and
$$
\int_{M\times M}\int_{M\times M}u_0^{\gamma}(z,y)\,\nu_n({\rm d}z)\,\nu_n({\rm d}y)\leq \nu_n(B_n)<\infty.
$$
Hence, by the $0$-order version of \cite[Exercise 4.2.2]{FOT11},
$\nu_n$ is of finite $0$-order energy integral relative to the process $X^{\gamma}$,
and the function $g(z):=\int_{B_n}u_0^{\gamma}(z,y)\,\nu_n({\rm d}y)$ is a quasi-continuous and
excessive $\mu$-version
of the $0$-potential of $\nu_n$.
Since $B_n$ is compact, it follows by \eqref{eq:cap-sup} and \eqref{eq:b-lower} that
for all sufficiently large $n\geq 1$,
$$
{\Capa}_{(0)}^{\gamma}(B_n)
\geq \nu_n(B_n)=\frac{\nu(B_n)}{\Gamma_n(\lambda)}
\geq \frac{c_9(\phi^{-1}(\lambda^n))^{t_F}}{\Gamma_n(\lambda)},
$$
where ${\Capa}_{(0)}^{\gamma}$ is the $0$-order capacity relative to $(\E^{\gamma},\F^{\gamma})$.
 Furthermore, if $\nu_n^{\gamma}$ denotes the equilibrium measure of $B_n$,
then, by Lemma \ref{lem:green-1},
\begin{equation*}
\begin{split}
P_x^{\gamma}(\sigma_{B_n}<\infty)
&=\int_{B_n}u_0^{\gamma}(x,y)\,\nu_n^{\gamma}({\rm d}y)
\geq c_{10}\nu_n^{\gamma}(B_n)
\int_{(\lambda^{n+1})^{\gamma}}^{\infty}\frac{1}{V((\phi_1^{\gamma})^{-1}(t))V((\phi_2^{\gamma})^{-1}(t))}\,{\rm d}t\\
&=c_{10}{\Capa}_{(0)}^{\gamma}(B_n)
\int_{(\lambda^{n+1})^{\gamma}}^{\infty}\frac{1}{V((\phi_1^{\gamma})^{-1}(t))V((\phi_2^{\gamma})^{-1}(t))}\,{\rm d}t\\
& \geq  \frac{c_{11}(\phi^{-1}(\lambda^n))^{t_F}}{\Gamma_n(\lambda)}
\int_{(\lambda^{n+1})^{\gamma}}^{\infty}\frac{1}{V((\phi_1^{\gamma})^{-1}(t))V((\phi_2^{\gamma})^{-1}(t))}\,{\rm d}t
\geq \frac{c_{12}}{c_*}.
\end{split}
\end{equation*}
Therefore, Proposition \ref{thm:wiener} yields the desired assertion.
\end{proof}

\subsection{Wiener test for regularity}
In this subsection, we show the Wiener test for the regularity of points
relative to the process $X^{\gamma}$.
Let $X^1$ and $X^2$ be two independent
 $\mu$-symmetric Hunt processes on $M$ satisfying
 Assumption {\bf(H)},
{\rm (WUHK)} and {\rm (HR)},
and let $X$ be the direct product of $X^1$ and $X^2$ on $M\times M$.
For $\gamma\in (0,1]$, $X^{\gamma}$ denotes the $\gamma$-subordinate process of $X$.

For any $x\in M\times M$, $\lambda\in (0,1)$ and $B\in {\cal B}(M\times M)$, define
\begin{equation}\label{set-add}
B_n^{x,\lambda,\phi}
=\left\{y\in B : \lambda^{n+1}\leq \phi_d(x,y)\leq \lambda^n\right\},\quad n\ge 1.
\end{equation}
It follows from the proof of Lemma \ref{lem:0-1-reg} that
the following equivalence holds:
if the process $X^{\gamma}$ is transient and $\{x\}$ is polar relative to  $X^{\gamma}$,
and the Blumenthal zero-one law holds for the process $X^\gamma$,
 then
\begin{equation}\label{lem:0-1-reg--}
P_x^{\gamma}(\sigma_B=0)=1\iff P_x^{\gamma}(\sigma_B=0)>0
\iff P_x^{\gamma}\left(\limsup_{n\to\infty}\left\{\sigma_{B_n^{x,\lambda,\phi}}<\infty\right\}\right)=1.
\end{equation}
Using this equivalence, we can prove

\begin{prop}\label{thm:wiener-1}
Let $M$ satisfy Assumptions {\rm \ref{assum:volume}}, {\rm \ref{assum:compact}} and {\rm \ref{assum:comp}}.
Suppose that the independent processes $X^1$ and $X^2$ satisfy Assumption {\bf(H)}, {\rm (WUHK)} and {\rm(HR)}.
Take $\lambda\in (0,1)$ so that  $\phi_i^{-1}(t/\lambda)\geq 2\phi_i^{-1}(t)$ for $i=1,2$ and any $t>0$.
If for some $\gamma\in (0,1]$,
the process $X^{\gamma}$ is transient and $\{x\}$ is polar relative to $X^{\gamma}$,
then, for any $B\in {\cal B}(M\times M)$,
$$P_x^{\gamma}(\sigma_B=0)=1\iff\sum_{n=1}^{\infty}P_x^{\gamma}(\sigma_{B_n^{x,\lambda,\phi}}<\infty)=\infty.$$
\end{prop}
\begin{proof}
It follows from Proposition \ref{prop:tail} that,
under the assumptions of this
proposition, the Blumenthal zero-one law holds for the process $X^\gamma$.
We then take an approach similar to the proof of \cite[p.67, Theorem 3.3]{PS78}.
To simplify the notation, we write $B_n$ for $B_n^{x,\lambda,\phi}$.
Let $m$ and $n$ be positive integers such that $|m-n|>1$.
Without loss of generality, we assume that $n>m+1$.
For any  $z\in B_m$ and $y\in B_n$, if $\phi_1(d(x_1,z_1))\geq \phi_2(d(x_2,z_2))$, then
$$\phi_1(d(x_1,z_1))\geq \lambda^{m+1}\geq \lambda^{n-1}
\geq \phi_d(x,y)/\lambda.$$
Hence, by the triangle inequality
and $\phi_1^{-1}(\phi_d(x,y)/\lambda)\geq 2\phi_1^{-1}(\phi_d(x,y))$,
$$
d(y_1,z_1)
\geq d(x_1,z_1)-d(x_1,y_1)
\geq \phi_1^{-1}\left(\phi_d(x,y)/\lambda\right)
-\phi_1^{-1}\left(\phi_d(x,y)\right)\geq \phi_1^{-1}(\phi_d(x,y)),
$$
which yields
$$\phi_d(y,z)\geq \phi_d(x,y).$$
This argument is valid also
when $\phi_1(d(x_1,z_1))\leq \phi_2(d(x_2,z_2))$ holds.
Hence, by following the proof of Proposition \ref{thm:wiener},
there exists a constant $c_1>0$ such that for any $m,n\in {\mathbb N}$ with $|m-n|>1$,
$$P_x^{\gamma}(\sigma_{B_m}<\infty, \sigma_{B_n}<\infty)
\leq c_1P_x^{\gamma}(\sigma_{B_m}<\infty)P_x^{\gamma}(\sigma_{B_n}<\infty).$$
By combining \eqref{lem:0-1-reg--} with Lemma \ref{lem:b-c} below,
the proof is complete.
\end{proof}

Let $F\subset M$ be a locally $s_F$-set, and $\eta$ the corresponding $s_F$-measure.
Fix $x\in {\rm diag}(F)$ and  $\lambda\in (0,1)$.
Let $B={\rm diag}(F)$ and $B_n^{x,\lambda,\phi}$ as in \eqref{set-add}.
Then, there exists a constant $c_1>0$ such that for all $n\ge1$,
\begin{equation}\label{eq:r-b-upper}
\nu_{\eta}(B_n)\leq c_1(\phi^{-1}(\lambda^n))^{s_F}.
\end{equation}
Furthermore,
we can also follow the argument of \eqref{eq:sup-1} to show that
there exist $c_2>0$ and $n_0\geq 1$ such that for all $n\ge n_0$,
\begin{equation}\label{eq:r-b-lower}
\nu_{\eta}(B_n)\geq c_2(\phi^{-1}(\lambda^n))^{s_F}.
\end{equation}

For $B\in {\cal B}(M\times M)$, let $B_{\gamma}^r$ be the totality of regular points for $B$
relative to the process $X^{\gamma}$, i.e.,
$$B_{\gamma}^r=\left\{y\in M\times M : P_y^{\gamma}(\sigma_B=0)=1\right\}.$$
If $B$ is closed, then $B_{\gamma}^r\subset B$ by the right continuity of sample paths of $X^{\gamma}$.

\begin{prop}\label{prop:reg-set}
Let $M$ satisfy Assumptions {\rm \ref{assum:volume}}, {\rm \ref{assum:compact}} and {\rm \ref{assum:comp}}.
Suppose that  the independent processes $X^1$ and $X^2$ satisfy Assumption {\bf(H)}, {\rm (WUHK)} and {\rm(HR)}.
Let $F\subset M$ be a locally $s_F$-set for some $s_F>0$ with $\gamma(s_F)<1$.
Assume that the following conditions hold for some
$\gamma\in (\gamma(s_F),1]${\rm :}
\begin{itemize}
\item
$J^{\gamma}<\infty$ and
\begin{equation}\label{eq:div-polar}
\int_0^1\frac{1}{V((\phi_1^{\gamma})^{-1}(t))V((\phi_2^{\gamma})^{-1}(t))}\,\d t=\infty.
\end{equation}
\item \eqref{eq:prod-res-cond} holds.

\item There exists a constant $c_1>0$ such that
for any $T\in (0,1/2)$,
\begin{equation}\label{eq:reg-set-1}
\int_0^T\frac{((\phi^{\gamma})^{-1}(t))^{s_F}}{V((\phi_1^{\gamma})^{-1}(t))V((\phi_2^{\gamma})^{-1}(t))}\,{\rm d}t
\leq c_1((\phi^{\gamma})^{-1}(T))^{s_F}
\int_T^1\frac{1}{V((\phi_1^{\gamma})^{-1}(t))V((\phi_2^{\gamma})^{-1}(t))}\,{\rm d}t.
\end{equation}
\end{itemize}
Then, for any $x\in {\rm diag}(F)$,
$P_x^{\gamma}(\sigma_{{\rm diag}(F)}=0)=1$, that is, $({\rm diag}(F))_{\gamma}^r={\rm diag}(F)$.
\end{prop}

\begin{proof}
We prove this proposition by applying Proposition \ref{thm:wiener-1} to the process $X^{\gamma}$.
To do so, we first verify that $X^{\gamma}$ is transient and any one point set is polar relative to $X^{\gamma}$.
Since $J^{\gamma}<\infty$ by assumption, $X^{\gamma}$ is transient by Lemma \ref{lem:trans} (2).
By Lemma \ref{lem:d-res-lower} and \eqref{eq:div-polar} with Remark \ref{rem:heat-kernel}
(iii),
there exists a constant
$c_0>0$ such that  for any $x\in M$,
$$u_1^{\gamma}(x,x)\ge c_0\int_0^1\frac{1}{V((\phi_1^{\gamma})^{-1}(t))V((\phi_2^{\gamma})^{-1}(t))}\,\d t=\infty.$$
Hence, by Lemma \ref{lem:cap-heat}, any one point set is polar relative to $X^{\gamma}$.

Let $F$ be a locally $s_F$-set, and $\eta$ the corresponding $s_F$-measure.
We simply write $\nu$ for $\nu_{\eta}$.
We take  $\varepsilon_0\in (0,1)$ so small that
$\phi^{-1}(\varepsilon_0 r)/ \phi^{-1}(r)\leq 1/2$ for any $r>0$.
For fixed $x\in {\rm diag}(F)$ and $\lambda\in (0,1)$,
let $B={\rm diag}(F)$ and $B_n=B_n^{x,\lambda,\phi}$.

Let $z\in M\times M$.
We first assume that  $\phi_d(z,x)\geq \lambda^n/\varepsilon_0$.
Then, by the triangle inequality,
we have for any $y\in B_n$,
$\phi_d(z,y)\geq \lambda^n.$
Combining this with Lemma \ref{lem:green-1} and \eqref{eq:r-b-upper}, we get
\begin{equation}\label{eq:r-b-int-1}
\begin{split}
\int_{B_n}u_0^{\gamma}(z,y)\,\nu({\rm d}y)
&\leq c_1\int_{B_n}\left(\int_{\phi_d^{\gamma}(z,y)}^{\infty}
\frac{1}{V((\phi_1^{\gamma})^{-1}(t))V((\phi_2^{\gamma})^{-1}(t))}\,\d t\right)\,\nu(\d y)\\
&\leq  c_1\nu(B_n)\int_{(\lambda^n)^{\gamma}}^{\infty}
\frac{1}{V((\phi_1^{\gamma})^{-1}(t))V((\phi_2^{\gamma})^{-1}(t))}\,\d t\\
&\leq c_2(\phi^{-1}(\lambda^n))^{s_F}
\int_{(\lambda^n)^{\gamma}}^{\infty}
\frac{1}{V((\phi_1^{\gamma})^{-1}(t))V((\phi_2^{\gamma})^{-1}(t))}\,\d t.
\end{split}
\end{equation}

We next assume that $\phi_d(z,x)<\lambda^n/\varepsilon_0$.
Since $\phi_d(z,y)\leq c_3\lambda^n$ for any $y\in B_n$ and \eqref{eq:reg-set-1} holds,
we can follow
the calculation in \eqref{eq:b-int-2} and its subsequent argument
to prove that
\begin{equation}\label{eq:r-b-int-2}
\begin{split}
\int_{B_n}u_0^{\gamma}(z,y)\,\nu({\rm d}y)
&\leq \int_{\phi_d(z,y)\leq c_3\lambda^n}u_0^{\gamma}(z,y)\,\nu({\rm d}y)\\
&\leq c_4 \int_{\phi_d(z,y)\leq c_3\lambda^n}
\left(\int_{\phi_d^{\gamma}(z,y)}^{\infty}
\frac{1}{V((\phi_1^{\gamma})^{-1}(t))V((\phi_2^{\gamma})^{-1}(t))}\,\d t\right)\,\nu(\d y)\\
&\leq c_5(\phi^{-1}(\lambda^n))^{s_F}
\int_{(\lambda^n)^{\gamma}}^{\infty}
\frac{1}{V((\phi_1^{\gamma})^{-1}(t))V((\phi_2^{\gamma})^{-1}(t))}\,\d t.
\end{split}
\end{equation}

According to \eqref{eq:r-b-int-1} and \eqref{eq:r-b-int-2},
we have for any $n\geq 1$,
$$
\int_{B_n}u_0^{\gamma}(z,y)\,\nu({\rm d}y)
\leq \Gamma_n(\lambda), \quad  z\in M\times M,
$$
where
$c_*=c_2\vee c_5$ and
$$\Gamma_n(\lambda)
=c_*(\phi^{-1}(\lambda^n))^{s_F}
\int_{(\lambda^n)^{\gamma}}^{\infty}
\frac{1}{V((\phi_1^{\gamma})^{-1}(t))V((\phi_2^{\gamma})^{-1}(t))}\,\d t.$$
For any $n\ge 1$, define
$$
\nu_n=\frac{1}{\Gamma_n(\lambda)}\nu|_{B_n},
$$
so that
$$
\int_{M\times M}u_0^{\gamma}(z,y)\,\nu_n({\rm d}y)\leq 1, \quad  z\in M\times M,
$$
and
$$
\int_{M\times M}\int_{M\times M}u_0^{\gamma}(z,y)\,\nu_n({\rm d}z)\,\nu_n({\rm d}y)\leq \nu_n(B_n)<\infty.
$$
Hence, by following the proof of Proposition \ref{prop:rec-set},
there exists a constant $c_6\in (0,1]$ such that
$P_x^{\gamma}(\sigma_{B_n}<\infty)\geq c_6$ for any $n\geq 1$.
Then, by Proposition \ref{thm:wiener-1}, the proof is complete.
\end{proof}

\subsection{Generalized Borel-Cantelli lemma}
We state the following generalized Borel-Cantelli lemma for the readers' convenience.
\begin{lem}{\rm (\cite[p.65, Proposition 3.1]{PS78})}\label{lem:b-c}
Let $(\Omega,{\cal F},P)$ be a probability space,
and $\{A_n\}_{n=1}^{\infty}$ a sequence of events.
Assume that there exists a constant $c_1>0$ such that
for any $m,n\geq 1$ with $|m-n|>1$,
$$P(A_m\cap A_n)\leq c_1P(A_m)P(A_n).$$
Then,
$$P\left(\limsup_{n\rightarrow\infty}A_n\right)>0$$
if and only if
$$\sum_{n=1}^{\infty}P(A_n)=\infty.$$
\end{lem}

\ \

\noindent \textbf{Acknowledgements.} 
The research of Jian Wang is supported by the National Key R\&D Program of China (2022YFA1000033) and  NSF of China (Nos.\ 11831014, 12071076 and 12225104).

\end{document}